\documentclass[a4paper]{amsart}
\usepackage{amsmath} 
\usepackage{amssymb}
\usepackage{latexsym}
\usepackage[mathscr]{eucal}
\usepackage{mathrsfs}
\usepackage[varg]{txfonts}

\usepackage{ascmac}
\usepackage{framed}
\usepackage{ulem}
\usepackage{cancel}
\usepackage{arydshln}
\usepackage{bm}

\theoremstyle{plain}
\newtheorem{thm}{Theorem}[section]
\newtheorem{lem}[thm]{Lemma}
\newtheorem{prop}[thm]{Proposition}

\makeatletter
\def\thefootnote{\ifnum\c@footnote>\z@\leavevmode\lower.5ex%
      \hbox{$^{\@arabic\c@footnote)}$}\fi}
\makeatother

\theoremstyle{definition}

\theoremstyle{remark}

\def\dfrac#1#2{\displaystyle \frac{#1}{#2}}

\def\det{\mbox{\rm {det}}}

\def\diag{\mbox{\rm {diag}}}

\def\dlim{\displaystyle \lim}

\def\ad{\mbox{\rm {ad}}}

\def\Iso{\mbox{\rm {Iso}}}

\def\Ker{\mbox{\rm {Ker}}}

\def\tr{\mbox{\rm {tr}}}

\def\Re{\mbox{\rm {Re}}}
\def\Im{\mbox{\rm {Im}}}

\def\ov{\overline}
\def\wti{\widetilde}
\def\ti{\tilde}

\def\dsum{\displaystyle \sum}

\def\dfrac#1#2{\displaystyle \frac{#1}{#2}}

\def\C{\mbox{\boldmath $C$}}

\def\R{\mbox{\boldmath $R$}} 
\def\s{\mbox{\boldmath $S$}}

\def\Z{\mbox{\boldmath $Z$}}

\def\0{\mbox{\boldmath {0}}}    
\def\1{\mbox{\boldmath {1}}}      
\def\2{\mbox{\boldmath {2}}}      
\def\3{\mbox{\boldmath {3}}}      
\def\4{\mbox{\boldmath {4}}}      
\def\5{\mbox{\boldmath {5}}}      
\def\6{\mbox{\boldmath {6}}}      
\def\7{\mbox{\boldmath {7}}}      
\def\8{\mbox{\boldmath {8}}}      
\def\9{\mbox{\boldmath {9}}}

\def\s{\mbox{\boldmath $s$}}

\usepackage{epic,eepic}

\begin{document}

\title[Realization of globally exceptional Riemannian $4$-symmetric space
]
{Realization of globally exceptional Riemannian $4$-symmetric space
 $E_8/(E_8)^{\sigma'_{\!\!4}}$
}

 
\author[\sc Toshikazu Miyashita]{\sc Toshikazu Miyashita}



\begin{abstract}
The compact simply connected Riemannian 4-symmetric spaces were classified by J.A. Jim{\'{e}}nez. As homogeneous manifolds, these spaces are of the $G/H$, where $G$ is a connected compact simple Lie group with an automorphism $\tilde{\gamma}$ of oder 4 and  $H$ is a fixed points subgroup $G^\gamma$ of $G$. In the present article, for the exceptional compact Lie group $G=E_8$, we give the explicit form of automorphism  of order 4 induced by the $\R$-linear transformation $\sigma'_{\!\!4}$ and determine the structure of the group $(E_8)^{\sigma'_{\!\!4}}$. Thereby, we realize the globally exceptional Riemannian $4$-symmetric space $E_8/(E_8)^{\sigma'_{\!\!4}}$.

\end{abstract}

\maketitle

\section {Introduction}
It is well-known that the notion of the Riemannian $k$-symmetric space is a generalization of the Riemannian symmetric space. Its definition is as follows:

Let $G$ be a Lie group and $H$ a compact subgroup of $G$. A homogeneous space
$G/H$ with $G$-invariant Riemannian metric $g$ is called a Riemannian $k$-symmetric space if there exists an automorphism $\tilde{\gamma}$ of order $k$ on $G$ such that $({G^\gamma})_{0} \subset H \subset G^\gamma$, where $G^\gamma$ and $({G^\gamma})_{0}$ is the fixed points subgroup of $G$ and its identity component, respectively, and such that the transformation of $G/H$ induced by $\gamma$ is an isometry (\cite {kuri}).

Until now, the Riemannian $3$-symmetric spaces were classified by A. Gray \cite{gray}, and moreover the compact Riemannian $4$-symmetric spaces were classified by J.A. Jim\'{e}nez \cite{Jim} as mentioned in abstract.

Now, for the exceptional compact Lie groups $G=G_2, F_4, E_6, E_7$, we have already realized the Riemannian $4$-symmetric spaces $G/H$ by giving automorphisms $\tilde\gamma$ of order 4 explicitly and by determining the structure of the group $H=G^\gamma$ corresponding to the Lie algebras $\mathfrak{g}^\gamma$ of Tables III, IV and V in \cite{Jim}. In particular, for $G=E_8$, there exist seven types of Riemannian $4$-symmetric spaces. From now on, we call those spaces "Globally exceptional Riemannian $4$-symmetric spaces". The main purpose of this article is to give the automorphism $\tilde{\sigma'}_{\!\!4}$ of order 4 on $E_8$ explicitly and to determine the structure of the fixed points subgroup $(E_8)^{\sigma'_{\!\!4}}$ of $E_8$, where the structure of the group $(E_8)^{\sigma'_{\!\!4}}$ is as follows:
$$
 (E_8)^{\sigma'_{\!\!4}} \cong (S\!pin(6) \times S\!pin(10))/\Z_4.
$$ 
Here, the spinor groups $S\!pin(6)$ and $S\!pin(10)$ above are respectively realized as the subgroup $(F_4)_{E_1, E_2,E_3, F_1(e_k), k=0,1}$ 
and the subgroup $(E_8)^{\sigma'_{\!\!4}, \mathfrak{so}(6)}$ of  $(E_8)^{\sigma'_{\!\!4}}$, where the definitions or the details of $(F_4)_{E_1, E_2,E_3, F_1(e_k), k=0,1}$ and $(E_8)^{\sigma'_{\!\!4}, \mathfrak{so}(6)}$ are shown later. This amounts to the globally realization of one of seven types with an automorphism of order 4 on $E_8$.
 
According to \cite{kuri}, 
it is known that the involutive automorphisms of $k$-symmetric spaces $G/H$ preserving $H$ are important, because such involutive automorphisms play an important role in the classification of symmetric submanifolds on symmetric spaces. 

\noindent On the globally exceptional Riemannian $4$-symmetric space $E_8/(E_8)^{\sigma'_{\!\!4}}$, the involutive automorphism of this $4$-symmetric space preserving $(E_8)^{\sigma'_{\!\!4}}$ is given as the involutive automorphism $\ti{\sigma}$ induced by the $C$-linear transformation $\sigma$ of the complex vector space ${\mathfrak{e}_8}^C$, where $\sigma$ is defined later. Moreover, from $
(E_8)^\sigma \cong S\!s(16)(=S\!pin(16)/\Z_2)$ (\cite [Proposition 4.20.2]{miya}) and $(\sigma'_{\!\!4})^2=\sigma$, we can confirm that the space $E_8/(E_8)^{\sigma'_{\!\!4}}$ has the structure of fiber bundle as follows:
$$
   (E_8)^\sigma/(E_8)^{\sigma'_{\!\!4}} \,\to\, E_8/(E_8)^{\sigma'_{\!\!4}} \, \to \, E_8/(E_8)^\sigma,
$$
that is, 
$$
  S\!s(16)/(S\!pin(6) \times S\!pin(10))/\Z_4 \,\to \, E_8/(S\!pin(6) \times S\!pin(10))/\Z_4 \,\to \, E_8/S\!s(16).
$$
 
Again we would like to state about the group $(E_8)^{\sigma'_{\!\!4}, \mathfrak{so}(6)}$. The essential part to prove the isomorphism as a group is to show the connectedness of the group $(E_8)^{\sigma'_{\!\!4}, \mathfrak{so}(6)}$. In order to obtain this end, we need to treat the complex case as follows:
\begin{eqnarray*}
({F_4}^C)_{E_1, E_2,E_3, F_1(e_k), k=0,1} \!\!\! & \cong &\!\!\! S\!pin(6,C),
\\[1mm]
({E_7}^C)^{\sigma'_{\!\!4}}\!\!\! & \cong &\!\!\! (S\!L(2, C) \times S\!pin(6,C) \times S\!pin(6, C))/\Z_4, 
\\[1mm]
({E_7}^C)^{\sigma'_{\!\!4}, \mathfrak{so}(6,C)}\!\!\! & \cong &\!\!\! S\!L(2, C) \times S\!pin(6,C),
\\[1mm]
({E_8}^C)^{\sigma'_{\!\!4}, \mathfrak{so}(6,C)}\!\!\! & \cong &\!\!\! S\!pin(10, C), 
\\[1mm]
({E_8}^C)^{\sigma'_{\!\!4}} \!\!\! & \cong &\!\!\! (S\!pin(6,C) \times S\!pin(10,C))/\Z_4,
\end{eqnarray*}
and the connectedness of the group $({E_8}^C)^{\sigma'_{\!\!4}, \mathfrak{so}(6,C)}$,  the definitions or the details of the group on the left-hand side in each row above are also shown later.

In this article, in order to study the subgroups of $E_8$ as mentioned above, since we need to have some knowledge of the their complexification $G^C$ of $G=F_4, E_6, E_7$ or $E_8$, refer to \cite
{Imai}, \cite{realization G_2}, \cite{realization E_7} or \cite{realization E_8}, and we use the same notations as in \cite{realization G_2}, \cite{realization E_7}, \cite{realization E_8} or \cite{Yokotaichiro}.

Finally, the author would like to say that the feature of this article is to give elementary proofs by the homomorphism theorem except several proofs. 
\vspace{2mm}

\section{Preliminaries}

Let $\mathfrak{J}(3,\mathfrak{C}^C )$ and $\mathfrak{J}(3, \mathfrak{C})$
 be the exceptional $C$- and $\R$-Jordan algebras, respectively. In $\mathfrak{J}(3,\mathfrak{C}^C )$, the Jordan multiplication $X \circ Y$, the 
inner product $(X,Y)$ and a cross multiplication $X \times Y$, called the Freudenthal multiplication, are defined by
$$
\begin{array}{c}
      X \circ Y = \dfrac{1}{2}(XY + YX), \quad (X,Y) = \tr(X \circ Y),
\vspace{1mm}\\
       X \times Y = \dfrac{1}{2}(2X \circ Y-\tr(X)Y - \tr(Y)X + (\tr(X)\tr(Y) 
- (X, Y))E), 
\end{array}$$
respectively, where $E$ is the $3 \times 3$ unit matrix. Moreover, we define the trilinear form $(X, Y, Z)$, the determinant $\det \,X$ by
$$
(X, Y, Z)=(X, Y \times Z),\quad \det \,X=\dfrac{1}{3}(X, X, X),
$$
respectively, and briefly denote $\mathfrak{J}(3, \mathfrak{C}^C)$ and  $\mathfrak{J}(3, \mathfrak{C})$ 
by $\mathfrak{J}^C$ and $\mathfrak{J}$, respectively. In $\mathfrak{J}$, we can also define the relational formulas above.

The connected complex Lie group ${F_4}^C$ and the connected compact Lie group $F_4$ are defined by
\begin{eqnarray*}
      {F_4}^C \!\!\!&=&\!\!\! \{\alpha \in \Iso_C(\mathfrak{J}^C) \, | \, \alpha(X \circ Y) = \alpha X \circ \alpha Y\}
\\[1mm]
                   \!\!\!&=&\!\!\! \{\alpha \in \Iso_C(\mathfrak{J}^C) \, | \, \det\, \alpha X=\det\, X, (\alpha X, \alpha Y)=(X, Y)\}
\\[1mm]
                    \!\!\!&=&\!\!\! \{\alpha \in \Iso_C(\mathfrak{J}^C) \, | \, \det\, \alpha X=\det\, X, \alpha E=E\},
\\[1mm]
       F_4 \!\!\!&=&\!\!\! \{\alpha \in \Iso_{\bm{R}}(\mathfrak{J}) \, |\, \alpha(X \circ Y) = \alpha X \circ \alpha Y\},
\end{eqnarray*}
respectively. Let $\tau$ be the complex conjugation in $\mathfrak{J}^C$. 
Then we have $F_4 =({F_4}^C)^\tau$ (see \cite[Section 2.4]{realization G_2} in detail).
Moreover, the Lie algebra ${\mathfrak{f}_4}^C$ 
of the group ${F_4}^C$ 
is given by 
\begin{eqnarray*}
{\mathfrak{f}_4}^C\!\!\!&=&\!\!\!\{D +\ti{A}_1(a_1) +\ti{A}_2(a_2) +\ti{A}_3(a_3)\,|\, D \in \mathfrak{so}(8,C), a_k \in \mathfrak{C}^C, k=1,2,3\},
\end{eqnarray*}
where $\ti{A}_k(a_k)$ is the $C$-linear mapping of $\mathfrak{J}^C$ (see \cite[Theorem 2.2.2]{realization G_2} in detail).
\vspace{1mm}

We define an $\R$-linear transformation ${\sigma'}_{\!\!4}$  of $\mathfrak{J}$ by
$$
   {\sigma'}_{\!\!4} X={\sigma'}_{\!\!4} \begin{pmatrix} \xi_1 & x_3  & \ov{x}_2 \\
                                                    \ov{x}_3 & \xi_2 &  x_1 \\
                                                    x_2 & \ov{x}_1 &\xi_3     
                            \end{pmatrix} 
                         = \begin{pmatrix} \xi_1 & -x_3 e_1 & \ov{e_1 x_2} \\
                                                   -\ov{x_3 e_1} & \xi_2 & -e_1 x_1 e_1 \\
                                             e_1 x_2 & -\ov{e_1 x_1 e_1} & \xi_3                      
                             \end{pmatrix}, \,\,X \in \mathfrak{J},
$$
where the element $e_1$ is one of the basis of $\mathfrak{C}=\{e_0=1, e_1, e_2, \ldots, e_7 \}_{\bm{R}}$. 
Hereafter, a symbol $e_k$ means one of the basis of $\mathfrak{C}$ or $\mathfrak{C}^C$.
Then we have that ${\sigma'}_{\!\!4} \in S\!pin(8) \subset F_4 \subset {F_4}^C, ({\sigma'}_{\!\!4})^4=1, ({\sigma'}_{\!\!4})^2=\sigma$, where an $\R$-linear transformation $\sigma : \mathfrak{J} \to\mathfrak{J} $ is defined by 
$$
\sigma X=\sigma\begin{pmatrix} \xi_1 & x_3 & \ov{x}_2 \\
                                                \ov{x}_3 & \xi_2 & x_1 \\
                                                x_2 & \ov{x}_1 &\xi_3     
              \end{pmatrix} 
        = \begin{pmatrix} \xi_1 & -x_3 & -\ov{x}_2 \\
                                   -\ov{x}_3 & \xi_2 & x_1 \\
                                    -x_2 & \ov{x}_1 & \xi_3   
            \end{pmatrix},\,\,X \in \mathfrak{J}.
$$
(Note that the $\R$-linear transformation $\sigma$ of $\mathfrak{J}$ is naturally extended to the $C$-linear transformation of $\mathfrak{J}^C$.)
Hence ${\sigma'}_{\!\!4}$ induces the automorphisms $\tilde{\sigma'}_{\!\!4}$ of order 4 on $F_4$: $\tilde{\sigma'}_{\!\!4}(\alpha)={{\sigma'}_{\!\!4}}^{-1}\alpha{\sigma'}_{\!\!4}, \alpha \in F_4$, and using inclusion $F_4 \subset {F_4}^C$, the $\R$-linear transformation
 ${\sigma'}_{\!\!4}$ of $\mathfrak{J}$ is naturally extended to the $C$-linear transformation of $\mathfrak{J}^C$. Hence ${\sigma'}_{\!\!4}$ induces the automorphisms $\tilde{\sigma'}_{\!\!4}$ of order 4 on ${F_4}^C$: $\tilde{\sigma'}_{\!\!4}(\alpha)={{\sigma'}_{\!\!4}}^{-1}\alpha{\sigma'}_{\!\!4}, \alpha \in {F_4}^C$.
\vspace{1mm}

The simply connected complex Lie group ${E_6}^C$ is defined by
\begin{eqnarray*}
     {E_6}^C \!\!\!&=&\!\!\! \{\alpha \in \Iso_C(\mathfrak{J}^C) \, | \, \det \, \alpha X = \det \, X \}
\\
\!\!\!&=&\!\!\! \{\alpha \in \Iso_C(\mathfrak{J}^C) \, | \, (\alpha X, \alpha Y, \alpha Z)= (X,Y,Z)\}.
\end{eqnarray*}
Then we have naturally the inclusion ${F_4}^C \subset {E_6}^C$, and it is easy to see $({E_6}^C)_E = {F_4}^C$.
Moreover, the Lie algebra ${\mathfrak{e}_6}^C$ of the group ${E_6}^C$ is given by
$$
  {\mathfrak{e}_6}^C=\{\phi=\delta+\tilde{T}\,|\, \delta \in {\mathfrak{f}_4}^C, \, T \in (\mathfrak{J}^C)_0 \},
$$
where $(\mathfrak{J}^C)_0=\{X \in \mathfrak{J}^C \,|\,\tr(X)=0 \}$ and 
the $C$-linear mapping $\tilde{T}$ of $\mathfrak{J}^C$ is defined by $\tilde{T}X=T \circ X, X \in \mathfrak{J}^C$ (see \cite [Proposition 2.4.1, Theorem 3.2.1]{Yokotaichiro} in detail).
\vspace{1mm}

Let $\mathfrak{P}^C$ be the Freudenthal $C$-vector space 
$$
     \mathfrak{P}^C = \mathfrak{J}^C \oplus \mathfrak{J}^C \oplus C \oplus C, 
$$
in which the Freudenthal cross operation $P \times Q, P = (X, Y, \xi, \eta), Q = (Z, W, \zeta, \omega) \in \mathfrak{P}^C$, is defined as follows:
$$
      P \times Q = {\varPhi}(\phi, A, B, \nu), \quad
     \left \{ \begin{array}{l}
\vspace{1mm} 
          \phi = - \dfrac{1}{2}(X \vee W + Z \vee Y) \\
\vspace{1mm}
           A = - \dfrac{1}{4}(2Y \times W - \xi Z - \zeta X) \\
\vspace{1mm}
           B =  \dfrac{1}{4}(2X \times Z - \eta W - \omega Y) \\
\vspace{1mm}
           \nu = \dfrac{1}{8}((X, W) + (Z, Y) - 3(\xi\omega + \zeta\eta)),
\end{array} \right. 
$$
\noindent where $X \vee W \in {\mathfrak{e}_6}^C$ is defined by 
$$
    X \vee W = [\tilde{X}, \tilde{W}] + (X \circ W - \dfrac{1}{3}(X, W)E)^{\sim}, 
$$ 
\noindent  here the $C$-linear mappings $\tilde{X}, \tilde{W}$ of $\mathfrak{J}^C$ are same ones as in ${E_6}^C$.
\vspace{1mm}

The simply connected complex Lie group ${E_7}^C$ 
is defined by
\begin{eqnarray*}
    {E_7}^C \!\!\!&=&\!\!\!  \{\alpha \in \Iso_C(\mathfrak{P}^C) \, | \, \alpha(P \times Q)\alpha^{-1}=
\alpha P \times \alpha Q\}.
\end{eqnarray*}
Moreover, the Lie algebra ${\mathfrak{e}_7}^C$ 
of the group ${E_7}^C$
is given by
\begin{eqnarray*}
    {\mathfrak{e}_7}^C\!\!\!&=&\!\!\! \{\varPhi(\phi, A, B, \nu) \, | \, \phi \in {\mathfrak{e}_6}^C, A, B \in \mathfrak{J}^C, \nu \in C\}. 
\end{eqnarray*}
For $\alpha \in {E_6}^C$, the mapping $\ti{\alpha}: \mathfrak{P}^C \to \mathfrak{P}^C$ is defined by
$$
   \ti{\alpha}(X, Y, \xi, \eta)=(\alpha X, {}^t \alpha^{-1}Y, \xi, \eta),
$$
then we have $\ti{\alpha} \in {E_7}^C$, and so $\alpha$ and $\ti{\alpha}$ will be identified. The group ${E_7}^C$ contains ${E_6}^C$ as a subgroup by 
$$
 {E_6}^C=({E_7}^C)_{(0,0,1,0), (0,0,0,1)}.
$$
Hence we have the inclusion ${F_4}^C \subset {E_6}^C \subset {E_7}^C$.
Using these inclusions, the $C$-linear transformation
 ${\sigma'}_{\!\!4}$ of $\mathfrak{J}^C$ is naturally extended to the $C$-linear transformation of $\mathfrak{P}^C$:
$$
{\sigma'}_{\!\!4}(X, Y, \xi, \eta)=({\sigma'}_{\!\!4}X, {\sigma'}_{\!\!4} Y, \xi, \eta),\, (X, Y, \xi, \eta) \in \mathfrak{P}^C.
$$
Hence we see ${\sigma'}_{\!\!4} \in {E_7}^C$, and so ${\sigma'}_{\!\!4}$ induces the automorphisms $\tilde{\sigma'}_{\!\!4}$ of order 4 on ${E_7}^C$: $\tilde{\sigma'}_{\!\!4}(\alpha)={{\sigma'}_{\!\!4}}^{-1}\alpha{\sigma'}_{\!\!4}, \alpha \in {E_7}^C$.

Let ${\mathfrak{e}_8}^C$ be the $248$ dimensional $C$-vector space       
$$
    {\mathfrak{e}_8}^C = {\mathfrak{e}_7}^C \oplus \mathfrak{P}^C \oplus \mathfrak{P}^C \oplus C \oplus C \oplus C. 
$$
We define a Lie bracket $[R_1, R_2], R_1\!=\!(\varPhi_1, P_1, Q_1, r_1, s_1, t_1), R_2\!=\!(\varPhi_2, P_2, Q_2, r_2, s_2, t_2)$, by
$$
  [(\varPhi_1, P_1, Q_1, r_1, s_1, t_1), (\varPhi_2, P_2, Q_2, r_2, s_2, t_2)]=: 
  (\varPhi, P, Q, r, s, t), 
$$
$$
\left\{\begin{array}{l}
     \varPhi = {[}\varPhi_1, \varPhi_2] + P_1 \times Q_2 - P_2 \times Q_1
\vspace{1mm} \\
     P = \varPhi_1P_2 - \varPhi_2P_1 + r_1P_2 - r_2P_1 + s_1Q_2 - s_2Q_1 
\vspace{1mm} \\
     Q = \varPhi_1Q_2 - \varPhi_2Q_1 - r_1Q_2 + r_2Q_1 + t_1P_2 - t_2P_1
\vspace{1mm} \\
     r = - \dfrac{1}{8}\{P_1, Q_2\} + \dfrac{1}{8}\{P_2, Q_1\} + s_1t_2 - s_2t_1\vspace{1mm} \\
     s = \,\,\, \dfrac{1}{4}\{P_1, P_2\} + 2r_1s_2 - 2r_2s_1
\vspace{1mm} \\
     t = - \dfrac{1}{4}\{Q_1, Q_2\} - 2r_1t_2 + 2r_2t_1. 
\end{array} \right. 
$$
Then the $C$-vector space ${\mathfrak{e}_8}^C$ becomes a complex simple Lie algebra of type $E_8$.  
\vspace{2mm}

We define a $C$-linear transformation $\lambda_\omega$ of ${\mathfrak{e}_8}^C$
by
$$
       \lambda_\omega(\varPhi, P, Q, r, s, t) = (\lambda\varPhi\lambda^{-1}, \lambda Q, - \lambda P, -r, -t, -s), 
$$
where a $C$-linear transformation $\lambda$ of $\mathfrak{P}^C$ on the right-hand side is defined by
$ \lambda(X, Y, \xi, \eta)$ $ = (Y, - X, \eta, -\xi)$.
As in $\mathfrak{J}^C$, the complex conjugation in ${\mathfrak{e}_8}^C$ is denoted by $\tau$:     
$$
    \tau(\varPhi, P, Q, r, s, t) = (\tau\varPhi\tau, \tau P, \tau Q, \tau r, \tau s, \tau t). 
$$

The connected complex Lie group ${E_8}^C$ and the connected compact Lie group $E_8$ are defined by 
\begin{eqnarray*}
   {E_8}^C \!\!\!&=&\!\!\! \{\alpha \in \Iso_C({\mathfrak{e}_8}^C) \,|\, \alpha[R, R'] = [\alpha R, \alpha R']\}, 
\vspace{1mm} \\    
   E_8 \!\!\!&=&\!\!\! \{\alpha \in {E_8}^C \, | \, \tau\lambda_\omega \alpha\lambda_\omega  \tau = \alpha\} = ({E_8}^C)^{\tau\lambda_\omega }, 
\end{eqnarray*}
respectively. Moreover, the Lie algebra $\mathfrak{e}_8$ of the group $E_8$ is given by 
$$
 \mathfrak{e}_8=\{(\varPhi, P, -\tau\lambda P, r, s, -\tau s)\,|\,\varPhi \in \mathfrak{e}_7, P \in \mathfrak{J}^C, r \in i\R, s \in C  \}. 
$$

\noindent For $\alpha \in {E_7}^C$, the mapping $\tilde{\alpha} : {\mathfrak{e}_8}^C \to {\mathfrak{e}_8}^C$ is defined by
$$
     \tilde{\alpha}(\varPhi, P, Q, r, s, t) = (\alpha\varPhi\alpha^{-1}, \alpha P, \alpha Q, r, s, t),
$$
then we have $\tilde{\alpha} \in {E_8}^C$, and so $\alpha$ and $\tilde{\alpha}$ will be identified. The group ${E_8}^C$ contains ${E_7}^C$ as a subgroup by
$$       {E_7}^C  
=\{ \tilde{\alpha} \in {E_8}^C \,| \, \alpha \in {E_7}^C \}
=({E_8}^C)_{(0,0,0,1,0,0),(0,0,0,0,1,0),(0,0,0,0,0,1)}.
$$
Hence we have the inclusion ${E_7}^C \subset {E_8}^C$.
Using this inclusion, since the $C$-linear transformation
 ${\sigma'}_{\!\!4}$ of $\mathfrak{P}^C$ is naturally extended to the $C$-linear transformation of ${\mathfrak{e}_8}^C$: 
$$
{\sigma'}_{\!\!4}(\varPhi, P, Q, r,s,t)=({{\sigma'}_{\!\!4}}^{-1}\varPhi {\sigma'}_{\!\!4}, {\sigma'}_{\!\!4} P, {\sigma'}_{\!\!4} Q, r,s,t),\,(\varPhi, P, Q, r,s,t) \in{\mathfrak{e}_8}^C, 
$$ 
we have ${\sigma'}_{\!\!4} \in {E_8}^C$, and so ${\sigma'}_{\!\!4}$ induces the automorphisms $\tilde{\sigma'}_{\!\!4}$ of order 4 on ${E_8}^C$: $\tilde{\sigma'}_{\!\!4}(\alpha)={{\sigma'}_{\!\!4}}^{-1}\alpha{\sigma'}_{\!\!4}, \alpha \in {E_8}^C$, and so is $E_8$.
\vspace{1mm}

In the last of this section, we state two useful lemmas. 
\begin{lem}\label{lemma 2.1}
For Lie groups $G, G' $,  let a mapping $\varphi : G \to G'$ be a homomorphism of Lie groups. When $G'$ is connected, if $\Ker\,\varphi$ is discrete and $\dim(\mathfrak{g})=\dim(\mathfrak{g}')$, $\varphi$ is surjection.
\end{lem}
\begin{proof}
The proof is omitted (cf. \cite[Lemma 0.6 (2)]{realization G_2}).
\end{proof}
\begin{lem}[E. Cartan-Ra\v{s}evskii]\label{lemma 2.2}
Let $G$ be a simply connected Lie group with a finite order automorphism $\sigma$
 of $G$. Then $G^\sigma$ is connected.
\end{lem}
\begin{proof}
The proof is omitted (cf. \cite[Lemma 0.7]{realization G_2}).
\end{proof}

\noindent After this, using these lemmas without permission each times, we often prove lemma, proposition or theorem.
\section{ The group ${(F_4}^C)_{E_1, E_2, E_3, F_1(e_k), k=0, 1}$}

The aim of this section is to determine the structure of the group ${(F_4}^C)_{E_1, E_2,E_3, F_1(e_k), k=0, 1}$:
$$
({F_4}^C)_{E_1, E_2,E_3, F_1(e_k), k=0, 1}=\{ \alpha \in {F_4}^C \, |\,
\alpha E_i=E_i, i=1,2,3, 
\alpha  F_1(e_k)= F_1(e_k), k=0, 1\}. \vspace{1mm}
$$

Now, we start to make preparations. 
\vspace{1mm}

We define groups ${(F_4}^C)_{E_1, E_2,E_3}$ and $S\!pin(8,C)$ by
\begin{eqnarray*}
{(F_4}^C)_{E_1, E_2,E_3}\!\!\!&=&\!\!\! \{\alpha \in {F_4}^C\,|\,\alpha E_i=E_i,   i=1,2,3  \},
\\[1mm]
 S\!pin(8,C)\!\!\!&=&\!\!\!\{ (\alpha_1, \alpha_2, \alpha_3) \in S\!O(8,C)^{\times 3} \,|\, (\alpha_1 x)(\alpha_2 y)=\ov{\alpha_3 (\ov{xy})}, x, y \in \mathfrak{C}^C \},  
\end{eqnarray*}
respectively. Then we have the following theorem.
\begin{thm}\label{thm 3.1}
The group ${(F_4}^C)_{E_1, E_2,E_3}$ is isomorphic to $S\!pin(8,C)${\rm:} ${(F_4}^C)_{E_1, E_2,E_3} \cong $ \\ $ S\!pin(8,C)$.
\end{thm}
\begin{proof}
We define a mapping $\varphi:S\!pin(8,C) \to {(F_4}^C)_{E_1, E_2,E_3}$ by
$$
  \varphi((\alpha_1, \alpha_2, \alpha_3))X=\begin{pmatrix} 
                           \xi_1 & \alpha_3 x_3  & \ov{\alpha_2 x_2} \\
                       \ov{\alpha_3 x_3} & \xi_2 &  \alpha_1 x_1 \\
                          \alpha_2 x_2 & \ov{\alpha_1 x_1} &\xi_3     
                                                 \end{pmatrix}, \,X \in \mathfrak{J}^C. 
$$
This homomorphism $\varphi$ induces the isomorphism between ${(F_4}^C)_{E_1, E_2,E_3}$ and $S\!pin(8,C)$ (cf.  \cite[Theorem 2.7.1]{Yokotaichiro}).
\end{proof}
As necessary, we denote any element $\alpha \in {(F_4}^C)_{E_1, E_2,E_3}$ by $(\alpha_1, \alpha_2, \alpha_3) $ $\in S\!pin(8,C)$, that is, $\alpha=(\alpha_1, \alpha_2, \alpha_3)$.
\vspace{2mm}

We define an $\R$-linear transformation $\delta_1$ of $\mathfrak{C}$ by 
$$
 \delta_1: e_0 \to e_6,\, e_1 \to e_7,\, e_i \to e_i,\,i=2,3,4,5, \,e_6 \to e_0,\,e_7 \to e_1,
$$
basiswisely. Using matrix representation, the explicit form of $\delta_1$ is as follows:
$$
\delta_1=\begin{pmatrix}              0&&&&&&1&0  \\
                                      &0&&&&&0&1 \\
                                      &&1&&&&& \\
                                      &&&1&&&& \\
                                      &&&&1&&& \\
                                      &&&&&1&& \\
                                      1&0&&&&&0& \\
                                      0&1&&&&&&0 \\
                \end{pmatrix} \in M(8, \R),
$$
where the blanks are $0$. Then we easily see that $\delta_1 \in S\!O(8)$. The $\R$-linear transformation $\delta_1$ is naturally extended to the $C$-linear transformation of \vspace{1mm} $\mathfrak{C}^C$.

We consider groups ${(F_4}^C)_{E_1, E_2, E_3, F_1(e_k), k=0, 1,2,3,4,5}$ and ${(F_4}^C)_{E_1, E_2,E_3, F_1(e_k), k=2,3,4,5,6,7}$:
\begin{eqnarray*}
({F_4}^C)_{E_1, E_2,E_3, F_1(e_k), k=0,1,2,3,4,5}\!\!\!&=&\!\!\! \biggl\{ \alpha \in {F_4}^C \,\biggm| \,
\begin{array}{l}
\alpha E_i=E_i, i=1,2,3, \\
\alpha  F_1(e_k)= F_1(e_k), k=0,1,2,3,4,5 
\end{array}\biggr\},
\\[1mm]
({F_4}^C)_{E_1, E_2,E_3, F_1(e_k), k=2,3,4,5,6,7}\!\!\!&=&\!\!\! \biggl\{ \alpha \in {F_4}^C \,\biggm| \,
\begin{array}{l}
\alpha E_i=E_i, i=1,2,3, \\
\alpha  F_1(e_k)= F_1(e_k),k=2,3,4,5,6,7 
\end{array}\biggr\}.
\end{eqnarray*}
  
Hereafter, we often denote $k\!=\!0, 1, 2,3, 4, 5$ by abbreviated form $k\!=\!0,\ldots,5$, and also often denote these groups above by abbreviated forms $({F_4}^C)_{E_{1,2,3},F_1(0,\ldots,5)},  ({F_4}^C)_{E_{1,2,3},F_1(2,\ldots,7)}$ 
as example. The other cases are similar to these.

\begin{prop}\label{prop 3.2}
The group $({F_4}^C)_{E_{1,2,3},F_1(0,\ldots,5)}$ is isomorphic to the group $({F_4}^C)_{E_{1,2,3},F_1(2,\ldots,7)}${\rm:}\\$({F_4}^C)_{E_{1,2,3},F_1(0,\ldots,5)} \cong ({F_4}^C)_{E_{1,2,3},F_1(2,\ldots,7)}$.
\end{prop}
\begin{proof}

We define a mapping $\varphi:({F_4}^C)_{E_{1,2,3},F_1(0,\ldots,5)} \to ({F_4}^C)_{E_{1,2,3},F_1(2,\ldots,7)}$ by 
$$
    \varphi(\alpha)=\delta^{-1}\alpha\delta, 
$$
where $\delta=(\delta_1, \delta_2, \delta_3) \in S\!pin(8,C) \cong ({F_4}^C)_{E_1,E_2,E_3}$ (Theorem \ref{thm 3.1}), here $\delta_1$ is defined in previous page, and note that for this $\delta_1$ there exist 
$\delta_2, \delta_3 \in S\!O(8, C)$ by the Principal of triality on  $S\!O(8,C)$. From $\alpha, \delta \in ({F_4}^C)_{E_1,E_2,E_3}$,
it is easy to see that $\varphi(\alpha) \in ({F_4}^C)_{E_1,E_2,E_3}$.

\noindent Moreover, we have that
\begin{eqnarray*}
\varphi(\alpha)F_1(e_6)\!\!\!&=&\!\!\!(\delta^{-1}\alpha\delta)F_1(e_6)=(\delta^{-1}\alpha)F_1(\delta_1 e_6)
\\
\!\!\!&=&\!\!\!(\delta^{-1}\alpha)F_1(e_0)=\delta^{-1}F_1(e_0)=F_1({\delta_1}^{-1}e_0)=F_1(e_6)
\end{eqnarray*}
Similarly, we have  $
\varphi(\alpha)F_1(e_7)=F_1(e_7)$, and  
it is clear that  $
\varphi(\alpha)F_1(e_k)=F_1(e_k), k=2,3,4,5$. Hence 
we have $\varphi(\alpha) \in ({F_4}^C)_{E_{1,2,3},F_1(2,\ldots,7)}$, that is, 
$\varphi$ is well-defined. From the definition of the mapping $\varphi$, it is clear that $\varphi$ is bijection. 

Therefore we have the required isomorphism 
$$
({F_4}^C)_{E_{1,2,3},F_1(0,\ldots,5)} \cong ({F_4}^C)_{E_{1,2,3},F_1(2,\ldots,7)}.
$$
\end{proof}

Let the complex unitary group $U(1, \C^C)=\{\theta \in \C^C\,|\,{\theta}\,\ov{\theta}=1  \}$. Then we have the following lemma.

\begin{thm}\label{thm 3.3}
The group $({F_4}^C)_{E_{1,2,3},F_1(2,\ldots,7)}$ is isomorphic to  $U(1, \C^C)${\rm :} 
$({F_4}^C)_{E_{1,2,3},F_1(2,\ldots,7)} \cong U(1, \C^C)$.
\end{thm}
\begin{proof}
We define a mapping $\phi :U(1, \C^C) \to ({F_4}^C)_{E_{1,2,3},F_1(2,\ldots,7)}$ by
$$
  \phi(\theta)X=\begin{pmatrix} 
                                  \xi_1 & x_3\theta & \ov{\theta x_2} \\
                  \ov{x_3 \theta} & \xi_2 & \ov{\theta}x_1 \ov{\theta} \\
                   \theta x_2 & \theta\,\ov{x}_1\theta &\xi_3     
              \end{pmatrix}, X \in \mathfrak{J}^C. 
$$
Then $\phi$ is well-defined. Indeed, by using  the relational formula $\Re(x(yz))=\Re(y(zx))=\Re(z(xy)), x, y, z \in \mathfrak{C}^C$, we have that
\begin{eqnarray*}
\det\,(\phi(\theta)X)\!\!\!&=&\!\!\!\xi_1\xi_2\xi_3+2{\Re}((\ov{\theta}x_1 \ov{\theta})(\theta x_2)( x_3) \theta))-\xi_1(\ov{\theta}x_1 \ov{\theta})(\theta\,\ov{x}_1\theta)
\\
&&\hspace*{40mm}-\xi_2(\theta x_2)(\ov{\theta_2 x_2})-\xi_3(x_3\theta)(\ov{x_3\theta})
\\
\!\!\!&=&\!\!\!\xi_1\xi_2\xi_3+2{\Re}((\ov{\theta}x_1 \ov{\theta})({\theta (x_2 x_3) \theta}))-\xi_1|\ov{\theta}x_1 \ov{\theta}|^2
-\xi_2|\theta x_2|^2-\xi_3|x_3\theta|^2\,
\\
\!\!\!&=&\!\!\!\xi_1\xi_2\xi_3+2{\Re}(\theta(\ov{\theta}x_1 \ov{\theta}))(\theta (x_2 x_3))-\xi_1|x_1|^2
-\xi_2| x_2|^2-\xi_3|x_3|^2
\\
\!\!\!&=&\!\!\!\xi_1\xi_2\xi_3+2{\Re}(x_1 \ov{\theta})(\theta (x_2 x_3))-\xi_1|x_1|^2
-\xi_2| x_2|^2-\xi_3|x_3|^2
\\
\!\!\!&=&\!\!\!\xi_1\xi_2\xi_3+2{\Re}((x_2 x_3)((x_1 \ov{\theta})\theta )-\xi_1|x_1|^2
-\xi_2| x_2|^2-\xi_3|x_3|^2
\\
\!\!\!&=&\!\!\!\xi_1\xi_2\xi_3+2{\Re}((x_2 x_3)(x_1 )-\xi_1|x_1|^2
-\xi_2| x_2|^2-\xi_3|x_3|^2
\\
\!\!\!&=&\!\!\!\xi_1\xi_2\xi_3+2{\Re}(x_1 x_2 x_3 )-\xi_1 x_1\ov{x}_1
-\xi_2 x_2\ov{x}_2-\xi_ 3 x_3\ov{x}_3
\\
\!\!\!&=&\!\!\! \det\,X ,
\end{eqnarray*}
and it is clear that $(\phi(\theta)X, \phi(\theta)Y)=(X,Y), X,Y \in \mathfrak{J}^C$ and $\phi(\theta)E_i=E_i, i=1,2,3$.
Hence we see that $\phi(\theta) \in ({F_4}^C)_{E_{1,2,3}}$. Moreover, from $e_ia=\ov{a}e_i, i=2,\ldots, 7, a \in U(1,\C^C)$, we have that $\phi(\theta)F_1(e_i)=F_1(e_i),i=2,\ldots, 7$, that is, $\phi(\theta) \in ({F_4}^C)_{E_{1,2,3},F_1(2,\ldots,7)}$. Needless to say, $\phi$ is a homomorphism. We shall show that $\phi$ is surjection. Let $\alpha \in ({F_4}^C)_{E_{1,2,3},F_1(2,\ldots,7)}$. Here, set
$$
   ({\mathfrak{J}^C})_k = \{F_k(x)\,|\, x \in \mathfrak{C}^C \}=\{X \in \mathfrak{J}^C\,|\,2E_{k+1} \circ X=2E_{k+2} \circ X=X \}, k=1,2,3,
$$
where the indices are considered as mod $3$.
Then from $\alpha E_i=E_i$, we have $\alpha X \in ({\mathfrak{J}^C})_k$ for $X \in ({\mathfrak{J}^C})_k $, and so $\alpha$ induces $C$-isomorphisms 
$$
   \alpha : ({\mathfrak{J}^C})_k \to ({\mathfrak{J}^C})_k, \quad \alpha_k:\mathfrak{C}^C \to \mathfrak{C}^C
$$
satisfying the conditions $\alpha F_k(x)=F_k(\alpha_k x), x \in \mathfrak{C}^C, k=1,2,3$. 

Applying $\alpha$ on $F_k(x) \circ F_k(y)=(x,y)(E_{k+1}+E_{k+2})$, that is, $\alpha F_k(x) \circ \alpha F_k(y)=(x,y)(E_{k+1}+E_{k+2})$, on the other hand we have
$$
      \alpha F_k(x) \circ \alpha F_k(y)=F_k(\alpha_k x) \circ F_k(\alpha_k y)=(\alpha_k x, \alpha_k y)(E_{k+1}+E_{k+2}).
$$
Hence we have $(\alpha_k x, \alpha_k y)=(x, y), x, y \in \mathfrak{C}^C$, that is, $\alpha_k \in O(8, C), k=1,2,3$.

\noindent Moreover, applying $\alpha$ on $F_1(x) \circ F_2(y)=({1}/{2}) F_3(\ov{xy})$,
we have $(\alpha_1 x)(\alpha_2 y)=\ov{\alpha_3 (\ov{xy})}$. Indeed, 
apply $\alpha$ on the left-hand side:
\begin{eqnarray*}
\alpha (F_1(x) \circ F_2(y))\!\!\!&=&\!\!\!\alpha F_1(x)\circ \alpha F_2 (y)=F_1(\alpha_1 x) \circ F_2(\alpha_2 y)
=\dfrac{1}{2}F_3(\ov{(\alpha_1 x)(\alpha_2 y)}),
\end{eqnarray*}
on the other hand, apply $\alpha$ on the right-hand side:
 $\alpha\bigl(({1}/{2})F_3 (\ov{xy})\bigr)=({1}/{2})F_3 (\alpha_3 (\ov{xy}))$. Hence we have $F_3(\ov{(\alpha_1 x)(\alpha_2 y)})=F_3 (\alpha_3 (\ov{xy}))$, that is, $(\alpha_1 x)(\alpha_2 y)=\ov{\alpha_3 (\ov{xy})}$. 

Thus since $\alpha_1, \alpha_2, \alpha_3 \in O(8, C)$ satisfy the condition $(\alpha_1 x)(\alpha_2 y)=\ov{\alpha_3 (\ov{xy})}$, we see that $\alpha_1, \alpha_2, \alpha_3 \in S\!O(8, C)$ (see \cite[Theotem 1.14.4]{Yokotaichiro}), and moreover from $\alpha F_1 (e_i)=F_1(e_i)$ we have $\alpha_1 e_i=e_i, i=2,\ldots, 7$. 
Hence
since we can confirm that $\alpha_1$ induces $C$-isomorphism of $\C^C \subset \mathfrak{C}^C$, there exists $\theta \in U(1, \C^C)$ such that $\alpha_1 x=\ov{\theta} x \ov{\theta}, x \in \mathfrak{C}^C$. For this $\theta$, by the Principal triality we can set $\alpha_2 x=\theta x, \alpha_3 x=x\theta, x \in \mathfrak{C}^C$. The proof of surjection is completed. Finally, it is easy to obtain that $\Ker\,\phi=\{ 1\}$.

Therefore we have the required isomorphism 
$$
({F_4}^C)_{E_{1,2,3},F_1(2,\ldots,7)} \cong U(1, \C^C).
$$
\end{proof}

From Proposition \ref{prop 3.2} and Theorem \ref{thm 3.3}, we have the following proposition.
\begin{prop}\label{prop 3.4}
The group $({F_4}^C)_{E_{1,2,3},F_1(0,\ldots,5)}$ is isomorphic $U(1, \C^C)${\rm:}
$({F_4}^C)_{E_{1,2,3},F_1(0,\ldots,5)} \cong U(1, \C^C)$.

In particular, the group $({F_4}^C)_{E_{1,2,3},F_1(0,\ldots,5)}$ is connected.
\end{prop}

We consider a group $({F_4}^C)_{E_{1,2,3},F_1(0,\ldots,4)}$:
$$
({F_4}^C)_{E_{1,2,3},F_1(0,\ldots,4)}= \biggl\{ \alpha \in {F_4}^C \,\biggm| \,
\begin{array}{l}
\alpha E_i=E_i, i=1,2,3, \\
\alpha  F_1(e_k)= F_1(e_k), k=0,1,2,3,4 
\end{array}\biggr\},
$$
and we shall construct $S\!pin(3,C)$ in ${F_4}^C$.
\begin{lem}\label{lem 3.5}
The Lie algebra $({\mathfrak{f}_4}^C)_{E_{1,2,3},F_1(0,\ldots,4)}$ of the group $({F_4}^C)_{E_{1,2,3},F_1(0,\ldots,4)}$ is given by
\begin{eqnarray*}
 ({\mathfrak{f}_4}^C)_{E_{1,2,3},F_1(0,\ldots,4)}\!\!\!&=&\!\!\!\biggl\{\delta \in {\mathfrak{f}_4}^C\,\biggm|\, \begin{array}{l}
\delta E_i=0, i=1,2,3, \\
\delta  F_1(e_k)= 0, k=0,1,2,3,4 
\end{array}\biggr\}
\\[1mm]
\!\!\!&=&\!\!\!\{\delta=d_{56}G_{56}+d_{57}G_{57}+d_{67}G_{67}\,|\,d_{kl} \in C \}.  
\end{eqnarray*}

In particular, $\dim_C(({\mathfrak{f}_4}^C)_{E_{1,2,3},F_1(0,\ldots,4)})=3$.
\end{lem}
\begin{proof}
By doing simple computation, this lemma is proved easily (As for $G_{ij}$, see \cite[Section 1.3]{Yokotaichiro}).
\end{proof}

We define a $3$-dimensional $C$-vector subspace $(V^C)^3$ of $\mathfrak{J}^C$ by
\begin{eqnarray*}
  (V^C)^3\!\!\!&=&\!\!\!\biggl\{X \in \mathfrak{J}^C \,\biggm|\,
\begin{array}{l}E_1 \circ X=0, (E_2, X)=(E_3, X)=0,\\
(F_1(e_k),X)=0,k=0,1,2,3,4 
\end{array} \biggr\}
\\[1mm]
\!\!\!&=&\!\!\!\{X =F_1(t)\,|\,t=t_5 e_5+t_6 e_6+t_7 e_7, t_k \in C \}
\end{eqnarray*}
with the norm $(X,X)=2({t_5}^2+{t_6}^2+{t_7}^2)$.
Obviously, the group $({F_4}^C)_{E_{1,2,3}, F_1(0,\ldots,4)}$ acts on $(V^C)^3$.

\begin{prop}\label{prop 3.6}
The homogeneous space $({F_4}^C)_{E_{1,2,3}, F_1(0,\ldots,4)}/U(1,\C^C)$ is homeomorphic to the complex sphere $(S^C)^2${\rm :} $({F_4}^C)_{E_{1,2,3}, F_1(0,\ldots,4)}/U(1,\C^C) \simeq (S^C)^2$.

In particular, the group $({F_4}^C)_{E_{1,2,3}, F_1(0,\ldots,4)}$ is connected.
\end{prop}
\begin{proof}
We define a $2$-dimensional complex sphere $(S^C)^2$ by
\begin{eqnarray*}
  (S^C)^2\!\!\!&=&\!\!\!\{X \in (V^C)^3 \,|\, (X, X)=2 \}
\\
\!\!\!&=&\!\!\! \{X=F_1 (t) \,|\, t=t_5 e_5+t_6 e_6+t_7 e_7, {t_5}^2+{t_6}^2+{t_7}^2=1, t_k \in C \}. 
\end{eqnarray*}
Then the group $({F_4}^C)_{E_{1,2,3}, F_1(0,\ldots,4)}$ acts on $(S^C)^2$, obviously. We shall show that this action is transitive.
In order to prove this, it is sufficient to show that any element $F_1(t) \in (S^C)^2$ can be transformed to $F_1(e_5) \in (S^C)^2$. 

Now, for a given $X=F_1 (t) \in (S^C)^2$, we choose $s_0 \in \R, 0\leq s_0 \leq \pi$ such that $\tan s_0={\Re(t_6)}/{\Re(t_5)}$ (if $\Re(t_5)=0$, let $s_0=\pi/2$). 

\noindent Operate $g_{56}(s_0):=\exp (s_0 G_{56}) \in (({F_4}^C)_{E_{1,2,3}, F_1(0,\ldots,4)})_0$ on $X=F_1(t)$ (Lemma \ref{lem 3.5}), then we have that
\begin{eqnarray*}
g_{56}(s_0)X\!\!\!&=&\!\!\! g_{56}(s_0)F_1(t)
\\
\!\!\!&=&\!\!\!F_1 (((\cos s_0)t_5 + (\sin s_0)t_6)e_5 +((\cos s_0)t_6 - (\sin s_0)t_5)e_6+t_7 e_7)
\\
\!\!\!&=&\!\!\!F_1 (((\cos s_0)t_5 + (\sin s_0)t_6)e_5 +i((\cos s_0)\Im(t_6) - (\sin s_0)\Im(t_5))e_6+t_7 e_7)
\\
\!\!\!&=&\!\!\!F_1(t^{(1)}_5 e_5+ir_6e_6+t_7 e_7)=:X^{(1)},
\end{eqnarray*}
where $t^{(1)}_5:=(\cos s_0)t_5 + (\sin s_0)t_6 \in C, r_6:=(\cos s_0)\Im(t_6) - (\sin s_0)\Im(t_5) \in \R$. 

\noindent Moreover, we choose $s_1 \in \R,0\leq s_1 \leq \pi$ such that $\tan s_1={\Re(t_7)}/{\Re(t^{(1)}_5)}$ (if $\Re(t^{(1)}_5)=0$, let $s_1=$ $\pi/2$).
Operate $g_{57}(s_1):=\exp (s_ 1G_{57}) \in ({F_4}^C)_{E_{1,2,3}, F_1(0,\ldots,4)}$ on $X'$ (Lemma \ref{lem 3.5}), then we have that
\begin{eqnarray*}
g_{57}(s_1)X^{(1)}\!\!\!&=&\!\!\! g_{57}(s_1)F_1((t^{(1)}_5 e_5+ir_6e_6+t_7 e_7)
\\
\!\!\!&=&\!\!\!F_1 (((\cos s_1)t^{(1)}_5 + (\sin s_1)t_7)e_5 +ir_6e_6+((\cos s_1)t_7 - (\sin s_1)t^{(1)}_5)e_7)
\\
\!\!\!&=&\!\!\!F_1 (((\cos s_1)t^{(1)}_5 \!+ (\sin s_1)t_7)e_5 +ir_6e_6+i((\cos s_1)\Im(t_7) - (\sin s_1)\Im(t^{(1)}_5))e_7)
\\
\!\!\!&=&\!\!\!F_1(t^{(2)}_5 e_5+ir_6e_6+ir_7 e_7)=:X^{(2)},
\end{eqnarray*}
where $t^{(2)}_5:=(\cos s_1)t^{(1)}_5 + (\sin s_1)t_7 \in C, r_7:=(\cos s_1)\Im(t_7) - (\sin s_1)\Im(t^{(1)}_5) \in \R$.

\noindent Additionally, we choose $s_2 \in \R,0\leq s_2 \leq \pi$ such that $\tan s_2={r_7}/{r_6}$ (if $r_6=0$, let $s_2=$ $\pi/2$).
Operate $g_{67}(s_2):=\exp (s_2G_{67}) \in ({F_4}^C)_{E_{1,2,3}, F_1(0,\ldots,4)}$ on $X^{(2)}$ (Lemma \ref{lem 3.5}), then we have that
\begin{eqnarray*}
g_{67}(s_2)X^{(2)}\!\!\!&=&\!\!\! g_{67}(s_2)F_1(t^{(2)}_5 e_5+ir_6e_6+ir_7 e_7)
\\
\!\!\!&=&\!\!\!F_1 (t^{(2)}_5 e_5 +i((\cos s_2)r_6+(\sin s_2)r_7)e_6+i((\cos s_2)r_7 - (\sin s_2)r_6)e_7)
\\
\!\!\!&=&\!\!\!F_1(t^{(2)}_5 e_5+ir^{(1)}_6e_6)=:X^{(3)},
\end{eqnarray*}
where $r^{(1)}_6:=(\cos s_2)r_6+(\sin s_2)r_7 \in \R$.

Here, note that it follows from $X^{(3)}=F_1(t^{(2)}_5 e_5+ir^{(1)}_6e_6) \in (S^C)^2 \,((t^{(2)}_5)^2+(ir^{(1)}_6)^2=1)$, that is, $(t^{(2)}_5)^2=1+(r^{(1)}_6)^2, t^{(2)}_5 \in C, r^{(1)}_6 \in \R$ that we have $t^{(2)}_5 \in \R-\{ 0\}$ and $-1 <{r^{(1)}_6}/{t^{(2)}_5} <1$. 

\noindent So, we can choose $s_3 \in \R$ such that $\tan (is_3)={ir^{(1)}_6}/{t^{(2)}_5}( \in i\R)$.
Indeed, because of  $\tan (is_3)=-i\,(1-{2}/({e^{-2s_3}+1}))\,(-1<1-{2}/({e^{-2s_3}+1}) <1)$, together with $-1 <{r^{(1)}_6}/{t^{(2)}_5} <1$, we can choose  $s_3 \in \R$. As in the first case above, operate $g_{56}(is_3)$ on $X^{(3)}$, then we have that
\begin{eqnarray*}
g_{56}(is_3)X^{(3)} \!\!\!&=&\!\!\! g_{56}(is_3)F_1(t^{(2)}_5 e_5+ir^{(1)}_6e_6)
\\
\!\!\!&=&\!\!\! F_1((\cos (is_3)t^{(2)}_5+\sin (is_3)(ir^{(1)}_6))e_5+(\cos (is_3)(ir^{(1)}_6)-\sin (is_3)t^{(2)}_5)e_6)
\\
\!\!\!&=&\!\!\! F_1((\cos (is_3)t^{(2)}_5+\sin (is_3)(ir^{(1)}_6))e_5)
\\
\!\!\!&=&\!\!\!F_1(t^{(3)}_5e_5),
\end{eqnarray*}
where $t^{(3)}_5:=\cos (is_3)t^{(2)}_5+\sin (is_3)(ir^{(1)}_6) \in C$. 

 Hence, from $F_1(t^{(3)}_5e_5) \in (S^C)^2$ we have that $t^{(3)}_5=1$ or $t^{(3)}_5=-1$. In the latter case, again operate $g_{57}(\pi)$ on $F_1(-e_5)$, then we have that $g_{57}(\pi)F_1(-e_5)=F_1(e_5)$. 

\noindent This shows the transitivity of action to $(S^C)^2$ by the group $({F_4}^C)_{E_{1,2,3}, F_1(0,\ldots,4)}$. 
The isotropy subgroup of the group $({F_4}^C)_{E_{1,2,3}, F_1(0,\ldots,4)}$ at $F_1(e_5)$ is $({F_4}^C)_{E_{1,2,3}, F_1(0,\ldots,5)} \cong U(1, \C^C)$ (Proposition \ref{prop 3.4}). Thus we have the required homeomorphism
$$
    ({F_4}^C)_{E_{1,2,3}, F_1(0,\ldots,4)}/U(1,\C^C) \simeq (S^C)^2.
$$

Therefore we see that the group $({F_4}^C)_{E_{1,2,3}, F_1(0,\ldots,4)}$ is connected.
\end{proof}

\begin{thm}\label{thm 3.7}
The group $({F_4}^C)_{E_{1,2,3}, F_1(0,\ldots,4)}$ is isomorphic to $S\!pin(3,C)${\rm:} $({F_4}^C)_{E_{1,2,3}, F_1(0,\ldots,4)} \cong S\!pin(3,C)$.
\end{thm}
\begin{proof}
Let $O(3,C)=O((V^C)^3)=\{\beta \in \Iso_{C} ((V^C)^3) \,|\,(\beta X, \beta Y)=(X, Y)\}$. We consider the restriction $\beta=\alpha \bigm|_{(V^C)^3}$ of $\alpha \in ({F_4}^C)_{E_{1,2,3}, F_1(0,\ldots,4)}$ to $(V^C)^3$, then we have $\beta \in O(3,C)$. Hence we can define a homomorphism $p: ({F_4}^C)_{E_{1,2,3}, F_1(0,\ldots,4)} \to O(3, C)=O((V^C)^3)$ by
$$
   p(\alpha)=\alpha \bigm|_{(V^C)^3}.
$$
Since the mapping $p$ is continuous and the group $({F_4}^C)_{E_{1,2,3}, F_1(0,\ldots,4)}$ is connected (Proposition \ref{prop 3.6}), the mapping $p$ induces  a homomorphism $p :({F_4}^C)_{E_{1,2,3}, F_1(0,\ldots,4)} \to S\!O(3, C)=S\!O((V^C)^3)$. 
It is not difficult to obtain that  $\Ker\,p=\{1, \sigma \} \cong \Z_2$. Indeed, Let $\alpha \in \Ker\,p $. Then, since $\alpha \in ({F_4}^C)_{E_{1,2,3}, F_1(0,\ldots,4)} \subset ({F_4}^C)_{E_{1,2,3}} \cong S\!pin(8, C)$, we can set $\alpha=(\alpha_1, \alpha_2, \alpha_3)$ (Theorem \ref{thm 3.1}). Moreover, from $\alpha F_1(e_i)=F_1(e_i), i=0, \ldots, 4$ and $\alpha \bigm|_{(V^C)^3}=1$, we have $\alpha_1 x=x$ for all $x \in \mathfrak{C}^C$, that is, $\alpha_1=1$. Hence we have that 
$$
  \alpha=(1,1,1)  \qquad {\text{or}}\qquad \alpha=(1, -1, -1)=\sigma,
$$
that is, $\Ker\,p \subset \{1, \sigma \}$  
and vice versa. Thus we obtain $\Ker\, p=\{1, \sigma \}$. From Lemma \ref{lem 3.5}, we have that
$$
   \dim_{C}(({\mathfrak{f}_4}^C)_{E_{1,2,3}, F_1(0,\ldots,4)})=3=\dim{{}_C}(\mathfrak{so}(3,C)),
$$
and in addition to this, $S\!O(3,C)$ is connected and $\Ker\,p$ is discrete. Hence $p$ is surjection. Thus we have the isomorphism 
$$
   ({F_4}^C)_{E_{1,2,3}, F_1(0,\ldots,4)}/\Z_2 \cong S\!O(3,C).
$$

Therefore the group $({F_4}^C)_{E_{1,2,3}, F_1(0,\ldots,4)}$ is isomorphic to  $S\!pin(3, C)$ as the universal covering group of $S\!O(3,C)$, that is, 
$
    ({F_4}^C)_{E_{1,2,3}, F_1(0,\ldots,4)} \cong S\!pin(3, C).  
$
\end{proof}

Continuously, we shall construct $S\!pin(4,C)$ in ${F_4}^C$.
\vspace{1mm}

We consider a group $({F_4}^C)_{E_{1,2,3},F_1(0,\ldots,3)}$:
$$
({F_4}^C)_{E_{1,2,3},F_1(0,\ldots,3)}= \biggl\{ \alpha \in {F_4}^C \,\biggm| \,
\begin{array}{l}
\alpha E_i=E_i, i=1,2,3, \\
\alpha  F_1(e_k)= F_1(e_k), k=0,1,2,3 
\end{array}\biggr\}.
$$
\begin{lem}\label{lem 3.8}
The Lie algebra $({\mathfrak{f}_4}^C)_{E_{1,2,3},F_1(0,\ldots,3)}$ of the group $({F_4}^C)_{E_{1,2,3},F_1(0,\ldots,3)}$ is given by
\begin{eqnarray*}
 ({\mathfrak{f}_4}^C)_{E_{1,2,3},F_1(0,\ldots,3)}\!\!\!&=&\!\!\!\{\delta \in {\mathfrak{f}_4}^C\,|\, 
\delta E_i=0, i=1,2,3, 
\,\delta  F_1(e_k)= 0, k=0,1,2,3
\}
\\[1mm]
\!\!\!&=&\!\!\!\{
  \delta=d_{45}G_{45}+d_{46}G_{46}+d_{47}G_{47}
+d_{56}G_{56}+d_{57}G_{57}+d_{67}G_{67}
\,|\,d_{kl} \in C \}.
\end{eqnarray*}

In particular, $\dim_C(({\mathfrak{f}_4}^C)_{E_{1,2,3},F_1(0,\ldots,3)})=6$.
\end{lem}
\begin{proof}
By doing simple computation, this lemma is proved easily(As for $G_{ij}$, see \cite[Section 1.3]{Yokotaichiro}).
\end{proof}
We define a $4$-dimensional $C$-vector subspace $(V^C)^4$ of $\mathfrak{J}^C$ by
\begin{eqnarray*}
  (V^C)^4\!\!\!&=&\!\!\!\biggl\{X \in \mathfrak{J}^C \,\biggm|\,
\begin{array}{l}E_1 \circ X=0, (E_2, X)=(E_3, X)=0,\\
(F_1(e_k),X)=0,k=0,1,2,3 
\end{array} \biggr\}
\\[1mm]
\!\!\!&=&\!\!\!\{X =F_1(t)\,|\,t=t_4 e_4+t_5 e_5+t_6 e_6+t_7 e_7, t_k \in C \}
\end{eqnarray*}
with the norm $(X,X)=2({t_4}^2+{t_5}^2+{t_6}^2+{t_7}^2)$.
Obviously, the group $({F_4}^C)_{E_{1,2,3}, F_1(0,\ldots,3)}$ acts on $(V^C)^4$.

\begin{prop}\label{prop 3.9}
The homogeneous space $({F_4}^C)_{E_{1,2,3}, F_1(0,\ldots,3)}/S\!pin(3,C)$ is homeomorphic to the complex sphere $(S^C)^3${\rm :} $({F_4}^C)_{E_{1,2,3}, F_1(0,\ldots,3)}/S\!pin(3,C) \simeq (S^C)^3$.

In particular, the group $({F_4}^C)_{E_{1,2,3}, F_1(0,\ldots,3)}$ is connected.
\end{prop}
\begin{proof}
We define a $3$-dimensional complex sphere $(S^C)^3$ by
\begin{eqnarray*}
  (S^C)^3\!\!\!&=&\!\!\!\{X \in (V^C)^3 \,|\, (X, X)=2 \}
\\
\!\!\!&=&\!\!\! \{X=F_1 (t) \,|\, t=t_4e_4+t_5 e_5+t_6 e_6+t_7 e_7, {t_4}^2+{t_5}^2+{t_6}^2+{t_7}^2=1, t_k \in C \} 
\end{eqnarray*}
Then the group $({F_4}^C)_{E_{1,2,3}, F_1(0,\ldots,3)}$ acts on $(S^C)^3$, obviously. We shall show that this action is transitive.
In order to prove this, it is sufficient to show that any element $F_1(t) \in (S^C)^5$ can be transformed to $F_1(e_4) \in (S^C)^3$.

Now, for a given $X=F_1(t) \in (S^C)^3$, we choose $s_0 \in  \R, 0\leq s_0 \leq \pi$ such that $\tan s_0=-{\Re(t_4)}/{\Re(t_5)}$ (if $\Re(t_5)=0$, let $s_0=\pi/2$). Operate $g_{45}(s_0):=\exp (s_0 G_{45}) \in (({F_4}^C)_{E_{1,2,3}, F_1(0,\ldots,3)})_0$ on $X=F_1(t)$ (Lemma \ref{lem 3.8}), then we have that
\begin{eqnarray*}
g_{45}(s_0)X\!\!\!&=&\!\!\!\! g_{45}(s_0)F_1(t)
\\
\!\!\!&=&\!\!\!\!F_1 (((\cos s_0)t_4 + (\sin s_0)t_5)e_4 +((\cos s_0)t_5 - (\sin s_0)t_4)e_5+t_6 e_6+t_7 e_7)
\\
\!\!\!&=&\!\!\!\!F_1 (i((\cos s_0)\Im(t_4) \!+ (\sin s_0)\Im(t_5) )e_4 \!+((\cos s_0)t_5-(\sin s_0)t_4) e_5+t_6 e_6+t_7 e_7)
\\
\!\!\!&=&\!\!\!\!F_1(ir^{(1)}_4 e_4+t^{(1)}_5 e_5+t_6 e_6+t_7 e_7)=:X^{(1)},
\end{eqnarray*}
where $r^{(1)}_4:=(\cos s_0)\Im(t_4) \!+ (\sin s_0)\Im(t_5) \in \R, t^{(1)}_5:=(\cos s_0)t_5-(\sin s_0)t_4 \in C$.

\noindent Moreover, we choose $s_1 \in \R, 0\leq s_1 \leq \pi$ such that $\tan s_1=-{\Re(t_5)}/{\Re(t_6)}$ (if $\Re(t_6)=0$, let $s_1=\pi/2$). Operate $g_{56}(s_1):=\exp (s_1 G_{56}) \in (({F_4}^C)_{E_{1,2,3}, F_1(0,\ldots,3)})_0$ on $X^{(1)}$ (Lemma \ref{lem 3.8}), then we have that
\begin{eqnarray*}
g_{56}(s_1)X^{(1)}\!\!\!&=&\!\!\! g_{34}(s_1)F_1(ir^{(1)}_4 e_4+t^{(1)}_5 e_5+t_6 e_6+t_7 e_7)
\\
\!\!\!&=&\!\!\!F_1 (ir^{(1)}_4 e_4+((\cos s_1)t_5 + (\sin s_1)t_6)e_5 +((\cos s_1)t_6 - (\sin s_1)t_5)e_6+t_7 e_7)
\\
\!\!\!&=&\!\!\!F_1 (ir^{(1)}_4 e_4+i((\cos s_1)\Im(t_5) + (\sin s_1)\Im(t_6) )e_5 +((\cos s_1)t_6-(\sin s_1)t_5) e_6
\\
&& \hspace*{90mm}+t_7 e_7)
\\[-3mm]
\!\!\!&=&\!\!\!F_1(ir^{(1)}_4 e_4+ir^{(1)}_5 e_5+t^{(1)}_6 e_6+t_7 e_7)=:X^{(2)},
\end{eqnarray*}
where $r^{(1)}_5:=(\cos s_1)\Im(t_5) + (\sin s_1)\Im(t_6)  \in \R, t^{(1)}_6:= (\cos s_1)t_6-(\sin s_1)t_5\in C $.

\noindent Additionally, we choose $s_2 \in \R, 0\leq s_2 \leq \pi$ such that $\tan s_2=-{(r^{(1)}_4)}/{(r^{(1)}_5)}$ (if $r^{(1)}_5=0$, let $s_2=\pi/2$). Again, operate $g_{45}(s_2)=\exp (s_2 G_{45}) \in (({F_4}^C)_{E_{1,2,3}, F_1(0,\ldots,3)})_0$ on $X^{(2)}$, then we have that
\begin{eqnarray*}
g_{45}(s_2)X^{(2)}\!\!\!&=&\!\!\! g_{45}(s_2)F_1(ir^{(1)}_4 e_4+ir^{(1)}_5 e_5+t^{(1)}_6 e_6+t_7 e_7)
\\
\!\!\!&=&\!\!\!F_1 (i((\cos s_2)r^{(1)}_4 + (\sin s_2)r^{(1)}_5)e_4 +i((\cos s_2)r^{(1)}_5 - (\sin s_2)r^{(1)}_4)e_5+t^{(1)}_6 e_6
\\
&& \hspace*{88mm}+t_7 e_7)
\\[-3mm]
\!\!\!&=&\!\!\!F_1 (i((\cos s_2)r^{(1)}_5 - (\sin s_2)r^{(1)}_4)e_5+t^{(1)}_6 e_6+t_7 e_7)
\\
\!\!\!&=&\!\!\!F_1(t^{(2)}_5 e_5+t^{(1)}_6 e_6+t_7 e_7)=:X' \in (S^C)^2,
\end{eqnarray*}
where $ t^{(2)}_5:= i((\cos s_2)r^{(1)}_5 - (\sin s_2)r^{(1)}_4)\in C $.

 Since $S\!pin(3,C)\, (\cong ({F_4}^C)_{E_{1,2,3}, F_1(0,\ldots,4)} \!\subset ({F_4}^C)_{E_{1,2,3}, F_1(0,\ldots,3)})$ acts transitively on $(S^C)^2$ (Proposition \ref{prop 3.6}), there exists $\alpha \in S\!pin(3,C)$ such that 
$$
  \alpha X'=F_1(e_5), X' \in (S^C)^2. 
$$ 
Again operate $g_{45}({\pi}/{2}) \in ({F_4}^C)_{E_{1,2,3}, F_1(0,\ldots,3)}$ on $F_1(e_5)$, then we have that 
$$
g_{45}\Bigl(\dfrac{\pi}{2}\Bigr)F_1(e_5) =F_1(e_4).
$$ 
This shows the  transitivity of action to $(S^C)^3$ by the group $({F_4}^C)_{E_{1,2,3}, F_1(0,\ldots,3)}$. The isotropy subgroup of the group $({F_4}^C)_{E_{1,2,3}, F_1(0,\ldots,3)}$ at $F_1(e_4)$ is $({F_4}^C)_{E_{1,2,3}, F_1(0,\ldots,4)} \cong S\!pin(3,C)$ (Theorem \ref{thm 3.7}). 
Thus we have the required homeomorphism 
$$
({F_4}^C)_{E_{1,2,3}, F_1(0,\ldots,3)}/S\!pin(3,C) \simeq (S^C)^3.
$$

Therefore we see that the group $({F_4}^C)_{E_{1,2,3}, F_1(0,\ldots,3)}$ is connected.
\end{proof}

\begin{thm}\label{thm 3.10}
The group $({F_4}^C)_{E_{1,2,3}, F_1(0,\ldots,3)}$ is isomorphic to $S\!pin(4,C)${\rm:} $({F_4}^C)_{E_{1,2,3}, F_1(0,\ldots,3)}$ $ \cong S\!pin(4,C)$.
\end{thm}
\begin{proof}
Since we can prove this theorem as in Theorem \ref{thm 3.7}
, this proof is omitted.
\end{proof}

Continuously, we shall construct $S\!pin(5,C)$ in ${F_4}^C$.
\vspace{1mm}

We consider a group $({F_4}^C)_{E_{1,2,3},F_1(0,1,2)}$:
$$
({F_4}^C)_{E_{1,2,3},F_1(0,1,2)}= \biggl\{ \alpha \in {F_4}^C \,\biggm| \,
\begin{array}{l}
\alpha E_i=E_i, i=1,2,3, \\
\alpha  F_1(e_k)= F_1(e_k), k=0,1,2
\end{array}\biggr\}.
$$
\begin{lem}\label{lem 3.11}
The Lie algebra $({\mathfrak{f}_4}^C)_{E_{1,2,3},F_1(0,1,2)}$ of the group $({F_4}^C)_{E_{1,2,3},F_1(0,1,2)}$ is given by
\begin{eqnarray*}
 ({\mathfrak{f}_4}^C)_{E_{1,2,3},F_1(0,1,2)}\!\!\!&=&\!\!\!\{\delta \in {\mathfrak{f}_4}^C\,|\, 
\delta E_i=0, i=1,2,3, 
\, \delta  F_1(e_k)= 0, k=0,1,2
\}
\\[1mm]
\!\!\!&=&\!\!\!\Biggl\{\begin{array}{l}
 \delta=d_{34}G_{34}+d_{35}G_{35}+d_{36}G_{36}+d_{37}G_{37}+d_{45}G_{45}
\\
\,\,\,\,\,\,+d_{46}G_{46}+d_{47}G_{47}+d_{56}G_{56}+d_{57}G_{57}+d_{67}G_{67}
\end{array}\,\Biggm|\,d_{kl} \in C \Biggr\}.  
\end{eqnarray*}

In particular, $\dim_C(({\mathfrak{f}_4}^C)_{E_{1,2,3},F_1(0,1,2)})=10$.
\end{lem}
\begin{proof}
By doing simple computation, this lemma is proved easily (As for $G_{ij}$, see \cite[Section 1.3]{Yokotaichiro}).
\end{proof}

We define a $5$-dimensional $C$-vector subspace $(V^C)^5$ of $\mathfrak{J}^C$ by
\begin{eqnarray*}
  (V^C)^5\!\!\!&=&\!\!\!\biggl\{X \in \mathfrak{J}^C \,\biggm|\,
\begin{array}{l}E_1 \circ X=0, (E_2, X)=(E_3, X)=0,\\
(F_1(e_k),X)=0,k=0,1,2
\end{array} \biggr\}
\\[1mm]
\!\!\!&=&\!\!\!\{X =F_1(t)\,|\,t=t_3 e_3+t_4 e_4+t_5 e_5+t_6 e_6+t_7 e_7, t_k \in C \}
\end{eqnarray*}
with the norm $(X,X)=2({t_3}^2+{t_4}^2+{t_5}^2+{t_6}^2+{t_7}^2)$.
Obviously, the group $({F_4}^C)_{E_{1,2,3}, F_1(0,1,2)}$ acts on $(V^C)^5$.

\begin{prop}\label{prop 3.12}
The homogeneous space $({F_4}^C)_{E_{1,2,3}, F_1(0,1,2)}/S\!pin(4,C)$ is homeomorphic to the complex sphere $(S^C)^4${\rm :} $({F_4}^C)_{E_{1,2,3}, F_1(0,1,2)}/S\!pin(4,C) \simeq (S^C)^4$.

In particular, the group $({F_4}^C)_{E_{1,2,3}, F_1(0,1,2)}$ is connected.
\end{prop}
\begin{proof}
We define a $4$-dimensional complex sphere $(S^C)^4$ by
\begin{eqnarray*}
  (S^C)^4\!\!\!&=&\!\!\!\{X \in (V^C)^4 \,|\, (X, X)=2 \}
\\
\!\!\!&=&\!\!\! \biggl\{X=F_1 (t) \,\biggm|\,\begin{array}{l}
t=t_3 e_3+t_4 e_4+t_5 e_5+t_6 e_6+t_7 e_7, 
\\
{t_3}^2+{t_4}^2+{t_5}^2+{t_6}^2+{t_7}^2=1, t_k \in C 
\end{array}\biggr \}. 
\end{eqnarray*}
Then the group $({F_4}^C)_{E_{1,2,3}, F_1(0,1,2)}$ acts on $(S^C)^4$, obviously. We shall show that this action is transitive.
In order to prove this, it is sufficient to show that any element $F_1(t) \in (S^C)^4$ can be transformed to $F_1(e_3) \in (S^C)^4$.

Now, for a given $X=F_1(t) \in (S^C)^4$, we choose $s_0 \in \R, 0\leq s_0 \leq \pi$ such that $\tan s_0=-{\Re(t_3)}/{\Re(t_4)}$ (if $\Re(t_4)=0$, let $\s_0=\pi/2$). Operate $g_{34}(s_0):=\exp (s_0 G_{34}) \in (({F_4}^C)_{E_{1,2,3}, F_1(0,1,2)})_0$ on $X=F_1(t)$ (Lemma \ref{lem 3.11}), then we have that
\begin{eqnarray*}
g_{34}(s_0)X\!\!\!&=&\!\!\! g_{34}(s_0)F_1(t)
\\
\!\!\!&=&\!\!\!F_1 (((\cos s_0)t_3 + (\sin s_0)t_4)e_3 +((\cos s_0)t_4 - (\sin s_0)t_3)e_4+t_5 e_5+t_6 e_6+t_7 e_7)
\\
\!\!\!&=&\!\!\!F_1 (i((\cos s_0)\Im(t_3) + (\sin s_0)\Im(t_4) )e_3 +((\cos s_0)t_4-(\sin s_0)t_3) e_4
\\
&& \hspace*{75mm}+t_5 e_5+t_6 e_6+t_7 e_7)
\\[-3mm]
\!\!\!&=&\!\!\!F_1(ir^{(1)}_3 e_3+t^{(1)}_4 e_4+t_5 e_5+t_6 e_6+t_7 e_7)=:X^{(1)},
\end{eqnarray*}
where $r^{(1)}_3:= (\cos s_0)\Im(t_3) + (\sin s_0)\Im(t_4) \in \R, t^{(1)}_4:=(\cos s_0)t_4-(\sin s_0)t_3 \in C$.

\noindent Moreover, we choose $s_1 \in \R, 0\leq s_1 \leq \pi$ such that $\tan s_1=-{\Re(t_4)}/{\Re(t_5)}$ (if $\Re(t_5)=0$, let $s_1=\pi/2$). Operate $g_{45}(s_1):=\exp (s_1 G_{45}) \in (({F_4}^C)_{E_{1,2,3}, F_1(0,1,2)})_0$ on $X^{(1)}$ (Lemma \ref{lem 3.11}), then we have that
\begin{eqnarray*}
g_{45}(s_1)X^{(1)}\!\!\!&=&\!\!\! g_{45}(s_1)F_1(ir^{(1)}_3 e_3+t^{(1)}_4 e_4+t_5 e_5+t_6 e_6+t_7 e_7)
\\
\!\!\!&=&\!\!\!F_1 (ir^{(1)}_3 e_3+((\cos s_1)t_4 + (\sin s_1)t_5)e_4 +((\cos s_1)t_5 - (\sin s_1)t_4)e_5
\\
&& \hspace*{82mm}+t_6 e_6+t_7 e_7)
\\
\!\!\!&=&\!\!\!F_1 (ir^{(1)}_3 e_3+i((\cos s_1)\Im(t_4) + (\sin s_1)\Im(t_5) )e_4 +((\cos s_1)t_5-(\sin s_1)t_4) e_5
\\
&& \hspace*{82mm}+t_6 e_6+t_7 e_7)
\\[-3mm]
\!\!\!&=&\!\!\!F_1(ir^{(1)}_3 e_3+ir^{(1)}_4 e_4+t^{(1)}_5 e_5+t_6 e_6+t_7 e_7):=X^{(2)},
\end{eqnarray*}
where $r^{(1)}_4:=(\cos s_1)\Im(t_4) + (\sin s_1)\Im(t_5) \in \R, t^{(1)}_5:=(\cos s_1)t_5-(\sin s_1)t_4 \in C $.

\noindent Additionally, we choose $s_2 \in \R, 0\leq s_2 \leq \pi$ such that $\tan s_2=-{(r^{(1)}_3)}/{(r^{(1)}_4)}$ (if $r^{(1)}_4=0$, let $s_2=\pi/2$). Again, operate $g_{34}(s_2)=\exp (s_2 G_{34}) \in (({F_4}^C)_{E_{1,2,3}, F_1(0,1,2)})_0$ on $X^{(2)}$ (Lemma \ref{lem 3.11}), then we have that
\begin{eqnarray*}
g_{34}(s_2)X^{(2)}\!\!\!&=&\!\!\! g_{34}(s_2)F_1(ir^{(1)}_3 e_3+ir^{(1)}_4 e_4+t^{(1)}_5 e_5+t_6 e_6+t_7 e_7)
\\
\!\!\!&=&\!\!\!F_1 (i((\cos s_2)r^{(1)}_3 + (\sin s_2)r^{(1)}_4)e_3 +i((\cos s_2)r^{(1)}_4 - (\sin s_2)r^{(1)}_3)e_4+t_5 e_5
\\
&& \hspace*{70mm}+t_5 e_5++t_6 e_6+t_7 e_7)
\\[-1mm]
\!\!\!&=&\!\!\!F_1 (i((\cos s_2)r^{(1)}_4 - (\sin s_2)r^{(1)}_3)e_4+t_5 e_5+t_6 e_6+t_7 e_7)
\\
\!\!\!&=&\!\!\!F_1(t^{(2)}_4 e_4+t_5 e_5+t_6 e_6+t_7 e_7)=:X' \in (S^C)^3,
\end{eqnarray*}
where $ t^{(2)}_4:=i((\cos s_2)r^{(1)}_4 - (\sin s_2)r^{(1)}_3) \in C $.

 Since $S\!pin(4,C)\, (\cong ({F_4}^C)_{E_{1,2,3}, F_1(0,\ldots,3)} \!\subset ({F_4}^C)_{E_{1,2,3}, F_1(0,\ldots,2)})$ acts transitively on $(S^C)^3$ (Proposition \ref{prop 3.9}), there exists $\alpha \in S\!pin(4,C)$ such that 
$$
  \alpha X'=F_1(e_4), X' \in (S^C)^3. 
$$ 
Again operate $g_{34}({\pi}/{2}) \in ({F_4}^C)_{E_{1,2,3}, F_1(0,1,2)}$ on $F_1(e_4)$, then we have that 
$$
g_{34}\Bigl(\dfrac{\pi}{2}\Bigr)(F_1(e_4)) =F_1(e_3).
$$
This shows the transitivity of action to $(S^C)^4$ by the group $({F_4}^C)_{E_{1,2,3}, F_1(0,1,2)}$. The isotropy subgroup of the group $({F_4}^C)_{E_{1,2,3}, F_1(0,1,2)}$ at $F_1(e_3)$ is $({F_4}^C)_{E_{1,2,3}, F_1(0,\ldots,3)} \cong S\!pin(4,C)$ (Theorem \ref{thm 3.10}). 
Thus we have the required homeomorphism 
$$
({F_4}^C)_{E_{1,2,3}, F_1(0,1,2)}/S\!pin(4,C) \simeq (S^C)^4.
$$

Therefore we see that the group $({F_4}^C)_{E_{1,2,3}, F_1(0,1,2)}$ is connected.
\end{proof}

\begin{thm}\label{thm 3.13}
The group $({F_4}^C)_{E_{1,2,3}, F_1(0,1,2)}$ is isomorphic to $S\!pin(5,C)${\rm:} $({F_4}^C)_{E_{1,2,3}, F_1(0,1,2)} $ $ \cong S\!pin(5,C)$.
\end{thm}
\begin{proof}
Since we can also prove this theorem as in Theorem \ref{thm 3.7}, this proof is omitted.
\end{proof}
\vspace{1mm}

Now, we determine the structure of the group $({F_4}^C)_{E_1, E_2,E_3, F_1(e_k), k=0,1}$ as the aim of this section . 

\begin{lem}\label{lem 3.14}
The Lie algebra $({\mathfrak{f}_4}^C)_{E_1, E_2,E_3, F_1(e_k), k=0,1}$ of the group $({F_4}^C)_{E_1, E_2,E_3, F_1(e_k), k=0,1}$ is given by
\begin{eqnarray*}
 ({\mathfrak{f}_4}^C)_{E_1, E_2,E_3, F_1(e_k), k=0,1}\!\!\!&=&\!\!\!\{\delta \in {\mathfrak{f}_4}^C\,|\, 
\delta E_i=0, i=1,2,3, 
\,\delta  F_1(e_k)= 0, k=0,1 
\}
\\[1mm]
\!\!\!&=&\!\!\!\left\{ \begin{array}{l}
\delta=d_{23}G_{23}+d_{24}G_{24}+d_{25}G_{25}+d_{26}G_{26}+d_{27}G_{27}\\
\quad +d_{34}G_{34}+d_{35}G_{35}+d_{36}G_{36}+d_{37}G_{37}+d_{45}G_{45}\\  
\quad +d_{46}G_{46}+d_{47}G_{47}+d_{56}G_{56}+d_{57}G_{57}+d_{67}G_{67}
\end{array} 
\!\left|\! \begin{array}{l}
\!{} \\
\! d_{kl} \in C \\
\!{}
\end{array}\right. \right \}.  
\end{eqnarray*}

In particular, $\dim_{{}_C}(({\mathfrak{f}_4}^C)_{E_1, E_2,E_3, F_1(e_k), k=0,1})=15$.
\end{lem}
\begin{proof}
By doing simple computation, this lemma is proved easily (As for $G_{ij}$, see \cite[Section 1.3]{Yokotaichiro}).
\end{proof}

We define a $6$-dimensional $C$-vector subspace $(V^C)^6$ of $\mathfrak{J}^C$ by
\begin{eqnarray*}
  (V^C)^6\!\!\!&=&\!\!\!\biggl\{X \in \mathfrak{J}^C \,\biggm|\,
\begin{array}{l}E_1 \circ X=0, 
(E_2, X)=(E_3, X)=0,\\
(F_1(e_k),X)=0,k=0,1
\end{array} \biggr\}
\\[1mm]
\!\!\!&=&\!\!\!\{X =F_1(t)\,|\,t=t_2 e_2+t_3 e_3+t_4 e_4+t_5 e_5+t_6 e_6+t_7 e_7, t_k \in C \}
\end{eqnarray*}
with the norm $(X,X)=2({t_2}^2+{t_3}^2+{t_4}^2+{t_5}^2+{t_6}^2+{t_7}^2)$.
The group $({F_4}^C)_{E_1, E_2,E_3, F_1(e_k), k=0,1}$ acts on $(V^C)^6$, obviously.

\begin{prop}\label{prop 3.15}
The homogeneous space $({F_4}^C)_{E_1, E_2,E_3, F_1(e_k), k=0,1}/S\!pin(5,C)$  is homeomorphic to the complex sphere $(S^C)^5${\rm:} $({F_4}^C)_{E_1, E_2,E_3, F_1(e_k), k=0,1}/S\!pin(5,C) \simeq (S^C)^5$.

In particular, the group $({F_4}^C)_{E_1, E_2,E_3, F_1(e_k), k=0,1}$ is connected.
\end{prop}
\begin{proof}
We define a $5$-dimensional complex sphere $(S^C)^5$ by
\begin{eqnarray*}
  (S^C)^5\!\!\!&=&\!\!\!\{X \in (V^C)^6 \,|\, (X, X)=2 \}
\\
\!\!\!&=&\!\!\! \biggl\{X=F_1 (t) \,\biggm|\, \begin{array}{l}t=t_2 e_2+t_3 e_3+t_4 e_4+t_5 e_5+t_6 e_6+t_7 e_7, \\
{t_2}^2+{t_3}^2+{t_4}^2+{t_5}^2+{t_6}^2+{t_7}^2=1, t_k \in C   
\end{array}\biggr\}.
\end{eqnarray*}
Then the group $({F_4}^C)_{E_1, E_2,E_3, F_1(e_k), k=0,1}$ acts on $(S^C)^5$, obviously. We shall show that this action is transitive.
In order to prove this, it is sufficient to show that any element $F_1(t) \in (S^C)^5$ can be transformed to $F_1(e_2) \in (S^C)^5$.

Now, for a given $X=F_1(t) \in (S^C)^5$, we choose $s_0 \in \R, 0\leq s_0 \leq \pi$ such that $\tan s_0=-{\Re(t_2)}/{\Re(t_3)}$ (if $\Re(t_3)=0$, let $s_0=\pi/2$). Operate $g_{23}(s_0):=\exp (s_0 G_{23}) \in ({F_4}^C)_{E_1, E_2,E_3, F_1(e_k), k=0,1}$ on $X=F_1(t)$ (Lemma \ref{lem 3.14}), then we have that
\begin{eqnarray*}
g_{23}(s_0)X\!\!\!&=&\!\!\! g_{23}(s_0)F_1(t)
\\
\!\!\!&=&\!\!\!F_1 (((\cos s_0)t_2 + (\sin s_0)t_3)e_2 +((\cos s_0)t_3 - (\sin s_0)t_2)e_3
\\
&& \hspace*{60mm}+t_4 e_4+t_5 e_5+t_6 e_6+t_7 e_7)
\\
\!\!\!&=&\!\!\!F_1 (i((\cos s_0)\Im(t_2) + (\sin s_0)\Im(t_3) )e_2 +((\cos s_0)t_3-(\sin s_0)t_2) e_3
\\
&& \hspace*{60mm}+t_4 e_4+t_5 e_5+t_6 e_6+t_7 e_7)
\\
\!\!\!&=&\!\!\!F_1(ir^{(1)}_2 e_2+t^{(1)}_3 e_3+t_4 e_4+t_5 e_5+t_6 e_6+t_7 e_7)=:X^{(1)},
\end{eqnarray*}
where $r^{(1)}_2:= (\cos s_0)\Im(t_2) + (\sin s_0)\Im(t_3) \in \R, t^{(1)}_3:=(\cos s_0)t_3-(\sin s_0)t_2 \in C$.

\noindent Moreover, we choose $s_1 \in \R, 0\leq s_1 \leq \pi$ such that $\tan s_1=-{\Re(t_3)}/{\Re(t_4)}$ (if $\Re(t_4)=0$, let $s_1=\pi/2$). Operate $g_{34}(s_1):=\exp (s_1 G_{34}) \in (({F_4}^C)_{E_1, E_2,E_3, F_1(e_k), k=0,1})_0$ on $X^{(1)}$ (Lemma \ref{lem 3.14}), then we have that
\begin{eqnarray*}
g_{34}(s_1)X^{(1)}\!\!\!&=&\!\!\! g_{34}(s_1)F_1(ir^{(1)}_2 e_2+t^{(1)}_3 e_3+t_4 e_4+t_5 e_5+t_6 e_6+t_7 e_7)
\\
\!\!\!&=&\!\!\!F_1 (ir^{(1)}_2 e_2+((\cos s_1)t_3 + (\sin s_1)t_4)e_3 +((\cos s_1)t_4 - (\sin s_1)t_3)e_4
\\
&& \hspace*{70mm}+t_5 e_5+t_6 e_6+t_7 e_7)
\\
\!\!\!&=&\!\!\!F_1 (ir^{(1)}_2 e_2+i((\cos s_1)\Im(t_3) + (\sin s_1)\Im(t_4) )e_3 +((\cos s_1)t_4-(\sin s_1)t_3) e_4
\\
&& \hspace*{70mm}+t_5 e_5+t_6 e_6+t_7 e_7)
\\
\!\!\!&=&\!\!\!F_1(ir^{(1)}_2 e_2+ir^{(1)}_3 e_3+t^{(1)}_4 e_4+t_5 e_5+t_6 e_6+t_7 e_7)=:X^{(2)},
\end{eqnarray*}
where $r^{(1)}_3:=(\cos s_1)\Im(t_3) + (\sin s_1)\Im(t_4) \in \R, t^{(1)}_4:=(\cos s_1)t_4-(\sin s_1)t_3 \in C $.

\noindent Additionally, we choose $s_2 \in \R, 0\leq s_2 \leq \pi$ such that $\tan s_2=-{(r^{(1)}_2)}/{(r^{(1)}_3)}$ (if $r^{(1)}_3=0$, let $s_2=\pi/2$). Again, operate $g_{23}(s_2)=\exp (s_2 G_{23}) \in (({F_4}^C)_{E_1, E_2,E_3, F_1(e_k), k=0,1})_0$ on $X^{(2)}$, then we have that
\begin{eqnarray*}
g_{23}(s_2)X^{(2)}\!\!\!&=&\!\!\! g_{23}(s_2)F_1(ir^{(1)}_2 e_2+ir^{(1)}_3 e_3+t^{(1)}_4 e_4+t_5 e_5+t_6 e_6+t_7 e_7)
\\
\!\!\!&=&\!\!\!F_1 (i((\cos s_2)r^{(1)}_2 + (\sin s_2)r^{(1)}_3)e_2 +i((\cos s_2)r^{(1)}_3 - (\sin s_2)r^{(1)}_2)e_3
\\
&& \hspace*{60mm}+t_4 e_4+t_5 e_5+t_6 e_6+t_7 e_7)
\\
\!\!\!&=&\!\!\!F_1 (i((\cos s_2)r^{(1)}_3 - (\sin s_2)r^{(1)}_2) e_3+t_4 e_4+t_5 e_5+t_6 e_6+t_7 e_7)
\\
\!\!\!&=&\!\!\!F_1(ir^{(2)}_3 e_3+t_4 e_4+t_5 e_5+t_6 e_6+t_7 e_7)=:X' \in (S^C)^4,
\end{eqnarray*}
where $ r^{(2)}_3:=(\cos s_2)r^{(1)}_3 - (\sin s_2)r^{(1)}_2 \in \R $.

Since $S\!pin(5,C) (\cong ({F_4}^C)_{E_{1,2,3}, F_1(0,\ldots,2)} \!\subset ({F_4}^C)_{E_1, E_2,E_3, F_1(e_k), k=0,1})$ acts transitively on $(S^C)^4$ (Proposition \ref{prop 3.12}), there exists $\alpha \in S\!pin(5,C)$ such that 
$$
  \alpha X'=F_1(e_3), X' \in (S^C)^4. 
$$ 
Again operate $g_{23}({\pi}/{2}) \in (({F_4}^C)_{E_1, E_2,E_3, F_1(e_k), k=0,1})_0$ on $F_1(e_3)$, then we have that
$$
g_{23}\Bigl(\dfrac{\pi}{2}\Bigr)F_1(e_3)=F_1(e_2).
$$
 This shows the transitivity of action to $(S^C)^5$ by the group $({F_4}^C)_{E_1, E_2,E_3, F_1(e_k), k=0,1}$. The isotropy subgroup of the group $({F_4}^C)_{E_1, E_2,E_3, F_1(e_k), k=0,1}$ at $F_1(e_2)$ is  $S\!pin(5,C)$ (Theorem \ref{thm 3.13}). 
Thus we have the required homeomorphism 
$$
({F_4}^C)_{E_1, E_2,E_3, F_1(e_k), k=0,1}/S\!pin(5,C) \simeq (S^C)^5.
$$

Therefore we see that the group $({F_4}^C)_{E_1, E_2,E_3, F_1(e_k), k=0,1}$ is connected.
\end{proof}

\begin{thm}\label{thm 3.16}
The group $({F_4}^C)_{E_1, E_2,E_3, F_1(e_k), k=0,1}$ is isomorphic to $S\!pin(6,C)${\rm :}\\ $({F_4}^C)_{E_1, E_2,E_3, F_1(e_k), k=0,1} \cong S\!pin(6,C)$.
\end{thm}
\begin{proof}
This proof is proved by an argument similar to the proof in Theorem \ref{thm 3.7}, however we write as possible in detail.
Let $O(6,C)=O((V^C)^6)=\{\beta \in \Iso_{C} ((V^C)^6)\,|\,(\beta X, \beta Y)=(X, Y)\}$. We consider the restriction $\beta=\alpha \bigm|_{(V^C)^6}$ of $\alpha \in ({F_4}^C)_{E_1, E_2,E_3, F_1(e_k), k=0,1}$ to $(V^C)^6$, then we have $\beta \in O(6,C)$. Hence we can define a homomorphism $p: ({F_4}^C)_{E_1, E_2,E_3, F_1(e_k), k=0,1} \to O(6, C)=O((V^C)^6)$ by
$$
   p(\alpha)=\alpha \bigm|_{(V^C)^6}.
$$

Since the mapping $p$ is continuous and the group $({F_4}^C)_{E_1, E_2,E_3, F_1(e_k), k=0,1}$ is connected (Proposition \ref{prop 3.15}), the mapping $p$ induces  a homomorphism $p :({F_4}^C)_{E_1, E_2,E_3, F_1(e_k), k=0,1} \to S\!O(6, C)$ $=S\!O((V^C)^6)$. 
It is easy to obtain that  $\Ker\,p=\{1, \sigma \} \cong \Z_2$. Indeed, Let $\alpha \in \Ker\,p $. Then, since $\alpha \in ({F_4}^C)_{E_1, E_2,E_3, F_1(e_k), k=0,1} \subset ({F_4}^C)_{E_1,E_2,E_3} \cong S\!pin(8, C)$ (Theorem \ref{thm 3.1}), we can set $\alpha=(\alpha_1, \alpha_2, \alpha_3)$. Moreover, from $\alpha F_1(e_k)=F_1(e_k), k=0,2$ and $\alpha \bigm|_{(V^C)^6}=1$, we have $\alpha_1 x=x$ for all $x \in \mathfrak{C}^C$, that is, $\alpha_1=1$. Hence we have that 
$$
  \alpha=(1,1,1)  \qquad {\text{or}}\qquad \alpha=(1, -1, -1)=\sigma,
$$
that is, $\Ker\,p \subset \{1, \sigma  \}$ and vice versa. Thus we obtain $\Ker\,p = \{1, \sigma  \}$.
From Lemma \ref{lem 3.11}, we have that
$
   \dim_{C}(({F_4}^C)_{E_1, E_2,E_3, F_1(e_k), k=0,1})=15=\dim{{}_C}(\mathfrak{so}(6,C)),
$
and in addition to this, since $S\!O(6,C)$ is connected and $\Ker\,p$ is discrete, $p$ is surjection. Thus we have the isomorphism 
$$
  ({F_4}^C)_{E_1, E_2,E_3, F_1(e_k), k=0,1}/\Z_2 \cong S\!O(6,C).
$$

Therefore the group $({F_4}^C)_{E_1, E_2,E_3, F_1(e_k), k=0,1}$ is isomorphic  $S\!pin(6, C)$ as the universal covering group of $S\!O(6,C)$, that is, 
$({F_4}^C)_{E_1, E_2,E_3, F_1(e_k), k=0,1} \cong S\!pin(6, C)$.
\end{proof}

Here, we make a summary of the results as the low dimensional spinor groups which were constructed in this section. It is as follows:
\begin{eqnarray*}
&&
({F_4}^C)_{E_1, E_2, E_3, F_1(e_k), k=0,1}\,\,\,\cong S\!pin(6,C)
\\[0mm]
&&\hspace*{15mm} \cup
\\[-1.5mm]
&&
({F_4}^C)_{E_1, E_2, E_3, F_1(e_k), k=0,1,2}\cong S\!pin(5,C)
\\[0mm]
&&\hspace*{15mm} \cup
\\[-1.5mm]
&&
({F_4}^C)_{E_1, E_2, E_3, F_1(e_k), k=0,\ldots,3}\cong S\!pin(4,C)
\\[0mm]
&&\hspace*{15mm} \cup
\\[-1.5mm]
&&
({F_4}^C)_{E_1, E_2, E_3, F_1(e_k), k=0,\ldots,4}\cong S\!pin(3,C)
\\[0mm]
&&\hspace*{15mm} \cup
\\[-1.5mm]
&&
({F_4}^C)_{E_1, E_2, E_3, F_1(e_k), k=0,\ldots,5}\cong S\!pin(2,C) \cong U(1, \C^C).
\end{eqnarray*}

In the last of this section, we prove the following important lemma.
\begin{lem}\label{lem 3.17}
The group $({F_4}^C)_{E_1, E_2,E_3, F_1(e_k), k=0,1} \cong S\!pin(6, C)$ is the subgroup of the group $({F_4}^C)^{\sigma'_{\!\!4}}=\{\alpha \in {F_4}^C \,|\, \sigma'_{\!\!4}\alpha=\alpha \sigma'_{\!\!4} \}${\rm :} $S\!pin(6, C) \cong ({F_4}^C)_{E_1, E_2,E_3, F_1(e_k), k=0,1} \subset  ({F_4}^C)^{\sigma'_{\!\!4}}$.
\end{lem}
\begin{proof}
We consider the following complex eigenspaces of $\sigma'_{\!\!4}$ in $\mathfrak{J}^C$:
\begin{eqnarray*}
(\mathfrak{J}^C)_{\sigma'_{\!\!4}}\!\!\!&=&\!\!\!\{X \in \mathfrak{J}^C\,|\, \sigma'_{\!\!4}X=X \} (\subset (\mathfrak{J}^C)_\sigma)
\\
\!\!\!&=&\!\!\!\{X=\xi_1 E_1+\xi_2 E_2+\xi_3 E_3+F_1(x_1)\,|\,\xi_k \in C, x_1 \in \C^C  \},
\\[1mm]
(\mathfrak{J}^C)_{-\sigma'_{\!\!4}}\!\!\!&=&\!\!\!\{X \in \mathfrak{J}^C\,|\, \sigma'_{\!\!4}X=-X \}(=\{ F_1({x_1}^\perp)\} \oplus (\mathfrak{J}^C)_{-\sigma} )
\\
\!\!\!&=&\!\!\!\{X=F_1({x_1}^\perp) +F_2(x_2)+F_3(x_3))\,|\,{x_1}^\perp \in (\C^C)^\perp\,{\text{in}}\,\mathfrak{C}^C, x_k \in \mathfrak{C}^C  \},
\end{eqnarray*}
where $(\mathfrak{J}^C)_\sigma, (\mathfrak{J}^C)_{-\sigma}$ are the complex eigenspaces of the $C$-linear transformation $\sigma$ in $\mathfrak{J}^C$.

Now, let $\alpha \in ({F_4}^C)_{E_1, E_2, E_3, F_1(e_k), k=0,1}$, and set $X =X_1 + X_2 \in (\mathfrak{J}^C)_{\sigma'_{\!\!4}} \oplus (\mathfrak{J}^C)_{-\sigma'_{\!\!4}}=\mathfrak{J}^C$. 
Then we have that
\begin{eqnarray*}
\sigma'_{\!\!4}\alpha X \!\!\!&=&\!\!\!\sigma'_{\!\!4}\alpha(X_1 +X_2)
=\sigma'_{\!\!4}\alpha X_1+\sigma'_{\!\!4}\alpha X_2
\\[0mm]
\!\!\!&=&\!\!\!\sigma'_{\!\!4}X_1+
\sigma'_{\!\!4}\alpha(F_1({x_1}^\perp) +F_2(x_2)+F_3(x_3))
\\[0mm]
\!\!\!&=&\!\!\!X_1+\sigma'_{\!\!4}\alpha F_1({x_1}^\perp)+\sigma'_{\!\!4}(F_2({x_2}')+F_3({x_3}')) \,\,(\alpha \in  ({F_4}^C)_{E_1}=({F_4}^C)^\sigma)
\\[0mm]
\!\!\!&=&\!\!\!X_1+\sigma'_{\!\!4}\alpha F_1({x_1}^\perp)-(F_2({x_2}')+F_3({x_3}'))
\\[0mm]
\!\!\!&=&\!\!\!X_1+\sigma'_{\!\!4}F_1({x_1}^\perp)-(F_2({x_2}')+F_3({x_3}'))
\\[0mm]
\!\!\!&=&\!\!\!X_1-F_1({x_1}^\perp)-F_2({x_2}')-F_3({x_3}'),
\\[0.5mm]
\alpha \sigma'_{\!\!4} X \!\!\!&=&\!\!\!\alpha \sigma'_{\!\!4}(X_1+X_2)=\alpha (X_1-X_2)=X_1-\alpha X_2
\\[0mm]
\!\!\!&=&\!\!\!X_1-\alpha (F_1({x_1}^\perp) +F_2(x_2)+F_3(x_3))
\\[0mm]
\!\!\!&=&\!\!\!X_1-\alpha F_1({x_1}^\perp)-(F_2({x_2}')+F_3({x_3}'))
\\[0mm]
\!\!\!&=&\!\!\!X_1-F_1({x_1}^\perp)-F_2({x_2}')-F_3({x_3}'). 
\end{eqnarray*}

Hence we see  $\sigma'_{\!\!4}\alpha=\alpha \sigma'_{\!\!4}$, that is, 
 $S\!pin(6, C) \cong ({F_4}^C)_{E_1, E_2,E_3, F_1(e_k), k=0,1} \subset  ({F_4}^C)^{\sigma'_{\!\!4}}$ .
\end{proof}
\vspace{-3mm}

\section{ The groups $({E_7}^C)^{{\sigma'}_{\!\!4}}$ and $({E_7}^C)^{{\sigma'}_{\!\!4}, \mathfrak{so}(6,C)}$}

The aim of this section is to show the connectedness of the group
 $({E_7}^C)^{{\sigma'}_{\!\!4}, \mathfrak{so}(6,C)}$ after determining the structure of the group \vspace{1mm}$({E_7}^C)^{{\sigma'}_{\!\!4}}$. 

Now, we define subgroups $({E_7}^C)^{{\sigma'}_{\!\!4}}$ and $({E_7}^C)^{{\sigma'}_{\!\!4}, \mathfrak{so}(6,C)}$ of ${E_7}^C$ respectively by
\begin{eqnarray*}
     ({E_7}^C)^{{\sigma'}_{\!\!4}} \!\!\! &=&\!\!\! \{\alpha \in {E_7}^C \, | \, {\sigma'}_{\!\!4}\alpha = \alpha{\sigma'}_{\!\!4} \},
\\[1mm]
     ({E_7}^C)^{{\sigma'}_{\!\!4}, \mathfrak{so}(6,C)} \!\!\! &=&\!\!\! \{\alpha \in ({E_7}^C)^{{\sigma'}_{\!\!4}} \, | \, \varPhi_D\alpha = \alpha\varPhi_D \; \mbox{for all}\; D \in \mathfrak{so}(6, C)\},
\end{eqnarray*}
where $\varPhi_D = (D, 0, 0, 0) \in {\mathfrak{e}_7}^C 
, D \in \mathfrak{so}(6,C)
\cong ({\mathfrak{f}_4}^C)_{E_1, E_2, E_3, F_1(e_k),k=0,1}$.
Hereafter, we often denote $D$ above by $D_6$.
\begin{lem}\label{lem 4.1}
We have the following.

{\rm (1)}\, The Lie algebra $({\mathfrak{e}_7}^C)^{{\sigma'}_{\!\!4}}$ of the group $({E_7}^C)^{{\sigma'}_{\!\!4}}$ is given by
\begin{eqnarray*}
({\mathfrak{e}_7}^C)^{{\sigma'}_{\!\!4}} \!\!\! &=&\!\!\! \{\varPhi \in {\mathfrak{e}_7}^C \, | \, {\sigma'}_{\!\!4}\varPhi = \varPhi\,{\sigma'}_{\!\!4} \} \
\\[1mm]
\!\!\! &=&\!\!\!\left\{\varPhi(\phi, A, B, \nu) \in {\mathfrak{e}_7}^C \,\left| \, 
\begin{array}{l}
     \phi =\left(
\begin{array}{@{\,}cc|ccc@{\,}}
\multicolumn{2}{c|}{\raisebox{-2.0ex}[-5pt][0pt]{\large $D_2$}}&&& \\
&&\multicolumn{3}{c}{\raisebox{0.9ex}[0pt]{\large $0$}} \\
\hline
&&\multicolumn{3}{c}{\raisebox{-10pt}[0pt][0pt]{\large $D_6$}}\\
\multicolumn{2}{c|}{\raisebox{1.0ex}[0pt]{\large $0$}}&&& \\
\end{array}
\right) \begin{array}{l}
\!+ \tilde{A}_1(a)
\\ \vspace{1mm}
\!+(\tau_1 E_1\!+\!\tau_2 E_2\!+\!\tau E_3\!+\!F_1(t_1))^\sim
\end{array}
\vspace{1mm}\\
\quad D_2 \in \mathfrak{so}(2,C), D_6 \in \mathfrak{so}(6,C), a \in \C^C, \tau_k \in C,
\\ \quad \tau_1+\tau_2+\tau_3=0, t_1 \in \C^C,
\\
A=\begin{pmatrix}\xi_1 & 0       & 0   \\
                    0   & \xi_2  & x_1 \\
                    0   & \ov{x}_1& \xi_3
   \end{pmatrix},
B=\begin{pmatrix}\eta_1 & 0       & 0   \\
                    0   & \eta_2  & y_1 \\
                    0   & \ov{y}_1& \eta_3
   \end{pmatrix},
\vspace{1mm}\\
\quad \xi_k \in C, x_1 \in \C^C, \eta_k \in C, y_1 \in \C^C, 
\\[1mm]
\nu \in C 
\end{array}
\right. \right \}.
\end{eqnarray*}

In particular, $\dim_C(({\mathfrak{e}_7}^C)^{{\sigma'}_{\!\!4}})=((1+15)+2+(2+2))+(3+2) \times 2 +1=33$.
\vspace{2mm}

{\rm (2)}\, The Lie algebra $({\mathfrak{e}_7}^C)^{{\sigma'}_{\!\!4}, \mathfrak{so}(6,C)}$ of the group $({E_7}^C)^{{\sigma'}_{\!\!4}, \mathfrak{so}(6,C)}$ is given by

\begin{eqnarray*}
({\mathfrak{e}_7}^C)^{{\sigma'}_{\!\!4}, \mathfrak{so}(6,C)} \!\!\! &=&\!\!\!\! \{\varPhi \in ({\mathfrak{e}_7}^C)^ {{\sigma'}_{\!\!4}}\, | \, [\varPhi, \varPhi_D]=0 \,\,\mbox{for all}\, \,D \in \mathfrak{so}(6,C) \} \
\\[1mm]
\!\!\! &=&\!\!\!\!\left\{\varPhi(\phi, A, B, \nu) \in {\mathfrak{e}_7}^C \left| \, 
\begin{array}{l}
     \phi = \left(
\begin{array}{@{\,}cc|cc@{\,\,\,}}
\multicolumn{2}{c|}{\raisebox{-2.0ex}[-5pt][0pt]{\large $D_2$}}&& \\
&&\multicolumn{2}{c}{\raisebox{0.9ex}[0pt]{ \large $0$}} \\
\hline
&&\multicolumn{2}{c}{\raisebox{-10pt}[0pt][0pt]{\large $0$}}\\
\multicolumn{2}{c|}{\raisebox{3pt}[0pt]{\large $0$}}&& \\
\end{array}
\right) \begin{array}{l}
\!+ \tilde{A}_1(a)
\\ \vspace{1mm}
\!+(\tau_1 E_1\!+\!\tau_2 E_2\!+\!\tau E_3\!+\!F_1(t_1))^\sim
\end{array}
\vspace{1mm}\\
\quad D_2 \in \mathfrak{so}(2,C), a \in \C^C, \tau_k \in C,
\\ \quad \tau_1+\tau_2+\tau_3=0, t_1 \in \C^C,
\\
A=\begin{pmatrix}\xi_1 & 0       & 0   \\
                    0   & \xi_2  & x_1 \\
                    0   & \ov{x}_1& \xi_3
   \end{pmatrix},
B=\begin{pmatrix}\eta_1 & 0       & 0   \\
                    0   & \eta_2  & y_1 \\
                    0   & \ov{y}_1& \eta_3
   \end{pmatrix},
\vspace{1mm}\\
\quad \xi_k \in C, x_1 \in \C^C, \eta_k \in C, y_1 \in \C^C, 
\\
\nu \in C 
\end{array}
\right.\!\!\!\!\! \right \}.
\end{eqnarray*}

In particular, $\dim_C(({\mathfrak{e}_7}^C)^{{\sigma'}_{\!\!4}, \mathfrak{so}(6,C)})=(1+2+(2+2))+(3+2) \times 2 +1=18$.
\end{lem}
\begin{proof}
By doing simple computation, this lemma is proved easily.
\end{proof}

First, making some preparations, we shall determine the structure of the group $({E_7}^C)^{{\sigma'}_{\!\!4}}$. Hereafter, we often use the following notations in $\mathfrak{P}^C$:
\begin{eqnarray*}
&&\dot{X}=(X, 0,0,0), \underset{\dot{}}{Y}=(0,Y,0,0), \dot{\xi}=(0,0,\xi, 0), \underset{\dot{}}{\eta}=(0,0,0,\eta),
\\[-2mm]
&&\tilde{E}_1=(0, E_1, 0, 1), \,\,\tilde{E}_{-1}=(0, -E_1, 0, 1),\,\,E_2\dot{+}E_3=(E_2+E_3, 0,0,0), 
\\
&&E_2\dot{-}E_3=(E_2-E_3, 0,0,0),\,\,\dot{F}_1(e_k)=(F_1(e_k),0,0,0), k=0, \ldots, 7.
\end{eqnarray*}

We define $C$-linear transformations $\kappa, \mu$ of $\mathfrak{P}^C$ by
\begin{eqnarray*}
  \kappa (X, Y, \xi, \eta)\!\!\!&=&\!\!\!(-\kappa_1 X, \kappa_1 Y, -\xi, \eta), 
\\
   \mu (X, Y, \xi, \eta)\!\!\!&=&\!\!\!(2 E_1 \times Y+\eta E_1, 
2 E_1 \times X+\xi E_1, (E_1, Y), (E_1, X)),
\end{eqnarray*}
respectively, where $\kappa_1 X=(E_1, X)E_1-4E_1 \times (E_1 \times X), X \in \mathfrak{J}^C$. The explicit forms of $\kappa, \mu$ are as follows:
\begin{eqnarray*}
\kappa(X, Y, \xi, \eta)\!\!\!&=&\!\!\!\kappa(\begin{pmatrix}
                               \xi_1 & x_3 & \ov{x}_2 \\
                               \ov{x}_3 & \xi_2 & x_1 \\
                               x_2  & \ov{x}_1 & \xi_3
                      \end{pmatrix}, \,\,
                      \begin{pmatrix}
                               \eta_1 & y_3 & \ov{y}_2 \\
                               \ov{y}_3 & \eta_2 & y_1 \\
                               y_2 & \ov{y}_1 & \eta_3
                      \end{pmatrix}, \xi, \eta ) 
\end{eqnarray*}
\begin{eqnarray*}
\hspace*{20mm}\!\!\!&=&\!\!\!(\begin{pmatrix}
                               -\xi_1 &    0  &   0 \\
                                   0  & \xi_2 & x_1 \\
                                   0  & \ov{x}_1 & \xi_3
                      \end{pmatrix}, \,\,
                      \begin{pmatrix}
                               \eta_1 & 0       & 0   \\
                                    0 & -\eta_2 & -y_1 \\
                                    0 & -\ov{y}_1 & -\eta_3
                      \end{pmatrix}, -\xi, \eta ),
\\[2mm]
\mu(X, Y, \xi, \eta)\!\!\!&=&\!\!\!(\begin{pmatrix}
                                 \eta &    0   &   0 \\
                                   0  & \eta_3 & -y_1 \\
                                   0  & -\ov{y}_1 & \eta_2
                      \end{pmatrix}, \,\,
                      \begin{pmatrix}
                                  \xi & 0       & 0   \\
                                    0 & \xi_3 & -x_1 \\
                                    0 & -\ov{x}_1 & \xi_2
                      \end{pmatrix}, \eta_1, \xi_1 ).
\end{eqnarray*}
By doing simple computation, we can easily confirm that $\kappa {\sigma'}_{\!\!4}={\sigma'}_{\!\!4}\kappa, \mu{\sigma'}_{\!\!4}={\sigma'}_{\!\!4}\mu$.
\vspace{1mm}

We define a group $(({E_7}^C)^{\kappa, \mu})_{\ti{E}_1,\ti{E}_{-1}, E_2\dot{+}E_3,E_2\dot{-}E_3, \dot{F}_1(e_k), k=0,1 }$ by
$$
(({E_7}^C)^{\kappa, \mu})_{\ti{E}_1,\ti{E}_{-1}, E_2\dot{+}E_3,E_2\dot{-}E_3, \dot{F}_1(e_k), k=0,1 }=\left\{\alpha \in {E_7}^C \,\left |\,\begin{array}{l}
        \kappa\alpha=\alpha\kappa,  \mu\alpha=\alpha\mu,\\
        \alpha \ti{E}_1=\ti{E}_1, \alpha \ti{E}_{-1}=\ti{E}_{-1}\\ 
        \alpha  (E_2\dot{+}E_3) =E_2\dot{+}E_3, \\
        \alpha  (E_2\dot{-}E_3) =E_2\dot{-}E_3, \\
        \alpha  \dot{F}_1(e_k)=\dot{F}_1(e_k), k=0,1
    \end{array} \right. \right \}.
$$
\begin{prop}\label{prop 4.2}
The group $(({E_7}^C)^{\kappa, \mu})_{\ti{E}_1,\ti{E}_{-1}, E_2\dot{+}E_3,E_2\dot{-}E_3, \dot{F}_1(e_k), k=0,1 }$ is isomorphic to \\$S\!pin(6, C)${\rm :}
$(({E_7}^C)^{\kappa, \mu})_{\ti{E}_1,\ti{E}_{-1}, E_2\dot{+}E_3,E_2\dot{-}E_3, \dot{F}_1(e_k), k=0,1 }\cong S\!pin(6,C)$.
\end{prop}
\begin{proof}
Let $\alpha \in (({E_7}^C)^{\kappa, \mu})_{\ti{E}_1,\ti{E}_{-1}, E_2\dot{+}E_3,E_2\dot{-}E_3, \dot{F}_1(e_k), k=0,1 }$. Then \vspace{0.5mm}from $\alpha (0, E_1, 0, 1)=(0, E_1, 0, 1)$ and $\alpha (0, -E_1, 0, 1)=(0, -E_1, 0, 1)$, we have that 
$
  \alpha (0,E_1,0,0)=(0,E_1,0,0), \,\alpha (0,0,0,1)$ $=(0,0,0,1).
$
Hence we see that
$
   \alpha (E_1,0,0,0)=(E_1,0,0,0), \,\alpha (0,0,1,0)=(0,0,1,0).
$
Indeed, it follows that
\begin{eqnarray*}
& &\alpha (E_1,0,0,0) =\alpha \mu (0,0,0,1)=\mu\alpha(0,0,0,1)
= \mu (0,0,0,1)=(E_1,0,0,0), 
\\
& &\alpha (0,0,1,0) =\alpha \mu (0,E_1,0,0)=\mu\alpha(0,E_1,0,0)=\mu (0,E_1,0,0)=(0,0,1,0).
\end{eqnarray*}
Thus from $\alpha \dot{1}=\dot{1}$ and $\alpha \underset{\dot{}}{1}=\underset{\dot{}}{1}$, we see $\alpha \in {E_6}^C$, moreover from $\alpha E_i =E_i, i=1,2,3$, that is, $\alpha E=E$, we see $\alpha \in {F_4}^C$. Note that suppose $\alpha \in {F_4}^C$, $\alpha$ satisfies  $\kappa\alpha=\alpha\kappa, \alpha\mu=\mu\alpha$, automatically. Hence we have $\alpha \in ({F_4}^C)_{E_1, E_2, E_3, {F}_1(e_k), k=0,1}$, and vice versa. Thus we have  
$$
(({E_7}^C)^{\kappa, \mu})_{\ti{E}_1,\ti{E}_{-1}, E_2\dot{+}E_3,E_2\dot{-}E_3, \dot{F}_1(e_k), k=0,1 }
= ({F_4}^C)_{E_1, E_2, E_3, {F}_1(e_k), k=0,1}.
$$
Therefore, from Theorem \ref{thm 3.16} we have the required isomorphism 
$$
   (({E_7}^C)^{\kappa, \mu})_{\ti{E}_1,\ti{E}_{-1}, E_2\dot{+}E_3,E_2\dot{-}E_3, \dot{F}_1(e_k), k=0,1 } \cong S\!pin(6,C).
$$
\end{proof}
First we shall construct one more $S\!pin(3,C)$ in ${F_4}^C$.

Here, we define a group $({F_4}^C)_{E_1, F_1(e_k), k=2,\ldots,7}
$ by
\begin{eqnarray*}
&&({F_4}^C)_{E_1, F_1(e_k), k=2,\ldots,7}=\{\alpha \in {F_4}^C \,|\, \alpha E_1=E_1, \alpha F_1(e_k)=F_1(e_k), k=2,\ldots, 7 \}, 
\end{eqnarray*}
moreover define a $3$-dimensional $C$-vector subspace $({V_-}^{\!\!C})^3$ of $\mathfrak{J}^C$ by
\vspace{-1mm}
\begin{eqnarray*}
({V_-}^{\!\!C})^3\!\!\!&=&\!\!\! \{X \in \mathfrak{J}^C \,|\, E_1 \circ X=0, \tr (X)=0, (F_1(e_k), X)=0, k=2, \ldots, 7  \}
\\[1mm]
\!\!\!&=&\!\!\!\biggl \{\begin{pmatrix} 0 & 0      & 0  \\
                                          0 & \xi    & x  \\
                                          0 & \ov{x} & -\xi
                          \end{pmatrix}\, \biggm| \, \xi \in C, x \in \C^C  \biggr \}
\end{eqnarray*}
with the norm $(X,X)=2(\xi^2+\ov{x}x)$.
Obviously, the group $({F_4}^C)_{E_1, F_1(e_k), k=2,\ldots,7}$ acts on $({V_-}^{\!\!C})^3$.

\begin{lem}\label{lem 4.3}
The Lie algebra $({\mathfrak{f}_4}^C)_{E_1, F_1(e_k), k=2,\ldots,7}$ of the group $({F_4}^C)_{E_1, F_1(e_k), k=2,\ldots,7}$ is given by
\begin{eqnarray*}
({\mathfrak{f}_4}^C)_{E_1, F_1(e_k), k=2,\ldots,7}\!\!\!&=&\!\!\!\left\{\delta \in {\mathfrak{f}_4}^C \,\left| \, \begin{array}{l}
\delta E_1=0, \\
\delta F_1(e_k)=0, k=2, \ldots, 7
\end{array} \right.\right \}
\\
\!\!\!&=&\!\!\! \Biggl\{ \delta=\left(
\begin{array}{@{\,}cc|ccc@{}}
\multicolumn{2}{c|}{\raisebox{-2.0ex}[-5pt][0pt]{\large $D_2$}}&&& \\
&&\multicolumn{3}{c}{\raisebox{0.9ex}[0pt]{\large $0$}} \\
\hline
&&\multicolumn{3}{c}{\raisebox{-10pt}[0pt][0pt]{\large $0$}}\\
\multicolumn{2}{c|}{\raisebox{1.0ex}[0pt]{\large $0$}}&&& \\
\end{array}
\right)+\tilde{A}_1(a)\,\Biggm|\, D_2 \in \mathfrak{so}(2,C), a \in \C^C  \Biggr\}.
\end{eqnarray*}

In particular, $\dim_C(({\mathfrak{f}_4}^C)_{E_1, F_1(e_k), k=2,\ldots,7})=3$.
\end{lem}
\begin{proof}
By doing simple computation, this lemma is proved easily.
\end{proof}

\begin{prop}\label{prop 4.4}
The homogeneous space $({F_4}^C)_{E_1, F_1(e_k), k=2,\ldots,7}/U(1,\C^C)$ is homeomorphic to the complex sphere $({S_-}^{\!\!C})^2${\rm :} $({F_4}^C)_{E_1, F_1(e_k), k=2,\ldots,7}/U(1,\C^C) \simeq ({S_-}^{\!\!C})^2$.

In particular, the group $({F_4}^C)_{E_1, F_1(e_k), k=2,\ldots,7}$ is connected.
\end{prop}
\begin{proof}
We define a $2$-dimensional complex sphere $({S_-}^{\!\!C})^2$  by 
\begin{eqnarray*}
({S_-}^{\!\!C})^2\!\!\!&=&\!\!\! \{X \in ({V_-}^{\!\!C})^3 \,|\, (X, X)=2  \}
\\[1mm]
\!\!\!&=&\!\!\!\biggl \{\begin{pmatrix} 0 & 0      & 0  \\
                                          0 & \xi    & x  \\
                                          0 & \ov{x} & -\xi
                          \end{pmatrix}\, \biggm| \xi^2+x\ov{x}=1, \, \xi \in C, x \in \C^C  \biggr \}. 
\end{eqnarray*}
Then the group $({F_4}^C)_{E_1,F_1(e_k), k=2,\ldots,7}$ acts on $({S_-}^{\!\!C})^2$, obviously. We shall show that this action is transitive. In order to prove this, it is sufficient to show that any element $X \in ({S_-}^{\!\!C})^2$  can be transformed to $E_2 - E_3 \in ({S_-}^{\!\!C})^2$.
Here, we prepare some element of $({F_4}^C)_{E_1,F_1(e_k), k=2,\ldots,7}$.
For $a \in \C^C$ such that $a\ov{a}\ne 0$, we define a $C$-linear transformation $\alpha(a)$ of $\mathfrak{J}^C, \alpha(a) X(\xi, x)=:Y(\eta, y)$, by
\begin{eqnarray*}
&&\left\{\begin{array}{l}
     \eta_1 = \xi_1 
\vspace{1mm}\\
     \eta_2 = \displaystyle{\frac{\xi_2 + \xi_3}{2}} 
            + \displaystyle{\frac{\xi_2 - \xi_3}{2}}\cos 2\nu
            + \displaystyle{\frac{(a,x_1)}{\nu}}\sin2 \nu
\vspace{1mm}\\
     \eta_3 = \displaystyle{\frac{\xi_2 + \xi_3}{2}}
            - \displaystyle{\frac{\xi_2 - \xi_3}{2}}\cos2 \nu
            - \displaystyle{\frac{(a,x_1)}{\nu}}\sin2 \nu,
\end{array}\right.\\
&&\left\{\begin{array}{l}
     y_1 = x_1 
            - \displaystyle{\frac{(\xi_2 - \xi_3)a}{2\nu}}\sin2 \nu
            - \displaystyle{\frac{2(a, x_1)a}{\nu^2}}\sin^2 \nu
\vspace{1mm}\\
     y_2 = x_2\cos\nu - \displaystyle{\frac{\ov{x_3a}}{\nu}}\sin\nu
\vspace{1mm}\\
     y_3 = x_3\cos|a| + \displaystyle{\frac{\ov{ax_2}}{\nu}}\sin\nu,
\end{array}\right.
\end{eqnarray*}
where $\nu \in C, \nu^2=a\ov{a}$. 

\noindent Then we see $\alpha(a)=\exp \tilde{A}(a) $ $ \in (({F_4}^C)_{E_1,F_1(e_k), k=2,\ldots,7})_0$ (Lemma \ref{lem 4.3}).

Now, let $X=\begin{pmatrix} 0 & 0      & 0  \\
                                          0 & \xi    & x  \\
                                          0 & \ov{x} & -\xi
                          \end{pmatrix} \in ({S_-}^{\!\!C})^2$. 
We choose $a \in \C^C$ such that $(a, x)=0$ and $a\ov{a}=({\pi}/{4})^2$. Operate $\alpha(a) \in (({F_4}^C)_{E_1,F_1(e_k), k=2,\ldots,7})_0$ on $X$, then we have that 
$$
\alpha(a)X=\begin{pmatrix} 0 & 0      & 0  \\
                           0 & 0    & x'  \\
                           0 & \ov{x}' & 0
                          \end{pmatrix}=:X', \,x'\ov{x}'=1.
$$
Moreover, using this $x'$ above, operate $\alpha({\pi}/{4}x')$ on $X'$, then we have 
$$
\alpha\Bigl(\dfrac{\pi}{4}x'\Bigr)X'=E_2-E_3.
$$
This shows the transitivity of this action to $({S_-}^{\!\!C})^2$ by the group $({F_4}^C)_{E_1,F_1(e_k), k=2,\ldots,7}$. The isotropy subgroup of $({F_4}^C)_{E_1,F_1(e_k), k=2,\ldots,7}$ at $E_2-E_3$ is the group $({F_4}^C)_{E_1,E_2-E_3,F_1(e_k), k=2,\ldots,7}=({F_4}^C)_{E_1,E_2+E_3,E_2-E_3,F_1(e_k), k=2,\ldots,7}=({F_4}^C)_{E_1,E_2,E_3,F_1(e_k), k=2,\ldots,7} \cong U(1, \C^C)$ \vspace{0.5mm}(Theorem \ref{thm 3.3}).

Thus we have the required homeomorphism 
$$
  ({F_4}^C)_{E_1,F_1(e_k), k=2,\ldots,7}/U(1, \C^C) \simeq ({S_-}^{\!\!C})^2.
$$

Therefore we see that the group $({F_4}^C)_{E_1,F_1(e_k), k=2,\ldots,7}$ is connected.
\end{proof}

\begin{thm}\label{thm 4.5}
The group $({F_4}^C)_{E_1,F_1(e_k), k=2,\ldots,7}$ is isomorphic to $S\!pin(3,C)${\rm :} \\
$({F_4}^C)_{E_1,F_1(e_k), k=2,\ldots,7} \cong S\!pin(3,C)$.
\end{thm}
\begin{proof}
Let $O(3,C)=O(({V_-}^{\!\!C})^3)=\{\beta \in \Iso_{C} (({V_-}^{\!\!C})^3)\,|\,(\beta X, \beta Y)=(X, Y)\}$. We consider the restriction $\beta=\alpha \bigm|_{({V_-}^{\!\!C})^3}$ of $\alpha \in ({F_4}^C)_{E_1, F_1(e_k), k=2,\ldots,7}$ to $({V_-}^{\!\!C})^3$, then we have $\beta \in O(3,C)$. Hence we can define a homomorphism $p: ({F_4}^C)_{E_1, F_1(e_k), k=2,\dots,7} \to O(3, C)=O(({V_-}^{\!\!C})^3)$ by
$$
   p(\alpha)=\alpha \bigm|_{({V_-}^{\!\!C})^3}.
$$
Since the mapping $p$ is continuous and the group $({F_4}^C)_{E_1,F_1(e_k), k=2,\ldots,7}$ is connected, the mapping $p$ induces  a homomorphism $p: ({F_4}^C)_{E_1,F_1(e_k), k=2,\ldots,7} \to S\!O(3, C)=S\!O(({V_-}^{\!\!C})^3)$ by 
$$
   p(\alpha)=\alpha|_{{({V_-}^{\!\!C})^3}}.
$$

It is not difficult to obtain that $\Ker\,p=\{1, \sigma  \} \cong \Z_2$. 
Indeed, let $\alpha \in \Ker\,p=\{\alpha \in  ({F_4}^C)_{E_1,F_1(e_k), k=2,\ldots,7}\,|\,p(\alpha)=1 \}$, that is, $\alpha \in \{\alpha \in  ({F_4}^C)_{E_1,F_1(e_k), k=2,\ldots,7}\,|\,\alpha|_{{({V_-}^{\!\!C})^3}}=1 \}$. 
It follows from $\alpha E_1=E_1, \alpha E=E$ that $\alpha(E_2+E_3)=E_2+E_3$, moreover since $\alpha|_{{({V_-}^{\!\!C})^3}}=1$, 
we also see $\alpha(E_2-E_3)=E_2-E_3$. Hence, since we have $\alpha E_1=E_1, \alpha E_2=E_2,\alpha E_3=E_3$,
we see $\alpha \in ({F_4}^C)_{E_1, E_2, E_3} \cong S\!pin(8, C)$, and so set $\alpha=(\alpha_1, \alpha_2, \alpha_3), \alpha_k \in S\!O(8,C)$.
Thus again from $\alpha F_1(e_k)=F_1(e_k), {\scriptstyle k=2,\ldots,7}$ and $\alpha|_{{({V_-}^{\!\!C})^3}}=1$, we have $\alpha_1=1$. Hence from the Principal of triality on $S\!O(8,C)$, we see $\alpha=(1,1,1)=1$ or $\alpha=(1,-1,-1)=\sigma$, that is, $\Ker\,p \subset \{1, \sigma \}$ and vice versa. Thus we see $\Ker\,p=\{1, \sigma  \}$.
Finally, we shall show that $p$ is surjection. Since $S\!O(3,C)$ is connected, $\Ker\,p$ is discrete and 
$\dim_C(({\mathfrak{f}_4}^C)_{E_1,F_1(e_k), k=2,\ldots,7})=3=\dim_C(\mathfrak{so}(3,C))$ (Lemma \ref{lem 4.3}), 
$p$ is surjection.
Thus we have that 
$$
({F_4}^C)_{E_1,F_1(e_k), k=2,\ldots,7}/\Z_2 \cong S\!O(3,C).
$$ 

Therefore the group $({F_4}^C)_{E_1,F_1(e_k), k=2,\ldots,7}$ is isomorphic to $S\!pin(3, C)$ as the universal covering group of $S\!O(3,C)$, that is, $({F_4}^C)_{E_1,F_1(e_k), k=2,\ldots,7} \cong S\!pin(3, C)$.
\end{proof}
Next, we shall construct one more $S\!pin(4,C)$ in ${E_6}^C$, and so we define subgroups  
$(({E_6}^C)^\sigma)_{E_1}, (({E_6}^C)^\sigma)_{E_1, F_1(e_k), k=2,\ldots,7}
$ of ${E_6}^C$ by
\begin{eqnarray*}
  (({E_6}^C)^\sigma)_{E_1}\!\!\!&=&\!\!\!\{\alpha \in {E_6}^C \,|\, \sigma \alpha=\alpha\sigma, \alpha E_1=E_1 \} (\cong S\!pin(10,C)),
\\[0mm]
  (({E_6}^C)^\sigma)_{E_1, F_1(e_k), k=2,\ldots,7}\!\!\!&=&\!\!\!\{\alpha \in (({E_6}^C)^\sigma)_{E_1} \,|\, \alpha F_1(e_k)=F_1(e_k), k=2,\ldots, 7 \},
\end{eqnarray*}
respectively, where as for $(({E_6}^C)^\sigma)_{E_1} \!\cong \! S\!pin(10,C)$,  see \cite[Proposition 3.6.4]{realization G_2} in detail, and the $C$-linear transformation $\sigma$ defined in Section 2 induces the involutive automorphism $\ti{\sigma}$ of ${E_6}^C$,
moreover define a $4$-dimensional $C$-vector subspace $({V_-}^{\!\!C})^4$ of $\mathfrak{J}^C$ by
\begin{eqnarray*}
  ({V_-}^{\!\!C})^4\!\!\!&=&\!\!\!\{X \in \mathfrak{J}^C \,|\,4E_1 \times (E_1 \times X)=X, F_1(e_k) \times X=0, k=2, \ldots,7\}
\\[0mm]
  \!\!\!&=&\!\!\!\Bigl\{ \begin{pmatrix}0 & 0 & 0 \\
                                        0 & \xi_2 & x_1 \\
                                        0 &  \ov{x}_1 & \xi_3
                          \end{pmatrix}\,\Bigm |\,  \xi_i \in C, x_1 \in \C^C \subset \mathfrak{C}^C 
\Bigr\}
\end{eqnarray*}
with the norm $(-E_1, X, X)=-\xi_2\xi_3+\ov{x}_1 x_1$. The group $(({E_6}^C)^\sigma)_{E_1, F_1(e_k), k=2,\ldots,7}$ acts on $({V_-}^{\!\!C})^4$, obviously.
\begin{lem}\label{lem 4.6}
The Lie algebra $(({\mathfrak{e}_6}^C)^\sigma)_{E_1, F_1(e_k), k=2,\ldots,7}$ of the group $(({E_6}^C)^\sigma)_{E_1, F_1(e_k), k=2,\ldots,7}$ is given by
\begin{eqnarray*}
{(({\mathfrak{e}_6}^C)^\sigma)}_{E_1, F_1(e_k), k=2,\ldots,7}\!\!\!\!&=&\!\!\!\!\Bigl\{ \phi \in {\mathfrak{e}_6}^C \,\Bigm|\, 
\begin{array}{l} 
\sigma \phi=\phi \sigma,
\\
\phi E_1=0, \phi F_1(e_k)=0, k=2,\ldots, 7 
\end{array} \Bigr \}
\\[0mm]
\!\!\!\!&=&\!\!\!\!\Biggl\{ \phi=\left(
\begin{array}{@{\,}cc|ccc@{}}
\multicolumn{2}{c|}{\raisebox{-2.0ex}[-5pt][0pt]{\large $D_2$}}&&& \\
&&\multicolumn{3}{c}{\raisebox{0.9ex}[0pt]{\large $0$}} \\
\hline
&&\multicolumn{3}{c}{\raisebox{-10pt}[0pt][0pt]{\large $0$}}\\
\multicolumn{2}{c|}{\raisebox{1.0ex}[0pt]{\large $0$}}&&& \\
\end{array}
\right)
\begin{array}{l}
+\tilde{A}_1(a)\\
+(\tau E_2-\tau E_3+F_1(t))^\sim 
\end{array}\!\!\Biggm|\!\!\! \begin{array}{l}D_2 \in \mathfrak{so}(2,C),\\
 a, t \in \C^C, \tau \in C
\end{array} \Biggr\}.
\end{eqnarray*}

In particular, $\dim_C((({\mathfrak{e}_6}^C)^\sigma)_{E_1, F_1(e_k), k=2,\ldots,7})=6$.
\end{lem}
\begin{proof}
By doing simple computation, this lemma is proved easily.
\end{proof}
\begin{prop}\label{prop 4.7}
The homogeneous space $(({E_6}^C)^\sigma)_{E_1, F_1(e_k), k=2,\ldots,7}/S\!pin(3,C)$ is homeomorphic to the complex sphere $({S_-}^{\!\!C})^3${\rm :} $(({E_6}^C)^\sigma)_{E_1, F_1(e_k), k=2,\ldots,7}/S\!pin(3,C) \simeq ({S_-}^{\!\!C})^3$. 

In particular, the group $(({E_6}^C)^\sigma)_{E_1, F_1(e_k), k=2,\ldots,7}$ is connected.
\end{prop}
\begin{proof}
We define a $3$-dimensional complex sphere $({S_-}^{\!\!C})^3$ by 
\begin{eqnarray*}
  ({S_-}^{\!\!C})^3\!\!\!&=&\!\!\!\{ X \in ({V_-}^{\!\!C})^4 \,|\, (-E_1, X, X)=1  \}
\\[0mm]
         \!\!\!&=&\!\!\!\Bigl\{ \begin{pmatrix}0 & 0 & 0 \\
                                        0 & \xi_2 & x_1 \\
                                        0 &  \ov{x}_1 & \xi_3
                          \end{pmatrix}\,\Bigm |\,  -\xi_2\xi_3+\ov{x}_1 x_1=1, \xi_k \in C, x_1 \in \C^C \subset \mathfrak{C}^C \Bigr \}. 
\end{eqnarray*}
The group $(({E_6}^C)^\sigma)_{E_1, F_1(e_k), k=2,\ldots,7}$ acts on $({S_-}^{\!\!C})^3$. Indeed, for $\alpha \in (({E_6}^C)^\sigma)_{E_1, F_1(e_k), k=2,\ldots,7} \subset (({E_6}^C)^\sigma)_{E_1}$, it follows from \cite[Lemma 3.6.2]{realization G_2} that ${}^t \alpha ^{-1} \in(({E_6}^C)^\sigma)_{E_1, F_1(e_k), k=2,\ldots,7}$. Hence, for $X \in ({S_-}^{\!\!C})^3$ we have that
\begin{eqnarray*}
4 E_1 \times (E_1 \times \alpha X)\!\!\!&=&\!\!\! 4 {}^t\alpha^{-1} E_1 \times (\alpha E_1 \times \alpha X)
=4 {}^t\alpha^{-1} E_1 \times {}^t \alpha^{-1}(E_1 \times X)
\\ 
     \!\!\!&=&\!\!\! \alpha (4 E_1 \times (E_1 \times X))
\\
\!\!\!&=&\!\!\! \alpha X,
\\[1mm]
F_1(e_k) \times \alpha X \!\!\!&=&\!\!\! \alpha F_1(e_k) \times \alpha X
={}^t \alpha^{-1}(F_1(e_k) \times  X)
\\
      \!\!\!&=&\!\!\!  0,
\end{eqnarray*}
that is, $\alpha X \in ({V_-}^{\!\!C})^4$. Moreover, it is clear that $(-E_1, \alpha X, \alpha X)=1$. Thus we see $\alpha X \in ({S_-}^{\!\!C})^3$. 
We shall prove that this action is transitive. In order to
prove this, it is sufficient to show that any element $X \in ({S_-}^{\!\!C})^3$
can be transformed to $i(E_2 + E_3) \in ({S_-}^{\!\!C})^3$. 
Then we prepare some elements of $(({E_6}^C)^\sigma)_{E_1, F_1(e_k), k=2,\ldots,7}$. 

For $t \in \C^C \subset \mathfrak{C}^C$ such that $t\,\ov{t} \ne 0$, we define a $C$-linear transformation $\beta_1 (t)$ of $\mathfrak{J}^C$, $\beta_1 (t) X(\xi, x)=Y(\eta, y)$, by 
 \begin{eqnarray*}
&&\left\{\begin{array}{l}
     \eta_1 = \xi_1 
\vspace{0mm}\\
     \eta_2 = \displaystyle{\frac{\xi_2 - \xi_3}{2}} 
            + \displaystyle{\frac{\xi_2 + \xi_3}{2}}\cosh\nu
            + \displaystyle{\frac{(t,x_1)}{\nu}}\sinh\nu
\vspace{0mm}\\
     \eta_3 = -\displaystyle{\frac{\xi_2 - \xi_3}{2}}
            + \displaystyle{\frac{\xi_2 + \xi_3}{2}}\cosh\nu
            + \displaystyle{\frac{(t,x_1)}{\nu}}\sinh\nu,
\end{array}\right.\\
&&\left\{\begin{array}{l}
     y_1 = x_1 
            +\displaystyle{\frac{(\xi_2 + \xi_3)t}{2\nu}}\sinh\nu
            + \displaystyle{\frac{2(t, x_1)t}{\nu^2}}\sinh^2 \frac{\nu}{2}
\vspace{0mm}\\
     y_2 = x_2\cosh\dfrac{\nu}{2} + \displaystyle{\frac{\ov{x_3 t}}{\nu}}\sinh\frac{\nu}{2}
\vspace{0mm}\\
     y_3 = x_3\cosh\dfrac{\nu}{2} + \displaystyle{\frac{\ov{t x_2}}{\nu}}\sinh\frac{\nu}{2},
\end{array}\right.
\end{eqnarray*}
where $\nu \in C, \nu^2=t\,\ov{t}$, moreover define a $C$-linear transformation $\alpha_{23} (c)$ of $\mathfrak{J}^C$ by 
$$
\alpha_{23}(c)\begin{pmatrix}
                               \xi_1 & x_3 & \ov{x}_2 \\
                               \ov{x}_3 & \xi_2 & x_1 \\
                               x_2  & \ov{x}_1 & \xi_3
                      \end{pmatrix}=\begin{pmatrix}
                               \xi_1 & e^{c/2}x_3 & e^{-c/2}\ov{x}_2 \\
                               e^{c/2}\ov{x}_3 & e^c \xi_2 & x_1 \\
                               e^{-c/2}x_2  & \ov{x}_1 & e^{-c}\xi_3
                      \end{pmatrix},
$$
where $c \in C$.
Then since we can express $\beta_1 (t)={\exp F_1(t)}^\sim$ and $\alpha_{23}(c)=\exp c(E_2-E_3)^\sim$ for ${F_1(t)}^\sim,  c(E_2-E_3)^\sim \in {(({\mathfrak{e}_6}^C)^\sigma)}_{E_1, F_1(e_k), k=2,\ldots,7}$  (Lemma 
 \ref{lem 4.6}), we see that $\beta_1 (t), \alpha_{23}(c) \in ({(({E_6}^C)^\sigma)}_{E_1, F_1(e_k), k=2,\ldots,7})_0$.
\vspace{1mm}

Now, let $X=\begin{pmatrix}0 & 0 & 0 \\
                                        0 & \xi_2 & x_1 \\
                                        0 &  \ov{x}_1 & \xi_3
                          \end{pmatrix} \in ({S_-}^{\!\!C})^3$. Operate  some  $\alpha_0 \in ({(({E_6}^C)^\sigma)}_{E_1, F_1(e_k), k=2,\ldots,7})_0$ on $X$, and so $X$ can be transformed to $\begin{pmatrix}0 & 0 & 0 \\
                                        0 & \xi & x'_1 \\
                                        0 &  \ov{x}'_1 & -\xi
                          \end{pmatrix} \in ({S_-}^{\!\!C})^2$, that is, 
$$
\alpha_0 X=\begin{pmatrix}0 & 0 & 0 \\
                                        0 & \xi & x'_1 \\
                                        0 &  \ov{x}'_1 & -\xi
                          \end{pmatrix}\in ({S_-}^{\!\!C})^2.
$$
 Indeed, we have the following.
 
Case (i) where \vspace{1mm}$x_1 \ov{x}_1 \ne 0$.   

We choose some $t_0=i ({\pi}/{2})({e_1 x_1}/({x_1  \ov{x}_1})^{1/2}) \in \C^C$. Then since it is easy to verify that 
\begin{eqnarray*}
(t_0, x_1)\!\!\!&=&\!\!\! (i \Bigl(\dfrac{\pi}{2}\Bigr)\dfrac{e_1 x_1}{\sqrt{x_1  \ov{x}_1}}, x_1)=i \Bigl(\dfrac{\pi}{2}\Bigr)\dfrac{1}{\sqrt{x_1  \ov{x}_1}}(e_1x_1, x_1)=i \Bigl(\dfrac{\pi}{2}\Bigr)\dfrac{1}{\sqrt{x_1  \ov{x}_1}}(e_1, 1)(x_1, x_1)=0,
\\[1mm]
{\nu}^2\!\!\!&=&\!\!\!t_0\,\ov{t}_0=i\Bigl(\dfrac{\pi}{2}\Bigr)\dfrac{e_1 x_1}{\sqrt{x_1  \ov{x}_1}}
\,\ov{i \Bigl(\dfrac{\pi}{2}\Bigr)\dfrac{e_1 x_1}{\sqrt{x_1  \ov{x}_1}}}=
-\Bigl(\dfrac{\pi}{2}\Bigr)^2
\dfrac{(e_1 x_1)(\ov{e_1 x_1})}{{\sqrt{x_1  \ov{x}_1}}^2}=-\Bigl(\dfrac{\pi}{2}\Bigr)^2 \dfrac{x_1  \ov{x}_1}{x_1 \ov{x}_1}=-\Bigl(\dfrac{\pi}{2}\Bigr)^2,
\end{eqnarray*}                     
operate $\alpha_0:=\beta_1(t_0)$ on $X$, and so we easily see that the $\eta_2$-part and the $\eta_3$-part of $\alpha_0 X$ are 
 $({\xi_2-\xi_3})/{2}$ and $-({\xi_2-\xi_3})/{2}$, respectively. 
 Hence we can confirm the form of $\alpha_0 X$ as above.
 \vspace{1mm}
 
 Case (ii) where $x_1  \ov{x}_1 = 0$. 
 
 From the condition of $({S_-}^{\!\!C})^3$, we have $\xi_2 \xi_3=1$. Then set $\xi_2=e^{r_2+i\theta_2}, r_2, \theta_2 \in \R$. Operate $\alpha_{23}(-r_2-i\theta_2)$ on X, and so we easily see that the $\eta_2$-part and the $\eta_3$-part of $\alpha_{23}(-r_2-i\theta_2) X$ are equal to $1$, that is, 
$$
\alpha_{23}(-r_2-i\theta_2)X=\begin{pmatrix}0 & 0 & 0 \\
                                        0 & 1 & x_1 \\
                                        0 &   \ov{x}_1 & 1
                          \end{pmatrix}=:X_1.
$$
 
\noindent Moreover, operate $\alpha_{23}(i{\pi}/{2})$ on $X_1$, then we have that  
$$
\alpha_{23}\Bigl(i\dfrac{\pi}{2}\Bigr) X_1=\begin{pmatrix}0 & 0 & 0 \\
                                        0 & i & x_1 \\
                                        0 &   \ov{x}_1 & -i
                          \end{pmatrix} \in ({S_-}^{\!\!C})^2.
$$ Hence this case is reduced to Case (i).

 Since $S\!pin(3,C)\,( \cong ({F_4}^C)_{E_1, F_1(e_k), k=2,\ldots,7} \subset (({E_6}^C)^\sigma)_{E_1, F_1(e_k), k=2,\ldots,7})$ acts transitively on $({S_-}^{\!\!C})^2$ (Proposition \ref{prop 4.4}), there exists $\alpha \in S\!pin(3,C)$ such that  
$$
\alpha X= E_2 -E_3.
$$
Again, operate $\alpha_{23}(i{\pi}/{2})$ on $E_2 -E_3$, then we have that
$$
     \alpha_{23}\Bigl(i\dfrac{\pi}{2}\Bigr)(E_2-E_3)=i(E_2+E_3).   
$$
This shows the transitivity of this action to $({S_-}^{\!\!C})^3$ by 
the group $(({E_6}^C)^\sigma)_{E_1, F_1(e_k), k=2,\ldots,7}$. The isotropy subgroup of the group $(({E_6}^C)^\sigma)_{E_1, F_1(e_k), k=2,\ldots,7}$ at $i(E_2 +E_3)$ is 
$S\!pin(3,C) \cong ({F_4}^C)_{E_1, F_1(e_k), k=2,\ldots,7}=(({E_6}^C)^\sigma)_{E_1, E_2+E_3, F_1(e_k), k=2,\ldots,7}$ (Theorem \ref{thm 4.5}, Section 2).
Thus we have the required homeomorphism 
$$
(({E_6}^C)^\sigma)_{E_1, F_1(e_k), k=2,\ldots,7}/S\!pin(3,C) \simeq ({S_-}^{\!\!C})^3.
$$

Therefore we see that the group $(({E_6}^C)^\sigma)_{E_1, F_1(e_k), k=2,\ldots,7}$ is connected.
\end{proof}

\begin{thm}\label{thm 4.8}
The group $(({E_6}^C)^\sigma)_{E_1, F_1(e_k), k=2,\ldots,7}$ is isomorphic to $S\!pin(4, C)${\rm :}\\$(({E_6}^C)^\sigma)_{E_1, F_1(e_k), k=2,\ldots,7} \cong S\!pin(4, C)$.
\end{thm}
\begin{proof}
Let $O(4,C)=O(({V_-}^{\!\!C})^4)=\{\beta \in \Iso_{C} (({V_-}^{\!\!C})^4)\,|\,(E_1,\beta X, \beta Y)=(E_1, X, Y)\}$. We  consider the restriction $\beta=\alpha \bigm|_{({V_-}^{\!\!C})^4}$ of $\alpha \in (({E_6}^C)^\sigma)_{E_1, F_1(e_k), k=2,\ldots,7}$ to $({V_-}^{\!\!C})^4$, then we have $\beta \in O(4,C)$. Hence we can define a homomorphism $p: (({E_6}^C)^\sigma)_{E_1, F_1(e_k), k=2,\ldots,7} \to O(4, C)=O(({V_-}^{\!\!C})^4)$ by
$$
   p(\alpha)=\alpha \bigm|_{({V_-}^{\!\!C})^4}.
$$

Since the mapping $p$ is continuous and the group $(({E_6}^C)^\sigma)_{E_1, F_1(e_k), k=2,\ldots,7}$ is connected (Proposition \ref{prop 4.7}), the mapping $p$ induces  a homomorphism $p :(({E_6}^C)^\sigma)_{E_1, F_1(e_k), k=2,\ldots,7} \!\to S\!O(4, C)=S\!O(({V_-}^{\!\!C})^4)$. 
It is not difficult to obtain that  $\Ker\,p=\{1, \sigma \} \cong \Z_2$. Indeed, Let $\alpha \in \Ker\,p $. For $E_2+E_3, E_2-E_3 \in ({V_-}^{\!\!C})^4$, since $\alpha(E_2+E_3)=E_2+E_3$ and  $\alpha(E_2-E_3)=E_2-E_3$, that is, $\alpha E_2=E_2, \alpha E_3=E_3$, we have that $\alpha \in (({E_6}^C)^\sigma)_{E_1,E_2, E_3, F_1(e_k), k=2,\ldots,7} \cong ({F_4}^C)_{E_1,E_2, E_3, F_1(e_k), k=2,\ldots,7} \cong U(1, \C^C)$ (Theorem \ref{thm 3.3}). Hence there exists $\theta \in U(1, \C^C)$ such that $\alpha=\phi(\theta)$, where $\phi$ is defined in Theorem 3.3, then it follows from $\alpha F_1 (1)=F_1 (1), F_1(1) \in ({V_-}^{\!\!C})^4$ that we have $(\ov{\theta})^2=1$, that is, $\theta=1$ or $\theta=-1$. Thus we have that
$$
    \alpha=\phi(1)=1 \quad {\text{or}}\quad \alpha=\phi(-1)=\sigma,
$$
that is, $\Ker\,p \subset \{ 1, \sigma \}$ and vice versa. 
Hence we obtain that $\Ker\, p=\{1, \sigma  \}$. Finally, we shall show that $p$ is surjection. Since $S\!O(4,C)$ is connected, $\Ker\,p$ is discrete and $\dim_{C}((({\mathfrak{e}_6}^C)^\sigma)_{E_1, F_1(e_k), k=2,\ldots,7}) =6=\dim_{{C}}(\mathfrak{so}(4, C))$(Lemma \ref{lem 4.6}), $p$ is surjection. Thus we have that 
$$
(({E_6}^C)^\sigma)_{E_1, F_1(e_k), k=2,\ldots,7}/\Z_2 \cong S\!O(4, C).
$$

Therefore the group $(({E_6}^C)^\sigma)_{E_1, F_1(e_k), k=2,\ldots,7}$ is isomorphic to $S\!pin(4, C)$ as the universal double covering group of $S\!O(4, C)$, that is, $(({E_6}^C)^\sigma)_{E_1, F_1(e_k), k=2,\ldots,7} \cong S\!pin(4,C)$.
\end{proof}


We define a group $(({E_7}^C)^{\kappa, \mu})_{\ti{E}_1, \ti{E}_{-1},\dot{F}_1(e_k), k=2,\ldots,7}$ by
$$
(({E_7}^C)^{\kappa, \mu})_{\ti{E}_1, \ti{E}_{-1},\dot{F}_1(e_k), k=2,\ldots,7}=\left\{\alpha \in {E_7}^C \,\left|\,\begin{array}{l}\kappa\alpha=\alpha\kappa, \mu\alpha=\alpha\mu,\\
\alpha \ti{E}_1=\ti{E}_1, \alpha \ti{E}_{-1}=\ti{E}_{-1}\\
\alpha \dot{F}_1(e_k)=\dot{F}_1(e_k), k=2, \ldots, 7
\end{array} \right. \right \}. 
$$
Then we have the following proposition.

\begin{prop}\label{prop 4.9}
The group  $(({E_7}^C)^{\kappa, \mu})_{\ti{E}_1, \ti{E}_{-1},\dot{F}_1(e_k), k=2,\ldots,7}$ is equal to the group \\
$(({E_6}^C)^\sigma)_{E_1, F_1(e_k), k=2,\ldots,7}${\rm:}$(({E_7}^C)^{\kappa, \mu})_{\ti{E}_1, \ti{E}_{-1},\dot{F}_1(e_k), k=2,\ldots,7} = (({E_6}^C)^\sigma)_{E_1, F_1(e_k), k=2,\ldots,7}\cong S\!pin(4,C)$.
\end{prop}
\begin{proof}
Let $\alpha \in (({E_7}^C)^{\kappa, \mu})_{\ti{E}_1, \ti{E}_{-1},F_1(e_k), k=2,\ldots,7}$. From $\alpha \ti{E}_1=\ti{E}_1$ and $\alpha \ti{E}_{-1}=\ti{E}_{-1}$, we have $\alpha \underset{\dot{}}{1}=\underset{\dot{}}{1}$, and so as in the proof of Proposition \ref{prop 4.2}, \vspace{-1mm}we have $\alpha \dot{1}=\dot{1}$. Thus we see $\alpha \in ({E_7}^C)_{ \dot{1}, \underset{\dot{}}{1}}={E_6}^C$. Moreover, since we can confirm $\alpha \dot{E}_1=\dot{E}_1$ from the condition above \vspace{-1mm}, we have $\alpha \in ({E_6}^C)_{E_1}$, and from $\kappa\alpha=\alpha\kappa$, together with $-\sigma=\exp(\pi i\kappa)$, we have $(-\sigma)\alpha=\alpha(-\sigma)$, that is, $\sigma\alpha=\alpha\sigma$. Thus we have $\alpha \in (({E_6}^C)^\sigma)_{E_1, F_1(e_k), k=2,\ldots,7}$.

Conversely, let $\beta \in (({E_6}^C)^\sigma)_{E_1, F_1(e_k), k=2,\ldots,7}$. It is clear that $\beta \ti{E}_1=\ti{E}_1$ and $\beta \ti{E}_{-1}=\ti{E}_{-1}$. For $C$-linear transformation $\kappa_1$ of $\mathfrak{J}^C$: $\kappa_1 X=(E_1, X)E_1-4E_1 \times (E_1 \times X)$, we have $\kappa_1 \beta=\beta\kappa_1$. Indeed, note that suppose $\beta E_1=E_1$, we have ${}^t \beta^{-1} E_1=E_1$ (see \cite[Lemma 3.6.2]{realization G_2}).
\begin{eqnarray*}
\kappa_1 \beta X \!\!\!&=&\!\!\!(E_1, \beta X)E_1 - 4E_1 \times (E_1 \times \beta X)
\\[0.5mm]
\!\!\!&=&\!\!\!({}^t \beta^{-1}E_1, \beta E_1)\beta E_1-4{}^t \beta^{-1}E_1 \times (\beta E_1 \times \beta X)
\\[0.5mm]
\!\!\!&=&\!\!\!({}^t \beta {}^t \beta^{-1}E_1, E_1)\beta E_1 -4{}^t \beta^{-1}E_1 \times {}^t \beta^{-1}(E_1 \times X)
\\[0.5mm]
\!\!\!&=&\!\!\!(E_1, \beta X)\beta E_1 - 4\beta (E_1 \times (E_1 \times \beta X))
\\[0.5mm]
\!\!\!&=&\!\!\!\beta((E_1, X)E_1 - 4E_1 \times (E_1 \times  X))
\\[0.5mm]
\!\!\!&=&\!\!\!\beta \kappa_1 X,
\end{eqnarray*}
that is, $\kappa_1 \beta=\beta \kappa_1$. Similarly, we can show $ \kappa_1 {}^t \beta^{-1}={}^t \beta^{-1} \kappa_1$. Hence we have that 
\begin{eqnarray*}
\kappa \beta (X, Y, \xi, \eta)\!\!\!&=&\!\!\!\kappa (\beta X, {}^t \beta^{-1}, \xi, \eta)
\\[0.5mm]
\!\!\!&=&\!\!\!(-\kappa_1 \beta X, \kappa_1 {}^t \beta^{-1}Y, -\xi, \eta)
\\[0.5mm]
\!\!\!&=&\!\!\!(-\beta \kappa_1 X, {}^t \beta^{-1} \kappa_1 Y, -\xi, \eta)
\\[0.5mm]
\!\!\!&=&\!\!\!\beta(-\kappa_1 X, \kappa_1 Y, -\xi, \eta)
\\[0.5mm]
\!\!\!&=&\!\!\!\beta \kappa (X, Y, \xi, \eta),
\end{eqnarray*}
that is, $\kappa\beta=\beta\kappa$. Additionally, we can show that $\mu\beta=\beta\mu$. Indeed, for $(X, Y, \xi, \eta) \in \mathfrak{P}^C$, we do simple computation as follows:
\begin{eqnarray*}
\mu\beta(X, Y, \xi, \eta)\!\!\!&=&\!\!\!\mu(\beta X, {}^t \beta^{-1} Y, \xi, \eta)
\\[0.5mm]
\!\!\!&=&\!\!\!\varPhi(0, E_1, E_1, 0)(\beta X, {}^t \beta^{-1} Y, \xi, \eta)
\\[0.5mm]
\!\!\!&=&\!\!\!(2E_1 \times {}^t \beta^{-1} Y+\eta E_1, 2E_1 \times \beta X+\xi E_1, (E_1,{}^t \beta^{-1} Y), (E_1, \beta X) )
\\[0.5mm]
\!\!\!&=&\!\!\!(2{}^t \beta^{-1}E_1 \times {}^t \beta^{-1} Y\!+\!\eta \beta E_1, 2\beta E_1 \times \beta X\!+\!\xi {}^t \beta^{-1}E_1, (\beta E_1,{}^t \beta^{-1} Y), ({}^t \beta^{-1}E_1, \beta X) )
\\[0.5mm]
\!\!\!&=&\!\!\!(2\beta(E_1 \times  Y)\!+\!\eta \beta E_1, 2{}^t \beta^{-1}( E_1 \times  X)\!+\!\xi {}^t \beta^{-1}E_1, (\beta^{-1}\beta E_1,Y), ({}^t \beta^{-1} \beta^{-1}E_1, X) )
\\[0.5mm]
\!\!\!&=&\!\!\!(\beta(2E_1 \times  Y+\eta E_1), {}^t \beta^{-1}(2 E_1 \times  X+\xi E_1), ( E_1,Y), (E_1, X) )
\end{eqnarray*}
\begin{eqnarray*}
\!\!\!&=&\!\!\!\beta (2E_1 \times  Y+\eta E_1, 2 E_1 \times  X+\xi E_1, ( E_1,Y), (E_1, X) )
\\[0.5mm]
\!\!\!&=&\!\!\!\beta \varPhi(0, E_1, E_1, 0)(X, Y, \xi, \eta)
\\[0.5mm]
\!\!\!&=&\!\!\!\beta\mu(X,Y,\xi,\eta),
\end{eqnarray*}
that is, $\mu\beta=\beta\mu$. Hence we have $\beta \in (({E_7}^C)^{\kappa, \mu})_{\ti{E}_1, \ti{E}_{-1},\dot{F}_1(e_k), k=2,\ldots,7}$. This proof is completed.
\end{proof}

Continuously, we shall construct one more $S\!pin(5,C)$ in ${E_7}^C$.

We define a group $(({E_7}^C)^{\kappa, \mu})_{\tilde{E}_1,\dot{F}_1(e_k), k=2,\ldots,7}$ by
$$
(({E_7}^C)^{\kappa, \mu})_{\ti{E}_1, \dot{F}_1(e_k), k=2,\ldots,7}=\left\{\alpha \in {E_7}^C \,\left|\,\begin{array}{l}\kappa\alpha=\alpha\kappa, \mu\alpha=\alpha\mu,\\
\alpha \ti{E}_1=\ti{E}_1,\\
\alpha \dot{F}_1(e_k)=\dot{F}_1(e_k), k=2, \ldots, 7
\end{array} \right. \right \},
$$
moreover define a $5$-dimensional $C$-vector subspace $({V_-}^{\!\!C})^5$ of $\mathfrak{P}^C$ by
\begin{eqnarray*}
({V_-}^{\!\!C})^5\!\!\!&=&\!\!\!\left\{P \in \mathfrak{P}^C\,\left |\,\begin{array}{l}\kappa P=P, P \times \ti{E}_1=0,\\
                P \times \dot{F}_1(e_k)=0,k=2,\ldots,7
\end{array}\right. \right    \}
\\[1mm]
\!\!\!&=&\!\!\!\{(X, -\eta E_1, 0, \eta)\,|\, 4E_1 \times (E_1 \times X)=X, X \times F_1(e_k)=0,k=2, \ldots, 7, \eta \in C \}
\\[1mm]
\!\!\!&=&\!\!\!\Biggl\{(\begin{pmatrix} 0 & 0      & 0 \\
                                       0 & \xi_2  & x_1 \\
                                       0 & \ov{x}_1 & \xi_3 
                       \end{pmatrix},
                   \begin{pmatrix} -\eta & 0  & 0 \\
                                       0 & 0  & 0 \\
                                       0 & 0  & 0 
                       \end{pmatrix},
              0, \eta ) \,\Biggm |\,  x \in \C^C,
                                         \xi_2, \xi_3 ,\eta \in C
                                     \Biggr \}
\end{eqnarray*}
with the norm $(P, P)_\mu = ({1}/{2}) \{\mu P, P \}=-\xi_2 \xi_3+x_1\ov{x}_1-\eta^2$, here the alternative inner product $\{P, Q \}$ is defined as follows: $\{P, Q \}=(X, W)-(Z,Y)+\xi\omega-\zeta\eta$ for $P=(X,Y,\xi,\eta), Q=(Z,W,\zeta,\omega)$. The group $(({E_7}^C)^{\kappa, \mu})_{\ti{E}_1, F_1(e_k), k=2,\ldots,7}$ acts on $({V_-}^{\!\!C})^5$, obviously.

\begin{lem}\label{lem 4.10}
The Lie algebra $(({\mathfrak{e}_7}^C)^{\kappa, \mu})_{\ti{E}_1, \dot{F}_1(e_k), k=2,\ldots,7}$ of the group $(({E_7}^C)^{\kappa, \mu})_{\ti{E}_1, \dot{F}_1(e_k), k=2,\ldots,7}$ is given by 
\begin{eqnarray*}
(({\mathfrak{e}_7}^C)^{\kappa, \mu})_{\ti{E}_1, \dot{F}_1(e_k), k=2,\ldots,7}\!\!\!&=&\!\!\!\left\{\varPhi(\phi, A, B, \nu) \in {\mathfrak{e}_7}^C \,\left |\, \begin{array}{l}                               \kappa\varPhi=\varPhi\kappa,\mu\varPhi=\varPhi\mu, \\
\varPhi \ti{E}_1=0,\\
\varPhi \dot{F}_1(e_k)=0, k=2, \dots, 7
\end{array}\right. \right \}
\\[1mm]
\!\!\!&=&\!\!\! \left\{\varPhi(\phi, A, B, 0) 
\in {\mathfrak{e}_7}^C \,\left |\, \begin{array}{l} \phi 
\in ({\mathfrak{e}_6}^C)^\sigma, \\
\quad \phi E_1=\phi F_1(e_k)=0, 
\\[1mm]
A=\varepsilon_2 E_2 +\varepsilon_3 E_3 + F_1(a), 
\\
\quad \varepsilon_k \in C, a \in \C^C, 
\\[1mm]
B=-2E_1 \times A\\
\,\,\,\,\,\,=-\varepsilon_3 E_2 - \varepsilon_2 E_3 +F_1(a)
\end{array}\right. \right \}.
\end{eqnarray*}

In particular, $\dim_C((({\mathfrak{e}_7}^C)^{\kappa, \mu})_{\ti{E}_1, \dot{F}_1(e_k), k=2,\ldots,7})=6+(1+1+2)=10$.
\end{lem}
\begin{proof}
Suppose $\kappa\varPhi=\varPhi\kappa$ for $\varPhi \in {\mathfrak{e}_7}^C$, from $-\sigma=\exp(\pi i\kappa)$ we see that $(-\sigma)\varPhi=\varPhi(-\sigma)$, that is, $\sigma\varPhi=\varPhi\sigma$. Hence, we have $\phi \in ({\mathfrak{e}_6}^C)^\sigma$. Moreover, from $\mu\varPhi=\varPhi\mu$, the condition $\varPhi \ti{E}_1=0$ is equivalent to the condition $\varPhi(E_1, 0, 1, 0)=0$. Using these facts, by doing simple computation, we have the required result.   
\end{proof}

\begin{lem}\label{lem 4.11}
 For $0 \ne a \in C$, we define a mapping $\alpha_i(a) : 
\mathfrak{P}^C \to \mathfrak{P}^C$, $i=1, 2, 3$ by
$$
\alpha_i(a) = \begin{pmatrix}
 1 + (\cos|a| - 1)p_i & -2\tau a\dfrac{\sin|a|}{|a|}E_i & 0                      & a\dfrac{\sin|a|}{|a|}E_i \cr
       2a\dfrac{\sin|a|}{|a|}E_i & 1 + (\cos|a| - 1)p_i &       
               -\tau a\dfrac{\sin|a|}{|a|}E_i & 0 \cr
        0 & a\dfrac{\sin|a|}{|a|}E_i & \cos|a| & 0 \cr
          -\tau a\dfrac{\sin|a|}{|a|}E_i & 0 & 0 & \cos|a|
               \end{pmatrix},  
$$
then we have $\alpha_i(a) \in E_7 \subset {E_7}^C$, where $p_i : \mathfrak{J}^C \to \mathfrak{J}^C$ is the $C$-linear mapping defined by

$$
      p_i\begin{pmatrix} \xi_1 & x_3 & \ov{x}_2 \cr
                  \ov{x}_3 & \xi_2 & x_1 \cr
                  x_2 & \ov{x}_1 & \xi_3
          \end{pmatrix}
       = \begin{pmatrix} \xi_1 & \delta_{i3}x_3 & \delta_{i2}\ov{x}_2 \cr
                  \delta_{i3}\ov{x}_3 & \xi_2 & \delta_{i1}x_1 \cr
                  \delta_{i2}x_2 & \delta_{i1}\ov{x}_1 & \xi_3
         \end{pmatrix},
$$
where $\delta_{ij}$ is the Kronecker delta symble. The mappings $\alpha_1(a_1), \alpha_2(a_2), \alpha(a_3), (a_i \in C)$ are commutative for each other.
\end{lem}
\begin{proof}
For $ {\varPhi}_i(a) = {\varPhi}(0, aE_i, -\tau aE_i, 0 ) \in \mathfrak{e}_7$,
it follows from $\alpha_i(a) = \exp{\varPhi}_i(a)$ that $\alpha_i(a) \in E_7 \subset {E_7}^C$. The relation formula $[{\varPhi}_i(a_i), {\varPhi}_j(a_j)] = 0$ shows that $\alpha_i(a_i)$ and $\alpha_j(a_j)$ are commutative (As for the Lie algebra $\mathfrak{e}_7$ of the compact Lie group $E_7$, see \cite[Theorem 4.3.4]{Yokotaichiro} in detail).
\end{proof}

\begin{prop}\label{prop 4.12}
The homogeneous space $(({E_7}^C)^{\kappa, \mu})_{\ti{E}_1, \dot{F}_1(e_k), k=2,\ldots,7}/S\!pin(4,C)$ is homeomorphic to the complex sphere $({S_-}^{\!\!C})^4${\rm :} $(({E_7}^C)^{\kappa, \mu})_{\ti{E}_1, \dot{F}_1(e_k), k=2,\ldots,7}/S\!pin(4,C) \simeq ({S_-}^{\!\!C})^4$. 

In particular, the group $(({E_7}^C)^{\kappa, \mu})_{\ti{E}_1, \dot{F}_1(e_k), k=2,\ldots,7}$ is connected.
\end{prop}
\begin{proof}
We define a $4$-dimensional complex sphere $({S_-}^{\!\!C})^4$ by
\begin{eqnarray*}
({S_-}^{\!\!C})^4\!\!\!&=&\!\!\! \{P \in ({V_-}^{\!\!C})^5 \,|\, (P, P)_\mu =1  \}
\\[1mm]
\!\!\!&=&\!\!\!\Biggl\{(\begin{pmatrix} 0 & 0      & 0 \\
                                       0 & \xi_2  & x_1 \\
                                       0 & \ov{x}_1 & \xi_3 
                       \end{pmatrix},
                   \begin{pmatrix} -\eta & 0  & 0 \\
                                       0 & 0  & 0 \\
                                       0 & 0  & 0 
                       \end{pmatrix},
              0, \eta ) \,\Biggm |\, -\xi_2 \xi_3+x_1\ov{x}_1-\eta^2=1  \Biggr \}.
\end{eqnarray*}
The group $(({E_7}^C)^{\kappa, \mu})_{\ti{E}_1, \dot{F}_1(e_k), k=2,\ldots,7}$ acts on $({S_-}^{\!\!C})^4$. Indeed, for $\alpha \in (({E_7}^C)^{\kappa, \mu})_{\ti{E}_1, \dot{F}_1(e_k), k=2,\ldots,7}$ and $P \in ({S_-}^{\!\!C})^4$, 
from the following relational formulas:
\begin{eqnarray*}
\kappa \alpha P\!\!\!&=&\!\!\!\alpha \kappa P=\alpha P,
\\[1mm]
\alpha P \times \ti{E}_1\!\!\!&=&\!\!\!\alpha P \times \alpha \ti{E}_1=\alpha(P \times \ti{E}_1){}^t \alpha^{-1}=0,
\\[1mm]
\alpha P \times \dot{F}_1(e_k)\!\!\!&=&\!\!\!\alpha P \times \alpha \dot{F}_1(e_k)=\alpha(P \times \dot{F}_1(e_k)){}^t \alpha^{-1}=0,
\\[1mm]
(\alpha P, \alpha P)_\mu \!\!\!&=&\!\!\! \frac{1}{2} \{\mu \alpha P, \alpha P \}= \frac{1}{2} \{\alpha \mu P, \alpha P \}= \frac{1}{2} \{\mu P, P \}=1,
\end{eqnarray*}
we have $\alpha P \in ({S_-}^{\!\!C})^4$. We shall show that this action is transitive. In order to prove this, it is sufficient to show that any element $P \in ({S_-}^{\!\!C})^4$ can be transformed to $\ti{E}_{-1}=(0, -E_1, 0, 1)$.

Now, for a given
$$
  P=(\begin{pmatrix} 0 & 0      & 0 \\
                                       0 & \xi_2  & x \\
                                       0 & \ov{x} & \xi_3 
                       \end{pmatrix},
                   \begin{pmatrix} -\eta & 0  & 0 \\
                                       0 & 0  & 0 \\
                                       0 & 0  & 0 
                       \end{pmatrix},
              0, \eta ) \in({S_-}^{\!\!C})^4,
$$

\noindent we choose $a \in \R, 0 \leq a \leq {\pi}/{2}$ such that $\tan 2a ={2{\rm Re}(\eta)}/{{\rm Re}(\xi_2 + \xi_3)}$ (if ${\rm Re}(\xi_2 + \xi_3)=0$, let $a={\pi}/{4}$). Operate $\alpha_{23}(a):=\alpha_2(a)\alpha_3(a)=\exp(\varPhi(0, a(E_2+E_3),  -a(E_2+E_3), 0)) \in ((({E_7}^C)^{\kappa, \mu})_{\ti{E}_1, \dot{F}_1(e_k), k=2,\ldots,7})_0$ on $P$ (Lemmas \ref{lem 4.10}, \ref{lem 4.11}), then the part $({1}/{2}){\rm Re}(\xi_2 + \xi_3)\sin 2a -{\rm Re}(\eta) \cos 2a$ of $\eta$-term in $\alpha_{23}(a)P$ is equal to $0$, that is, 
$$
({1}/{2}){\rm Re}(\xi_2 + \xi_3)\sin 2a -{\rm Re}(\eta) \cos 2a=0.
$$
Again we choose $b \in \R, 0 \leq b \leq {\pi}/{2}$ such that \vspace{1mm} $\tan 2b ={2{\rm Im}(\eta)}/{{\rm Im}(\xi_2 + \xi_3)}$ (if ${\rm Im}(\xi_2 + \xi_3)=0$,  let $b={\pi}/{4}$), then $\eta$-term of $\alpha_{23}(b)\alpha_{23}(a)P$ is equal to $0$.

\noindent Hence we have that 
$$
       \alpha_{23}(b)\alpha_{23}(a)P=:P' \in ({S_-}^C)^3.  
$$

Since $S\!pin(4, C) (\cong (({E_6}^C)^{\sigma})_{{E}_1, {F}_1(e_k), k=2,\ldots,7}\subset (({E_7}^C)^{\kappa, \mu})_{\ti{E}_1, \dot{F}_1(e_k), k=2,\ldots,7})$ acts transitively on $({S_-}^{\!\!C})^3$, there exists $\beta \in S\!pin(4, C)$ such that 
$$
   \beta P'=(i(E_2 +E_3), 0, 0, 0)=:P''.
$$

\noindent Again, operate $\alpha_{23}(-{\pi}/{4})$ on $P''$, then we have that 
$$
 \alpha_{23}(-\dfrac{\pi}{4})P''=(0, -iE_1, 0, i)(=i\ti{E}_{-1}).    
$$

\noindent This shows the transitivity of this action to $({S_-}^{\!\!\!C})^4$ by 
the group $({E_7}^C)^{\kappa, \mu})_{\ti{E}_1, \dot{F}_1(e_k), k=2,\ldots,7}$. The isotropy subgroup of the group $({E_7}^C)^{\kappa, \mu})_{\ti{E}_1, \dot{F}_1(e_k), k=2,\ldots,7}$ at $i\ti{E}_{-1}$ is $S\!pin(4,C)$ (Theorem \ref{thm 4.8}, Proposition \ref{prop 4.9}). Thus we have the required homeomorphism 
$$
      (({E_7}^C)^{\kappa, \mu})_{\ti{E}_1, \dot{F}_1(e_k), k=2,\ldots,7}/S\!pin(4,C) \simeq ({S_-}^{\!\!\!C})^4.
$$

Therefore we see that the group $ (({E_7}^C)^{\kappa, \mu})_{\ti{E}_1, \dot{F}_1(e_k), k=2,\ldots,7}$ is connected.
 \end{proof}

\begin{thm}\label{thm 4.13}
The group $(({E_7}^C)^{\kappa, \mu})_{\ti{E}_1, \dot{F}_1(e_k), k=2,\ldots,7}$ is isomorphic to $S\!pin(5, C)${\rm :}\\$(({E_7}^C)^{\kappa, \mu})_{\ti{E}_1, \dot{F}_1(e_k), k=2,\ldots,7} \cong S\!pin(5,C)$.
\end{thm}
\begin{proof}
Let $O(5,C)=O(({V_-}^{\!\!C})^5)=\{\beta \in \Iso_{C} (({V_-}^{\!\!C})^5)\,|\,(\alpha P,\alpha P)_\mu=(P,P)_\mu\}$. We consider the restriction $\beta=\alpha \bigm|_{({V_-}^{\!\!C})^5}$ of $\alpha \in (({E_7}^C)^{\kappa, \mu})_{\ti{E}_1, \dot{F}_1(e_k), k=2,\ldots,7}$ to $({V_-}^{\!\!C})^5$, then we have $\beta \in O(5,C)$. Hence we can define a homomorphism $p: (({E_7}^C)^{\kappa, \mu})_{\ti{E}_1, \dot{F}_1(e_k), k=2,\ldots,7} \to O(5, C)=O(({V_-}^{\!\!C})^5)$ by
$$
   p(\alpha)=\alpha \bigm|_{({V_-}^{\!\!C})^5}.
$$
Since the mapping $p$ is continuous and the group $(({E_7}^C)^{\kappa, \mu})_{\ti{E}_1, \dot{F}_1(e_k), k=2,\ldots,7}$ is connected (Proposition \ref{prop 4.12}), the mapping $p$ induces  a homomorphism 
$$
p :(({E_7}^C)^{\kappa, \mu})_{\ti{E}_1, \dot{F}_1(e_k), k=2,\ldots,7} \to S\!O(5, C)=S\!O(({V_-}^{\!\!C})^5).
$$
 
It is not difficult to obtain that  $\Ker\,p=\{1, \sigma \} \cong \Z_2$. Indeed, let $\alpha \in \Ker\,p$. For $\ti{E}_{-1}\!\!=$ $(0,-E_1, 0,1) \in (V^{C})^5$, since $\alpha \ti{E}_{-1}=\ti{E}_{-1}$, together with $\alpha \ti{E}_{1}=\ti{E}_{1}$, we have that $\alpha \dot{E}_1=\dot{E}_1$ and $\alpha \underset{\dot{}}{1}=\underset{\dot{}}{1}$. \vspace{-1mm}Hence we have that $\alpha \!\in\! (({E_7}^C)^{\kappa, \mu})_{\dot{E}_1,\underset{\dot{}}{1}, \dot{F}_1(e_k), k=2,\ldots,7} \!\cong (({E_6}^C)^{\sigma})_{E_1, {F}_1(e_k), k=2,\ldots,7}$.
Moreover, for $E_2\dot{+}E_3, E_2\dot{-}E_3 \in ({V_-}^{\!\!C})^5$, since $\alpha(E_2\dot{+}E_3)=E_2\dot{+}E_3$ and  $\alpha(E_2\dot{-}E_3)=E_2\dot{-}E_3$, we have that $\alpha \!\in\! (({E_6}^C)^\sigma)_{E_1,E_2, E_3, F_1(e_k), k=2,\ldots,7} = ({F_4}^C)_{E_1,E_2, E_3, F_1(e_k), k=2,\ldots,7} \cong U(1, \C^C)$. Hence there exists $\theta \in U(1, \C^C)$ such that $\alpha=\phi(\theta)$, where $\phi$ is defined in Theorem \ref{thm 3.3}, and so since $\alpha F_1 (1)=F_1 (1), F_1(1) \in ({V_-}^{\!\!C})^5$, we have $(\ov{\theta})^2=1$, that is, $\theta=1$ or $\theta=-1$. Thus we have that
$$
    \alpha=\phi(1)=1 \quad {\text{or}}\quad \alpha=\phi(-1)=\sigma,
$$
that is, $\Ker\,p \subset \{ 1, \sigma \}$ and vice versa.
Hence we obtain that $\Ker\, p=\{1, \sigma  \}$. Finally, we shall show that $p$ is surjection. Since $S\!O(5,C)$ is connected, $\Ker\,p$ is discrete and $\dim_{C}((({\mathfrak{e}_7}^C)^{\kappa, \mu})_{\ti{E}_1, \dot{F}_1(e_k), k=2,\ldots,7})$ $ =10=\dim_{C}(\mathfrak{so}(5, C))$(Lemma \ref{lem 4.10}), $p$ is surjection. Thus we have that 
$$
(({E_7}^C)^{\kappa, \mu})_{\ti{E}_1, \dot{F}_1(e_k), k=2,\ldots,7}/\Z_2 \cong S\!O(5, C).
$$

Therefore the group $(({E_7}^C)^{\kappa, \mu})_{\ti{E}_1,\dot{F}_1(e_k), k=2,\ldots,7}$ is isomorphic to $S\!pin(5, C)$ as the universal double covering group of $S\!O(5, C)$, that is, $(({E_7}^C)^{\kappa, \mu})_{\ti{E}_1,\dot{F}_1(e_k), k=2,\ldots,7} \cong S\!pin(5,C)$.
\end{proof}

Continuously, we shall construct one more $S\!pin(6,C)$ in ${E_7}^C$.

We define a group $(({E_7}^C)^{\kappa, \mu})_{\dot{F}_1(e_k), k=2,\ldots,7}$ by
$$
(({E_7}^C)^{\kappa, \mu})_{\dot{F}_1(e_k), k=2,\ldots,7}=\left\{\alpha \in {E_7}^C \,\left|\,\begin{array}{l}\kappa\alpha=\alpha\kappa, \mu\alpha=\alpha\mu,\\
\alpha \dot{F}_1(e_k)=\dot{F}_1(e_k), k=2, \ldots, 7
\end{array} \right. \right \},
$$
moreover, define a $6$-dimensional $C$-vector subspace $({V_-}^C)^6$ of $\mathfrak{P}^C$ by
\begin{eqnarray*}
({V_-}^{\!\!C})^6\!\!\!&=&\!\!\!\left\{P \in \mathfrak{P}^C\,\left |\,\begin{array}{l}\kappa P=P,
\\
                P \times \dot{F}_1(e_k)=0,k=2,\ldots,7
\end{array}\right. \right    \}
\\[1mm]
\!\!\!&=&\!\!\!\Biggl\{(\begin{pmatrix} 0 & 0      & 0 \\
                                       0 & \xi_2  & x_1 \\
                                       0 & \ov{x}_1 & \xi_3 
                       \end{pmatrix},
                   \begin{pmatrix} \eta_1 & 0  & 0 \\
                                       0 & 0  & 0 \\
                                       0 & 0  & 0 
                       \end{pmatrix},
              0, \eta ) \,\Biggm |\,  x_1 \in \C^C,
                                \xi_2, \xi_3 ,\eta_1, \eta \in C
                                     \Biggr \}
\end{eqnarray*}
with the norm $(P, P)_\mu = ({1}/{2}) \{\mu P, P \}=-\xi_2 \xi_3+x_1\ov{x}_1+\eta_1 \eta$.

\begin{lem}\label{lem 4.14}
The Lie algebra $(({\mathfrak{e}_7}^C)^{\kappa, \mu})_{\dot{F}_1(e_k), k=2,\ldots,7}$ of the group $(({E_7}^C)^{\kappa, \mu})_{\dot{F}_1(e_k), k=2,\ldots,7}$ is given by 
\begin{eqnarray*}
(({\mathfrak{e}_7}^C)^{\kappa, \mu})_{\dot{F}_1(e_k), k=2,\ldots,7}\!\!\!&=&\!\!\!\!\left\{\varPhi(\phi, A, B, \nu) \in {\mathfrak{e}_7}^C \left |\begin{array}{l}                               \kappa\varPhi=\varPhi\kappa,\mu\varPhi=\varPhi\mu, \\
\varPhi \dot{F}_1(e_k)=0, k=2, \dots, 7
\end{array}\right. \right \}
\\[1mm]
\!\!\!&=&\!\!\! \!\left\{\varPhi(\phi, A, B, \nu) 
\in {\mathfrak{e}_7}^C \left |\! \begin{array}{l} \phi 
=\left(
\begin{array}{@{\,}cc|cc@{\,\,\,}}
\multicolumn{2}{c|}{\raisebox{-2.0ex}[-5pt][0pt]{\large $D_2$}}&& \\
&&\multicolumn{2}{c}{\raisebox{0.9ex}[0pt]{\large $0$}} \\
\hline
&&\multicolumn{2}{c}{\raisebox{-10pt}[0pt][0pt]{\large $0$}}\\
\multicolumn{2}{c|}{\raisebox{3pt}[0pt]{\large $0$}}&& \\
\end{array}
\right)
\begin{array}{l}
\!+ \tilde{A}_1(a)
\\ \vspace{1mm}
\!+(\tau_1 E_1\!+\!\tau_2 E_2\!+\!\tau E_3\!
\\ \vspace{1mm}
\qquad\qquad +\!F_1(t_1))^\sim
\end{array}
\vspace{1mm}\\
\qquad D_2 \in \mathfrak{so}(2, C), a \in \C^C, \tau_k \in C,
\\
\qquad \tau_1 +\tau_2 + \tau_3=0, t_1 \in \C^C,\\        
\!A=\varepsilon_2 E_2 +\varepsilon_3 E_3 + F_1(a),\, \varepsilon_k \in C, a \in \C^C, 
\\[1mm]
\!B=\upsilon_2 E_2 + \upsilon_3 E_3 +F_1(b),\, \upsilon_k \in C, b \in \C^C,
\\[1mm]
\nu=-({3}/{2})\tau_1
\end{array}\right. \!\!\!\right \}.
\end{eqnarray*}

In particular, $\dim_C((({\mathfrak{e}_7}^C)^{\kappa, \mu})_{\ti{E}_1, \dot{F}_1(e_k), k=2,\ldots,7})=(1+2+(2+2))+4+4=15$.
\end{lem}
\begin{proof}
From \cite[Section 4.6]{realization E_7}, we see that the explicit form of $({\mathfrak{e}_7}^C)^{\kappa, \mu}$ is given by
\begin{eqnarray*}
({\mathfrak{e}_7}^C)^{\kappa, \mu}\!\!\!&=&\!\!\!\{\varPhi(\phi, A, B, \nu) \in {\mathfrak{e}_7}^C \,|\,\kappa\varPhi=\varPhi\kappa, \mu\varPhi=\varPhi\mu  \}
\\
\!\!\!&=&\!\!\!\left\{\varPhi(\phi, A, B, \nu) \in {\mathfrak{e}_7}^C \,\left|\, \begin{array}{l}\phi \in ({\mathfrak{e}_6}^C)^\sigma,\\
A=\varepsilon_2 E_2 +\varepsilon_3 E_3 + F_1(a),\\
B=\upsilon_2 E_2 + \upsilon_3 E_3 +F_1(b),\\
\quad  \varepsilon_k, \upsilon_k \in C, a, b \in \C^C,\\
\nu=-({3}/{2})(\phi E_1, E_1)
                                  \end{array} 
\right. \right \}.
\end{eqnarray*}
Since the result of direct computation of $\varPhi \dot{F}_1(e_k)$ is as follows:
$$
  \varPhi(\phi, A, B, \nu)\dot{F}_1(e_k)=(\phi F_1(e_k)-\frac{\nu}{3}F_1(e_k), 2A \times F_1(e_k), 0, (B, F_1(e_k))), 
$$
for $\varPhi \dot{F}_1(e_k)=0$ we have that
$$
\left\{
\begin{array}{l}
\phi F_1(e_k)-\dfrac{\nu}{3}F_1(e_k)=0 \cdots (1)
\vspace{0.5mm}\\
2A \times F_1(e_k)=0 \cdots (2)
\vspace{0.5mm}\\
(B, F_1(e_k))=0 \cdots (3).
\end{array}
\right. 
$$
From the conditions (2) and (3), it is easy to verify that $x, y \in \C^C$. As for the condition (1), by doing direct computation, we obtain that 
\begin{eqnarray*}
&&\phi F_1(e_k)-\dfrac{\nu}{3}F_1(e_k)
\\[1mm]
\!\!\!&=&\!\!\!(\delta+\tilde{T})F_1(e_k)-\dfrac{\nu}{3}F_1(e_k) \,\,(\delta \in {\mathfrak{f}_4}^C, T \in (\mathfrak{J}^C)_\sigma, \tr(T)=0)
\\[1mm]
\!\!\!&=&\!\!\!\delta F_1(e_k)+\tilde{T}F_1(e_k)-\dfrac{\nu}{3}F_1(e_k)
\\[1mm]
\!\!\!&=&\!\!\!(D+\tilde{A}_1 (a_1))F_1(e_k)+\tilde{T}F_1(e_k)-\dfrac{\nu}{3}F_1(e_k)\,\,( D \in \mathfrak{so}(8, C), a_1 \in \mathfrak{C}^C)
\\[1mm]
\!\!\!&=&\!\!\!F_1 (De_k)\!+(a_1, e_k)(E_2 -E_3)\!+(\frac{1}{2}\tau_2 F_1(e_k)+\frac{1}{2}\tau_3 F_1(e_k)+(t_1, e_k)(E_2 +E_3))-\dfrac{\nu}{3}F_1(e_k)
\\[1mm]
\!\!\!&=&\!\!\!\{ (a_1, e_k)+(t_1, e_k) \}E_2 +\{ (a_1, e_k)-(t_1, e_k) \}E_3+F_1(De_k +\frac{1}{2}(\tau_2+\tau_3)e_k-\frac{\nu}{3}e_k),
\end{eqnarray*}
where $T=\tau_1 E_1+\tau_2 E_2 + \tau_3 E_3 +F_1(t_1), \tau_k \in C, t_1 \in \mathfrak{C}^C$.

\noindent Hence, from the condition (1), we see that 
$$
\left\{
\begin{array}{l}
(a_1, e_k)+(t_1, e_k)=(a_1, e_k)-(t_1, e_k)=0 \cdots (4)\\
De_k -\dfrac{1}{2}\tau_1 e_k-\dfrac{\nu}{3}e_k=0 \cdots (5) \,\,(\tau_1+\tau_2+\tau_3=0),
\end{array}
\right. 
$$
moreover, for the condition (5), together with $\nu=(-3/2)\tau_1$,
we have $D e_k=0, k=2,\ldots, 7$. Thus we have $D \in \mathfrak{so}(2, C) \subset \mathfrak{so}(8, C)$. From the condition (4), we have $a_1, t_1 \in \C^C$. Therefore we have the required the explicit form of the Lie algebra $(({\mathfrak{e}_7}^C)^{\kappa, \mu})_{\dot{F}_1(e_k), k=2,\ldots,7}$.
\end{proof}

\begin{lem}\label{lem 4.15}
For $\nu \in C$, we define a mapping $\beta (\nu): \mathfrak{P}^C \to \mathfrak{P}^C$ by
\begin{eqnarray*}
\beta (\nu)(X, Y, \xi, \eta)\!\!\!&=&\!\!\!( \begin{pmatrix}
                          e^{2\nu}\xi_1 & e^\nu x_3 & e^\nu\ov{x}_2  \\
                          e^\nu \ov{x}_3 & \xi_2 & x_1 \\
                          e^\nu x_2 & \ov{x}_1 & \xi_3
                          \end{pmatrix},
                                            \begin{pmatrix}
                  e^{-2\nu}\eta_1 & e^{-\nu} y_3 & e^{-\nu}\ov{y}_2  \\
                  e^{-\nu} \ov{y}_3 & \eta_2 & y_1 \\
                  e^{-\nu} y_2 & \ov{y}_1 & \eta_3
                          \end{pmatrix}, e^{-2\nu}\xi, e^{-\nu}\eta).
\end{eqnarray*}
Then we have $\beta(\nu) \in ((({E_7}^C)^{\kappa, \mu})_{\dot{F}_1(e_k), k=2,\ldots,7})_0$.
\end{lem}
\begin{proof}
From Lemma \ref{lem 4.14}, for $\nu \in C$ we see that 
$\varPhi((2/3)\nu(2E_1-(E_2+E_3))^\sim, 0,0,-2\nu) $ $ \in 
(({\mathfrak{e}_7}^C)^{\kappa, \mu})_{\dot{F}_1(e_k), k=2,\ldots,7}$. 
Hence we have that 
$$
\beta(\nu)=\exp(\varPhi((2/3)\nu(2E_1-(E_2+E_3))^\sim, 0,0,-2\nu)) \in  ((({E_7}^C)^{\kappa, 
\mu})_{\dot{F}_1(e_k), k=2,\ldots,7})_0.
$$
\end{proof}

\begin{prop}\label{prop 4.16}
The homogeneous space $(({E_7}^C)^{\kappa, \mu})_{\dot{F}_1(e_k), k=2,\ldots,7}/S\!pin(5,C)$ is homeomorphic to the complex sphere $({S_-}^{\!\!C})^5${\rm :} $(({E_7}^C)^{\kappa, \mu})_{\dot{F}_1(e_k), k=2,\ldots,7}/S\!pin(5,C) \simeq ({S_-}^{\!\!C})^5$. 

In particular, the group $(({E_7}^C)^{\kappa, \mu})_{\dot{F}_1(e_k), k=2,\ldots,7}$ is connected.
\end{prop}
\begin{proof}
We define a $5$-dimensional complex sphere $({S_-}^{\!\!C})^5$ by
\begin{eqnarray*}
({S_-}^{\!\!C})^5\!\!\!&=&\!\!\! \{P \in ({V_-}^{\!\!C})^6 \,|\, (P, P)_\mu =1  \}
\\[1mm]
\!\!\!&=&\!\!\!\biggl\{(\begin{pmatrix} 0 & 0      & 0 \\
                                       0 & \xi_2  & x_1 \\
                                       0 & \ov{x}_1 & \xi_3 
                       \end{pmatrix},
                   \begin{pmatrix} \eta_1 & 0  & 0 \\
                                       0 & 0  & 0 \\
                                       0 & 0  & 0 
                       \end{pmatrix},
              0, \eta ) \,\biggm |\, -\xi_2 \xi_3+x_1\ov{x}_1+\eta_1 \eta=1  \biggr\}.
\end{eqnarray*}
As in the proof of Proposition \ref{prop 4.12}, it is easy to verify that the group $(({E_7}^C)^{\kappa, \mu})_{\dot{F}_1(e_k), k=2,\ldots,7}$ acts on $({S_-}^{\!\!C})^5$, and so we shall show that this action is transitive. In order to prove this, it is sufficient to show that any $P \in ({S_-}^{\!\!C})^5$ can be transformed to $\ti{E}_1 \in ({S_-}^{\!\!C})^5$. 

Now, for a given 
$$
 P=(\begin{pmatrix} 0 & 0      & 0 \\
                                       0 & \xi_2  & x \\
                                       0 & \ov{x} & \xi_3 
                       \end{pmatrix},
                   \begin{pmatrix} \eta_1 & 0  & 0 \\
                                       0 & 0  & 0 \\
                                       0 & 0  & 0 
                       \end{pmatrix}, 0, \eta ) \in (S^C)^5, 
$$
first we shall show that there exists $\alpha \in (({E_7}^C)^{\kappa, \mu})_{\dot{F}_1(e_k), k=2,\ldots,7}$ such that \vspace{1mm}$\alpha P \in ({S_-}^{\!\!C})^4$.
\vspace{1mm}

Case (i) where $\eta_1 \ne 0, \eta \ne 0$. 

We choose $\nu \in C$ such that $-e^{-2\nu} \eta_1=e^{2\nu}\eta$, and 
operate $\beta (\nu)$ of Lemma \ref{lem 4.15} on $P$, then we have $\beta (\nu) P \in ({S_-}^{\!\!C})^4$.
\vspace{1mm}

Case (ii) where $\eta_1=0, \eta \ne 0, \xi_2 \ne 0$. 

Operate $\alpha=\exp \varPhi(0, E_3, 0,0) \in ((({E_7}^C)^{\kappa, \mu})_{\dot{F}_1(e_k), k=2,\ldots,7})_0$ on $P$ (Lemma \ref{lem 4.14}), then we have that
$$
  \alpha P=(\xi_2 E_2+(\xi_3+\eta)E_3 +F_1(x), \xi_2 E_1, 0, \eta).
$$
Hence this case is reduced to Case (i).
\vspace{1mm}

Case (iii) where $\eta_1=0, \eta \ne 0, \xi_3 \ne 0$. 

As in Case (ii), operate $\alpha=\exp \varPhi(0, E_2, 0,0) \in ((({E_7}^C)^{\kappa, \mu})_{\dot{F}_1(e_k), k=2,\ldots,7})_0$ on $P$ (Lemma \ref{lem 4.14}), then we have that 
$$
  \alpha P=((\xi_2+\eta)E_2+\xi_3 E_3 +F_1(x), \xi_3 E_1, 0, \eta).
$$
Hence this case is also reduced to Case (i).
\vspace{1mm}

Case (iv) where $\eta_1=\xi_2=\xi_3=0, \eta \ne 0$. 

For some $t \in \R$, operate $\alpha=\exp\varPhi(0, tF_1(x), 0,0) \in ((({E_7}^C)^{\kappa, \mu})_{\dot{F}_1(e_k), k=2,\ldots,7})_0$ on $P=(F_1(x), 0, 0, \eta)$ (Lemma \ref{lem 4.14}), then we have that
\begin{eqnarray*} 
\alpha P\!\!\!&=&\!\!\!((1+t\eta)F_1(x), -(2t+t^2\eta)(x,x)E_1, 0, \eta)\,\,((x,x)=1)
\\
\!\!\!&=&\!\!\!((1+t\eta)F_1(x), -(2t+t^2\eta)E_1, 0, \eta).
\end{eqnarray*}
Hence this case is also reduced to Case (i) for some $t \in \R$.
\vspace{1mm}

Case (v) where $\eta_1 \ne 0, \eta = 0, \xi_2 \ne 0$. 

Operate $\alpha=\exp \varPhi(0,0, E_2,0) \in ((({E_7}^C)^{\kappa, \mu})_{\dot{F}_1(e_k), k=2,\ldots,7})_0$ on $P$ (Lemma \ref{lem 4.14}), then we have that
$$
  \alpha P=(\xi_2 E_2+\xi_3 E_3 +F_1(x), \eta_1 E_1, 0, \xi_2).
$$
Hence this case is also reduced to Case (i).
\vspace{1mm}

Case (vi) where  $\eta_1 \ne 0, \eta = 0, \xi_3 \ne 0$. 

As in Case (v), operate $\alpha=\exp \varPhi(0,0, E_3,0) \in ((({E_7}^C)^{\kappa, \mu})_{\dot{F}_1(e_k), k=2,\ldots,7})_0$ on $P$ (Lemma \ref{lem 4.14}), then we have that
$$
  \alpha P=(\xi_2 E_2+\xi_3 E_3 +F_1(x), \eta_1 E_1, 0, \xi_3).
$$
Hence this case is also reduced to Case (i).
\vspace{1mm}

Case (vii) where $\eta_1 \ne 0, \eta = 0, \xi_2=\xi_3 =0$. 

For some $t \in \R$, operate $\alpha=\exp\varPhi(0,0, tF_1(x),0) \in ((({E_7}^C)^{\kappa, \mu})_{\dot{F}_1(e_k), k=2,\ldots,7})_0$ on $P=(F_1(x), \eta_1 E_1, 0, 0)$ (Lemma \ref{lem 4.14}), then we have that
\begin{eqnarray*} 
\alpha P\!\!\!&=&\!\!\!((1-t\eta_1)F_1(x), \eta_1 E_1, 0, (2t-t^2 \eta_1(x,x)))\,((x,x)=1)
\\
\!\!\!&=&\!\!\!((1-t\eta_1)F_1(x), \eta_1 E_1, 0, (2t-t^2 \eta_1)).
\end{eqnarray*}
Hence this case is also reduced to Case (i) for some $t \in \R$.
\vspace{1mm}

Case (viii) where  $\eta_1=\eta=0$. 

\noindent Then we see that $P \in ({S_-}^{\!\!C})^3 \subset ({S_-}^{\!\!C})^4$.

Since $S\!pin(5,C)\,(\cong (({E_7}^C)^{\kappa, \mu})_{\ti{E}_1,\underset{\dot{}}{1}, \dot{F}_1(e_k), k=2,\ldots,7} \subset (({E_7}^C)^{\kappa, \mu})_{\dot{F}_1(e_k), k=2,\ldots,7})$ acts transitively on $({S_-}^{\!\!C})^4$ (Proposition \ref{prop 4.12}), there exists $\beta \in S\!pin(5,C)$ such that 
$$
   \beta(\alpha P)=(0, -iE_1, 0, i)(=i\ti{E}_{-1}), 
$$
Again, operate $\beta(-i{\pi}/{4})$ of Lemma \ref{lem 4.15} on $\beta(\alpha P)$, then we have that
$$
     \beta(-i\frac{\pi}{4})(\beta (\alpha P))=(0, E_1, 0, 1)(=\ti{E}_1).
$$
This shows the transitivity of this action to $({S_-}^{\!\!C})^5$ by the group $(({E_7}^C)^{\kappa, \mu})_{\dot{F}_1(e_k), k=2,\ldots,7}$. The isotropy subgroup of the group $(({E_7}^C)^{\kappa, \mu})_{\dot{F}_1(e_k), k=2,\ldots,7}$ at $\ti{E}_1$ is $S\!pin(5, C)$ (Theorem \ref{thm 4.13}). 

Thus we have the required homeomorphism 
$$
(({E_7}^C)^{\kappa, \mu})_{\dot{F}_1(e_k), k=2,\ldots,7}/S\!pin(5,C) \simeq ({S_-}^{\!\!C})^5.
$$

Therefore we see that the group $(({E_7}^C)^{\kappa, \mu})_{\dot{F}_1(e_k), k=2,\ldots,7}$ is connected.
\end{proof}

\begin{thm}\label{thm 4.17}
The group $(({E_7}^C)^{\kappa, \mu})_{\dot{F}_1(e_k), k=2,\ldots,7}$ is isomorphic to $S\!pin(6, C)${\rm :}\\$(({E_7}^C)^{\kappa, \mu})_{\dot{F}_1(e_k), k=2,\ldots,7} \cong S\!pin(6,C)$.
\end{thm}
\begin{proof}
Let $O(6,C)=O(({V_-}^{\!\!C})^6)=\{\beta \in \Iso_{C} (({V_-}^{\!\!C})^6)\,|\,(\alpha P,\alpha P)_\mu=(P,P)_\mu \}$. We consider the restriction $\beta=\alpha \bigm|_{({V_-}^C)^6}$ of $\alpha \in (({E_7}^C)^{\kappa, \mu})_{\dot{F}_1(e_k), k=2,\ldots,7}$ to $({V_-}^{\!\!C})^6$, then we have $\beta \in O(6,C)$. Hence we can define a homomorphism $p: (({E_7}^C)^{\kappa, \mu})_{\dot{F}_1(e_k), k=2,\ldots,7} \to O(6, C)=O(({V_-}^{\!\!C})^6)$ by
$$
   p(\alpha)=\alpha \bigm|_{({V_-}^C)^6}.
$$
Since the mapping $p$ is continuous and the group $(({E_7}^C)^{\kappa, \mu})_{\dot{F}_1(e_k), k=2,\ldots,7}$ is connected (Proposition \ref{prop 4.16}), the mapping $p$ induces  a homomorphism 
$$
p :(({E_7}^C)^{\kappa, \mu})_{\dot{F}_1(e_k), k=2,\ldots,7} \to S\!O(6, C)=S\!O(({V_-}^{\!\!C})^6).
$$ 

It is not difficult to obtain that  $\Ker\,p=\{1, \sigma \} \cong \Z_2$. Indeed, let $\alpha \in \Ker\,p$. For $\ti{E}_{1}\!\!=(0,E_1, 0,1), \ti{E}_{-1}\!\!=(0,-E_1, 0,1) \in ({V_-}^{\!\!C})^6$, since $\alpha \ti{E}_{1}=\ti{E}_{1}$ and $\alpha \ti{E}_{-1}=\ti{E}_{-1}$, we have that $\alpha \dot{E}_1=\dot{E}_1$ and $\alpha \underset{\dot{}}{1}=\underset{\dot{}}{1}$. \vspace{-1mm}Hence we have that $\alpha \in (({E_7}^C)^{\kappa, \mu})_{\dot{E}_1,\underset{\dot{}}{1}, \dot{F}_1(e_k), k=2,\ldots,7} \cong (({E_6}^C)^{\sigma})_{E_1, {F}_1(e_k), k=2,\ldots,7}$.
Moreover, for $E_2\dot{+}E_3, E_2\dot{-}E_3 \in ({V_-}^{\!\!C})^6$, since $\alpha(E_2\dot{+}E_3)=E_2\dot{+}E_3$ and  $\alpha(E_2\dot{-}E_3)=E_2\dot{-}E_3$, we have that $\alpha \in (({E_6}^C)^\sigma)_{E_1,E_2, E_3, F_1(e_k), k=2,\ldots,7} \cong ({F_4}^C)_{E_1,E_2, E_3, F_1(e_k),} $ $ {}_{ k=2,\ldots,7} \cong U(1, \C^C)$. Hence there exists $\theta \in U(1, \C^C)$ such that $\alpha=\phi(\theta)$, where $\phi$ is defined in Theorem \ref{thm 3.3}, and so since $\alpha F_1 (1)=F_1 (1), F_1(1) \in ({V_-}^{\!\!C})^6$, we have $(\ov{\theta})^2=1$, that is, $\theta=1$ or $\theta=-1$. Thus since we see 
$$
\alpha=\phi(1)=1 \quad {\text{or}}\quad \alpha=\phi(-1)=\sigma,
$$
we have that
$\Ker\, p \subset \{1, \sigma  \}$ and vice versa. Hence we obtain that  $\Ker\, p=\{1, \sigma  \}$. Finally, we shall show that $p$ is surjection. Since $S\!O(6,C)$ is connected, $\Ker\,p$ is discrete and $\dim_{C}((({\mathfrak{e}_7}^C)^{\kappa, \mu})_{\dot{F}_1(e_k), k=2,\ldots,7})$ $ =15=\dim_{{C}}(\mathfrak{so}(6, C))$(Lemma \ref{lem 4.14}), $p$ is surjection. Thus we have that 
$$
(({E_7}^C)^{\kappa, \mu})_{ \dot{F}_1(e_k), k=2,\ldots,7}/\Z_2 \cong S\!O(6, C).
$$

Therefore the group $(({E_7}^C)^{\kappa, \mu})_{\dot{F}_1(e_k), k=2,\ldots,7}$ is isomorphic to $S\!pin(6, C)$ as the universal double covering group of $S\!O(6, C)$, that is, $(({E_7}^C)^{\kappa, \mu})_{\dot{F}_1(e_k), k=2,\ldots,7} \cong S\!pin(6,C)$.
\end{proof}
Here, as in previous section, we make a summary of the results as the low dimensional spinor groups which were constructed in this section.
\begin{eqnarray*}
&&(({E_7}^C)^{\kappa, \mu})_{
\dot{F}_1(e_k), k=2,\ldots,7} 
\cong S\!pin(6,C)
\\[0mm]
&&\hspace*{15mm} \cup
\\[-1.5mm]
&&(({E_7}^C)^{\kappa, \mu})_{\ti{E}_1,\dot{F}_1(e_k), k=2,\ldots,7}\cong 
S\!pin(5,C)
\\[0mm]
&&\hspace*{15mm} \cup
\\[-1.5mm]
&&
(({E_6}^C)^{\sigma})_{E_1,F_1(e_k), k=2,\ldots,7} \cong 
S\!pin(4,C)
\\[0mm]
&&\hspace*{15mm} \cup
\\[-1.5mm]
&&
({F_4}^C)_{E_1, F_1(e_k), k=2,\ldots,7} \cong 
S\!pin(3,C)
\\[0mm]
&&\hspace*{15mm} \cup
\\[-1.5mm]
&&
({F_4}^C)_{E_1,E_2, E_3, F_1(e_k), k=2,\ldots,7} \cong 
S\!pin(2,C) \cong U(1, \C^C)
\end{eqnarray*}
\indent Together with the results of previous section, we have had two sequences as for the low dimensional spinor groups. 
\vspace{2mm}

After this, by using two $S\!pin(6,C)$, we determine the structure of the groups $({({E_7}^C)}^{\kappa, \mu})^{\sigma'_{\!\!4}}, $ $({E_7}^C)^{\sigma'_{\!\!4}}$, and so we shall prove the connectedness of the group $({E_7}^C)^{\sigma'_{\!\!4}, \mathfrak{so}(6,C)}$.
\vspace{2mm}

First, we determine the structure of the group $({({E_7}^C)}^{\kappa, \mu})^{\sigma'_{\!\!4}}$.

\begin{lem}\label{lem 4.18}
The Lie algebra $({({\mathfrak{e}_7}^C)}^{\kappa, \mu})^{\sigma'_{\!\!4}}$ of the group $({({E_7}^C)}^{\kappa, \mu})^{\sigma'_{\!\!4}}$ is given by 
\begin{eqnarray*}
({({\mathfrak{e}_7}^C)}^{\kappa, \mu})^{\sigma'_{\!\!4}}\!\!\!&=&\!\!\!\left\{\varPhi(\phi, A, B, \nu) \in {\mathfrak{e}_7}^C \,\left|\, \begin{array}{l}\kappa \varPhi=\varPhi\kappa, \mu\varPhi=\varPhi\mu \\
\sigma'_{\!\!4} \varPhi=\varPhi \sigma'_{\!\!4} 
                                           \end{array} \right. \right \}
\end{eqnarray*}
\begin{eqnarray*}
 \!\!\!&=&\!\!\!\left\{\varPhi(\phi, A, B, \nu) \in {\mathfrak{e}_7}^C \,\left|\, \begin{array}{l}
     \phi =\left(
\begin{array}{@{\,}cc|ccc@{\,}}
\multicolumn{2}{c|}{\raisebox{-2.0ex}[-5pt][0pt]{\large $D_2$}}&&& \\
&&\multicolumn{3}{c}{\raisebox{0.9ex}[0pt]{\large $0$}} \\
\hline
&&\multicolumn{3}{c}{\raisebox{-10pt}[0pt][0pt]{\large $D_6$}}\\
\multicolumn{2}{c|}{\raisebox{1.0ex}[0pt]{\large $0$}}&&& \\
\end{array}
\right)
\begin{array}{l}
\!+ \tilde{A}_1(a)
\vspace{1mm}\\
\!+(\tau_1 E_1\!+\!\tau_2 E_2\!+\!\tau E_3\!
+\!F_1(t_1))^\sim
\end{array}
\vspace{1mm}\\
\quad D_2 \in \mathfrak{so}(2,C), D_6 \in \mathfrak{so}(6,C), a \in \C^C, \tau_k \in C,
\\ \quad \tau_1+\tau_2+\tau_3=0, t_1 \in \C^C,
\\
A=\varepsilon_2 E_2 +\varepsilon_3 E_3 +F_1(a),\,\varepsilon_k \in C, a \in \C^C,
\vspace{1mm}\\
B=\upsilon_2 E_2 +\upsilon_3 E_3 +F_1(b),\,\upsilon_k \in C, b \in \C^C,
\vspace{1mm}\\
\nu =-({3}/{2})\tau_1 
\end{array} \right. \right \}.                   
\end{eqnarray*}

In particular, $\dim_{{}_C}(({({\mathfrak{e}_7}^C)}^{\kappa, \mu})^{\sigma'_{\!\!4}})=((1+15)+2+(2+2))+(2+2) \times 2=30$.
\end{lem}
\begin{proof}
By doing simple computation, we can obtain the result above.
\end{proof}

\begin{prop}\label{prop 4.19}
The group $({({E_7}^C)}^{\kappa, \mu})^{\sigma'_{\!\!4}}$ is isomorphic to the group 
$(S\!pin(6,C) \times $ \\ $ S\!pin(6, C))/\Z_2, \Z_2 =\{ (1,1), (\sigma, \sigma)\}${\rm :} $({({E_7}^C)}^{\kappa, \mu})^{\sigma'_{\!\!4}} \cong (S\!pin(6,C) \times S\!pin(6, C))/\Z_2$.
\end{prop}
\begin{proof}
Let $S\!pin(6,C) \cong ({F_4}^C)_{E_1, E_2, E_3, \dot{F}_1(e_k), k=0,1} \!\cong 
(({E_7}^C)^{\kappa, \mu})_{\ti{E}_1, \ti{E}_{-1},E_2 \dot{+} E_3, E_2 \dot{-} E_3, \dot{F}_1(e_k), k=0,1}$ (Theorem \ref{thm 3.16}, Proposition \ref{prop 4.2}) and one more $S\!pin(6,C) \cong (({E_7}^C)^{\kappa, \mu})_{{F}_1(e_k), k=2,\ldots,7}$ (Theorem \ref{thm 4.17}). Then we define a mapping $\varphi_{\kappa, \mu, \sigma'_{\!\!4}}: S\!pin(6, C) \times S\!pin(6, C) \to ({({E_7}^C)}^{\kappa, \mu})^{\sigma'_{\!\!4}}$ by 
$$
  \varphi_{\kappa, \mu, \sigma'_{\!\!4}}(\beta_1, \beta_2)=\beta_1 \beta_2.
$$

First, we have to prove that the mapping $\varphi_{\kappa, \mu, \sigma'_{\!\!4}}$ is well-defined. It follows from Lemma \ref{lem 3.17} and Proposition \ref{prop 4.2} that $S\!pin(6,C) \cong ({F_4}^C)_{E_1, E_2, E_3, {F}_1(e_k), k=0,1} \subset ({F_4}^C)^{\sigma'_{\!\!4}} \subset (({E_7}^C)^{\kappa, \mu})^{\sigma'_{\!\!4}}$, and since $S\!pin(6,C) \cong (({E_7}^C)^{\kappa, \mu})_{\dot{F}_1(e_k), k=2,\ldots,7}$ and $({({E_7}^C)}^{\kappa, \mu})^{\sigma'_{\!\!4}}$ are connected, in order to prove $S\!pin(6,C) \cong (({E_7}^C)^{\kappa, \mu})_{\dot{F}_1(e_k), k=2,\ldots,7} \subset {({E_7}^C)}^{\kappa, \mu})^{\sigma'_4}$,  it is sufficient to show that the Lie algebra $\mathfrak{spin}(6,C) \cong (({\mathfrak{e}_7}^C)^{\kappa, \mu})_{\dot{F}_1(e_k), k=2,\ldots,7}$ is the subalgebra of the Lie algebra $(({\mathfrak{e}_7}^C)^{\kappa, \mu})^{\sigma'_{\!\!4}}$. However, from Lemmas \ref{lem 4.14}, \ref{lem 4.18}, it is clear. Hence the mapping $\varphi_{\kappa, \mu, \sigma'_{\!\!4}}$ is well-defined.

Next, we shall show that the mapping $\varphi_{\kappa, \mu, \sigma'_{\!\!4}}$ is a homomorphism. Since $S\!pin(6,C) \cong ({F_4}^C)_{E_1, E_2, E_3,{F}_1(e_k), k=0,1}$ and $S\!pin(6,C) \cong (({E_7}^C)^{\kappa, \mu})_{\dot{F}_1(e_k), k=2,\ldots,7}$ are connected, in order to prove that the mapping $\varphi_{\kappa, \mu, \sigma'_{\!\!4}}$ is a homomorphism, 
it is sufficient to show that $\varPhi_1$ commutes with $\varPhi_2$, that is, $[\varPhi_1, \varPhi_2]$ $=0$ 
for $\varPhi_1 \in \mathfrak{spin}(6,C) \cong ({\mathfrak{f}_4}^C)_{E_1, E_2, E_3, \dot{F}_1(e_k), k=0,1}$ and $\varPhi_2 \in \mathfrak{spin}(6,C) \cong (({\mathfrak{e}_7}^C)^{\kappa, \mu})_{\dot{F}_1(e_k), k=2,\ldots,7}$. However, it is also clear from Lemmas \ref{lem 3.14}, \ref{lem 4.14}. 

We determine the $\Ker\,\varphi_{\kappa, \mu, \sigma'_4}$. From the definition of kernel, we have that 
\begin{eqnarray*}
\Ker\,\varphi_{\kappa, \mu, \sigma'_4}\!\!\!&=&\!\!\! \{ (\beta_1, \beta_2) \in S\!pin(6,C) \times S\!pin(6,C)\,|\, \varphi_{\kappa, \mu, \sigma'_4}(\beta_1, \beta_2) =1 \}
\\[1mm]
\!\!\!&=&\!\!\! \{ (\beta_1, \beta_2) \in S\!pin(6,C) \times S\!pin(6,C)\,|\,\beta_1\beta_2=1 \}
\\[1mm]
\!\!\!&=&\!\!\! \{ (\beta_1, \beta_2) \in S\!pin(6,C) \times S\!pin(6,C)\,|\,\beta_1={\beta_2}^{-1} \}.
\end{eqnarray*}
Then, from the condition $\beta_1={\beta_2}^{-1}$, we see
$
    \beta_1 \dot{F}_1(e_k)={\beta_2}^{-1}\dot{F}_1(e_k)=\dot{F}_1(e_k), k=2,\ldots, 7,
$
that is, $\beta_1 \dot{F}_1(e_k)=\dot{F}_1(e_k)$. 
Moreover, since $\beta_1 \in  S\!pin(6,C) \cong ({F_4}^C)_{E_1, E_2, E_3, \dot{F}_1(e_k), k=0,1}
$, we see that $\beta_1 \dot{F}_1(x)=\dot{F}_1(x)$ for all $x \in 
\mathfrak{C}^C$. Here, from $\beta_1 \in  ({F_4}^C)_{E_1, E_2, E_3} 
\cong S\!pin(8, C)$, $\beta_1$ can be expressed by  $\beta_1=(\delta_1, \delta_2, \delta_3) \in S\!O(8, C)^{\times 3}$ such that $(\delta_1 x)(\delta_2 
 y)=\ov{\delta_3 (\ov{x y})},$ $ x,y \in \mathfrak{C}^C$, and so we have that $\delta_1 x =x$ for all $x \in \mathfrak{C}^C$. Hence we have $\delta_1 =1$, and so we see that 
$$
\beta_1=(1, 1, 1)=1\quad {\text{or}}\quad \beta_1=(1, -1, -1)=\sigma.
$$
Hence it follows from the condition $\beta_1={\beta_2}^{-1}$ that $\beta_2=1$ or $\beta_2=\sigma$, 
that is,  $\Ker\,\varphi_{\kappa, \mu, \sigma'_4} \subset \{ (1, 1), (\sigma, \sigma) \}$ and vice versa. Thus we obtain that  $\Ker\,\varphi_{\kappa, \mu, \sigma'_4}$ $=\{ (1, 1), (\sigma, \sigma) \} \cong \Z_2$. Finally, we shall show that $\varphi_{\kappa, \mu, \sigma'_4}$ is surjection. Since $\Ker\,\varphi_{\kappa, \mu, \sigma'_4}$ is discrete, the group $({({E_7}^C)}^{\kappa, \mu})^{\sigma'_{\!\!4}}$ is connected because of ${({E_7}^C)}^{\kappa, \mu} \cong S\!pin(12,C)$ (see \cite [Proposition 4.6.10]{realization E_7}) and $\dim_{{}_C}(({({\mathfrak{e}_7}^C)}^{\kappa, \mu})^{\sigma'_4})$ $=30=\dim_{{}_C}(\mathfrak{so}(6,C)\oplus \mathfrak{so}(6,C))$ (Lemma \ref{lem 4.18}), $\varphi_{\kappa, \mu, \sigma'_4}$ is surjection. 

Therefore we have the required isomorphism 
$$
 ({({E_7}^C)}^{\kappa, \mu})^{\sigma'_{\!\!4}} \cong (S\!pin(6,C) \times S\!pin(6, C))/\Z_2 . 
$$
\end{proof}
We determine the structure of the group $({E_7}^C)^{\sigma'_{\!\!4}}$ as one of aims of this section. 
\begin{lem}\label{lem 4.20}
The group $({E_7}^C)^{\sigma'_{\!\!4}}$ contains a subgroup
$$
      \psi(S\!L(2,C)) = \{ \psi(A) \in {E_7}^C  \, | \, A \in S\!L(2,C) \}
$$
which is isomorphic to the special linear group $S\!L(2,C) = \{ A \in M(2,C) \, | \, \det\,A = 1 \}$. Here, for $A \in S\!L(2,C)$, a mapping $\psi(A) : \mathfrak{P}^C \to \mathfrak{P}^C$ is defined by
$$
\psi(A)( \begin{pmatrix}\xi_1 & x_3 & \ov{x}_2 \cr
                             \ov{x}_3 & \xi_2 & x_1 \cr 
                             x_2 & \ov{x}_1 & \xi_3
                            \end{pmatrix},
                           \begin{pmatrix}\eta_1 & y_3 & \ov{y}_2 \cr
                             \ov{y}_3 & \eta_2 & y_1 \cr 
                             y_2 & \ov{y}_1 & \eta_3
                            \end{pmatrix}, \xi, \eta )
 =:(\begin{pmatrix}{\xi_1}' & {x_3}' & {\ov{x}_2}' \cr
                        {\ov{x}_3}' & {\xi_2}' & {x_1}' \cr
                         {x_2}' & {\ov{x}_1}' & {\xi_3}'
               \end{pmatrix},
               \begin{pmatrix}{\eta_1}' & {y_3}' & {\ov{y}_2}' \cr
                        {\ov{y}_3}' & {\eta_2}' & {y_1}' \cr
                         {y_2}' & {\ov{y}_1}' & {\eta_3}'
                \end{pmatrix},
                          \xi', \eta' ),
$$
where
$$
\begin{array}{c}
     \begin{pmatrix} {\xi_1}' \cr \eta'\end{pmatrix} = A\begin{pmatrix}\xi_1 \cr \eta\end{pmatrix}, \;\;
     \begin{pmatrix}\xi' \cr {\eta_1}'\end{pmatrix} = A\begin{pmatrix}\xi \cr \eta_1\end{pmatrix}, \;\;
     \begin{pmatrix}{\eta_2}'\cr {\xi_3}'\end{pmatrix} = A\begin{pmatrix}\eta_2 \cr \xi_3\end{pmatrix},\;\;
     \begin{pmatrix}{\eta_3}' \cr {\xi_2}'\end{pmatrix} = A\begin{pmatrix}\eta_3 \cr \xi_2\end{pmatrix},
\vspace{2mm}\\
     \begin{pmatrix}{x_1}' \cr {y_1}'\end{pmatrix} = \tau A\begin{pmatrix}x_1 \cr y_1\end{pmatrix}, \;\;
    \begin{pmatrix}{x_2}' \cr {y_2}'\end{pmatrix}=\begin{pmatrix}x_2 \cr y_2\end{pmatrix},  \;\;
    \begin{pmatrix}{x_3}' \cr {y_3}'\end{pmatrix}= \begin{pmatrix}x_3 \cr y_3\end{pmatrix}.
\end{array}
$$
\end{lem}
\begin{proof}The action of ${\varPhi}(\phi(\nu), aE_1, bE_1, \nu) \in ({\mathfrak{e}_7}^C)^{\sigma'_{\!\!4}}$(Lemma \ref{lem 4.1}) ($\phi(\nu) = ({2}/{3})\nu(2E_1-(E_2+E_3))^\sim,  a, b, \nu \in C$) on $\mathfrak{P}^C$ is as follows:
$$
     {\varPhi}(\phi(\nu), aE_1, bE_1, \nu)(X, Y, \xi, \eta) =:
(X', Y', \xi', \eta'), 
$$
where 
$$
\begin{array}{ll}
     \begin{pmatrix}{\xi_1} '\cr \eta'\end{pmatrix} = \begin{pmatrix}\nu & a \cr 
                                             b & -\nu\end{pmatrix}
    \begin{pmatrix}\xi_1 \cr \eta \end{pmatrix}, &
    \begin{pmatrix}\xi' \cr {\eta_1}'\end{pmatrix} = \begin{pmatrix}\nu & a \cr 
                                             b & -\nu\end{pmatrix}
     \begin{pmatrix}\xi \cr \eta_1\end{pmatrix},
\vspace{2mm}\\
    \begin{pmatrix}{\eta_2}' \cr {\xi_3}'\end{pmatrix} = \begin{pmatrix}\nu & a \cr 
                                            b & -\nu\end{pmatrix}
    \begin{pmatrix}\eta_2 \cr \xi_3\end{pmatrix}, &
  \begin{pmatrix}{\eta_3}' \cr {\xi_2}'\end{pmatrix} = \begin{pmatrix}\nu & a \cr
                                              b & - \nu\end{pmatrix}
     \begin{pmatrix}\eta_3 \cr \xi_2\end{pmatrix},
\vspace{2mm}\\
 \begin{pmatrix}{x_1}' \cr {y_1}'\end{pmatrix} =\begin{pmatrix} \tau\nu & \tau a \cr \tau b & -\tau \nu\end{pmatrix}
     \begin{pmatrix}x_1 \cr y_1\end{pmatrix}, & 
    \begin{pmatrix}{x_2}' \cr {y_2}'\end{pmatrix} =\begin{pmatrix}{x_3}' \cr {y_3}'\end{pmatrix} = \begin{pmatrix}0 \cr 0\end{pmatrix}.
\end{array}
$$

Therefore, for $A = \exp\begin{pmatrix}\nu & a \cr
                                b & -\nu\end{pmatrix}\in S\!L(2, C)\, (\begin{pmatrix}\nu & a \cr
                                b & -\nu\end{pmatrix} \in \mathfrak{sl}(2,C))$, we have that
$$
     \exp({\varPhi}(\phi(\nu), aE_1, bE_1, \nu)) = \psi(A) \in \psi(S\!L(2,C)) \subset ({E_7}^C)^{\sigma'_4}.
$$
\end{proof}

\begin{thm}\label{thm 4.21}
We have that $ ({E_7}^C)^{\sigma'_{\!\!4}} \cong (S\!L(2, C) \times S\!pin(6,C) \times S\!pin(6, C))/\Z_4 , \Z_4=\{(E, 1,1), (E, \sigma, \sigma), (-E,\sigma'_{\!\!4}, -\sigma'_{\!\!4}) , (-E,\sigma\sigma'_{\!\!4}, -\sigma\sigma'_{\!\!4}) \}$.
\end{thm}
\begin{proof}
Let $S\!L(2,C)=\{A \in M(2,C)\,|\, \det\,A=1 \}$ and two $S\!pin(6,C)$ as in Proposition \ref{prop 4.19}. Then we define a mapping $\varphi_{{E_7}^C,\sigma'_4}: S\!L(2,C) \times S\!pin(6,C) \times S\!pin(6,C) \to  ({E_7}^C)^{\sigma'_{\!\!4}}$ by
$$
 \varphi_{{E_7}^C, \sigma'_{\!\!4}}(A, \beta_1, \beta_2)=\psi(A)\beta_1\beta_2.     
$$
From Lemma \ref{lem 4.20} and Proposition \ref{prop 4.19}, it is clear that the mapping $\varphi_{{E_7}^C,\sigma'_{\!\!4}}$ is well-defined. It is easy to verify that $\varphi_{{E_7}^C,\sigma'_{\!\!4}}$ is a homomorphism. Indeed, note that $\beta_1, \beta_2 \in S\!pin(12, C) $ $\cong ({E_7}^C)^{\kappa, \mu}$. From \cite [Theorem 4.6.13]{realization E_7}, we see that $\psi(A)$ commutes with $\beta_1, \beta_2$, respectively. Moreover, as in Proposition \ref{prop 4.19},\, $\beta_1$ commutes with $\beta_2$. Hence since $\psi(A), \beta_1,\beta_2$ commute each other,  
 $\varphi_{{E_7}^C,\sigma'_{\!\!4}}$ is a homomorphism. We shall show that $\varphi_{{E_7}^C,\sigma'_{\!\!4}}$ is surjection. For $\alpha \in ({E_7}^C)^{\sigma'_{\!\!4}} \subset ({E_7}^C)^{\sigma}$, there exist $A \in S\!L(2,C)$ and $\beta \in S\!pin(12, C)$ such that $\alpha=\varphi(A, \beta)$ (see \cite [Theorem 4.6.13]{realization E_7}). Moreover, from the condition $\sigma'_{\!\!4} \alpha=\alpha\sigma'_{\!\!4}$, we have that
$$
\left\{\begin{array}{l}
    A=A \\
    \sigma'_{\!\!4} \,\beta {\sigma'_{\!\!4}}^{-1}=\beta
       \end{array} \right.         \quad {\text{or}}\quad \left\{\begin{array}{l}
    A=-A \\
    \sigma'_{\!\!4} \,\beta {\sigma'_{\!\!4}}^{-1}=-\sigma\beta.
       \end{array} \right.  
$$
Then the latter case is impossible because of $A=0$. As for the former case, from Proposition \ref{prop 4.19}, there exist $\beta_1 \in S\!pin(6,C)$ and $\beta_2 \in S\!pin(6,C)$ such that $\beta=\varphi_{\kappa, \mu, \sigma'_{\!\!4}}(\beta_1, \beta_2)$. Thus, $\varphi_{{E_7}^C,\sigma'_{\!\!4}}$ is surjection. Finally, we determine the $\Ker\,\varphi_{{E_7}^C,\sigma'_{\!\!4}}$. From $\Ker\,\varphi=\{(E,1), (-E,$ $-\sigma) \}$ (see \cite [Theorem 4.6.13]{realization E_7}) , we have that
\begin{eqnarray*}
\Ker\,\varphi_{{E_7}^C,\sigma'_4}\!\!\!&=&\!\!\!\left\{(A, \beta_1, \beta_2) \in S\!L(2, C) \times S\!pin(6,C) \times S\!pin(6, C) \,\left|\,
    A=E ,\,\, \beta_1 \beta_2=1 
\right. \right \}
\vspace{1mm}\\
&\cup & \left \{(A, \beta_1, \beta_2) \in S\!L(2, C) \times S\!pin(6,C) \times S\!pin(6, C) \,\left|\,
    A=-E, \,\,\beta_1 \beta_2=-\sigma 
\right. \right \}.
\end{eqnarray*}

\,Case (i) where $A=E, \,\,\beta_1 \beta_2=1$.

From $\Ker\,\varphi_{\kappa, \mu, \sigma'_4}=\{ (1, 1), (\sigma, \sigma) \}$ (Proposition \ref{prop 4.19}), we have that 
$$
 \left\{\begin{array}{l}
    A=E \\
    \beta_1=1 \\
    \beta_2=1
        \end{array} \right.  \quad {\text{or}} \quad \left\{\begin{array}{l}
    A=E \\
    \beta_1=\sigma \\
    \beta_2=\sigma.
        \end{array} \right.
$$

\,Case (ii) where $A=-E, \,\,\beta_1 \beta_2=-\sigma$.

Since $\beta_1 \beta_2=-\sigma \in (({E_7}^C)^{\kappa, \mu})^{\sigma'_{\!\!4}}$, there exist $\beta_1 \in S\!pin(6, C)$ and $\beta_2 \in S\!pin(6, C)$ such that $-\sigma=\beta_1 \beta_2$ (Proposition \ref{prop 4.19}). Here, we easily see that 
\begin{eqnarray*}
  && \sigma'_{\!\!4} \in  S\!pin(6, C) \cong({F_4}^C)_{E_1, E_2, E_3, F_1(e_k), k=0,1} \cong 
(({E_7}^C)^{\kappa, \mu})_{\ti{E}_1, \ti{E}_{-1},E_2 \dot{+} E_3, E_2 \dot{-} E_3, \dot{F}_1(e_k), k=0,1} ,
\\
&&-\sigma'_{\!\!4} \in   S\!pin(6, C) \cong (({E_7}^C)^{\kappa, \mu})_{\dot{F}_1(e_k), k=2,\ldots,7},
\end{eqnarray*}
and $\sigma'_{\!\!4}(-\sigma'_{\!\!4})=-\sigma$, and so together with  $\Ker\,\varphi_{\kappa, \mu, \sigma'_{\!\!4}}=\{(1,1), (\sigma, \sigma) \}$, we have that 
$$
      \left \{ \begin{array}{l}
            \beta_1=\sigma'_{\!\!4} \\
            \beta_2=-\sigma'_{\!\!4}
                  \end{array} \right.         \quad {\text{or}} \quad  \left \{ \begin{array}{l}
            \beta_1=\sigma\sigma'_{\!\!4} \\
            \beta_2=\sigma(-\sigma'_{\!\!4})
                  \end{array} \right.. 
$$

\noindent Hence we see  
$$
\Ker\,\varphi_{{E_7}^C,\sigma'_4} \subset \{(E, 1,1), (E, \sigma, \sigma), (-E,\sigma'_4, -\sigma'_4) , (-E,\sigma\sigma'_4, -\sigma\sigma'_4) \},
$$
and vice versa. Thus we obtain that
$$
\Ker\,\varphi_{{E_7}^C,\sigma'_4} = \{(E, 1,1), (E, \sigma, \sigma), (-E,\sigma'_4, -\sigma'_4) , (-E,\sigma\sigma'_4, -\sigma\sigma'_4) \} \cong \Z_4.
$$

Therefore we have the required isomorphism 
$$
({E_7}^C)^{\sigma'_4} \cong (S\!L(2, C) \times S\!pin(6,C) \times S\!pin(6, C))/\Z_4.
$$
\end{proof} 
\vspace{-1mm} 
 
Now, we determine the structure of the group $({E_7}^C)^{\sigma'_{\!\!4}, \mathfrak{so}(6,C)}$, and  prove the connectedness of its group as another aim of this section.

\begin{thm}\label{thm 4.22}
We have that $({E_7}^C)^{\sigma'_{\!\!4}, \mathfrak{so}(6,C)} \cong S\!L(2, C) \times S\!pin(6,C)$. 

In particular, the group $({E_7}^C)^{\sigma'_{\!\!4}, \mathfrak{so}(6,C)}$ is connected.
\end{thm}
\begin{proof}
Let $S\!L(2, C)$ and $S\!pin(6,C) \cong  (({E_7}^C)^{\kappa, \mu})_{\dot{F}_1(e_k),k=2,\ldots,7}$ as in Theorem \ref{thm 4.21}. Note that $\mathfrak{so}(6,C)=\{\varPhi_D =(D, 0,0,0) \in {\mathfrak{e}_7}^C \,|\, D \in \mathfrak{so}(6,C) 
\cong ({\mathfrak{f}_4}^C)_{E_1, E_2, E_3, F_1(e_k), k=0,1} \}$. \vspace{0.5mm}Then we define a mapping $\varphi_{{E_7}^C, \sigma'_4, \mathfrak{so}(6,C)}: S\!L(2, C) \times S\!pin(6,C) \to ({E_7}^C)^{\sigma'_4, \mathfrak{so}(6,C)}$ by 
$$
  \varphi_{{E_7}^C, \sigma'_{\!\!4}, \mathfrak{so}(6,C)}(A, \beta_2)=\psi(A)\beta_2, 
$$
where note that the mapping $\varphi_{{E_7}^C, \sigma'_{\!\!4}, \mathfrak{so}(6,C)}$ is the restricted mapping of the mapping $\varphi_{{E_7}^C, \sigma'_4}$ in Theorem \ref{thm 4.21}. We have to prove that $\varphi_{{E_7}^C, \sigma'_{\!\!4}, \mathfrak{so}(6,C)}$ is well-defined. In order to prove this, since $\psi(S\!L(2, C))$ and $S\!pin(6, C)$ are connected, it is sufficient to show that 
for $\varPhi(\phi(\nu), aE_1, $ $ bE_1, \nu) \in \psi_*(\mathfrak{sl}(2, C)), $ $\varPhi_2 \in \mathfrak{spin}(6,C) \cong (({\mathfrak{e}_7}^C)^{\kappa, \mu})_{\dot{F}_1(e_k),k=2,\ldots,7}$, 
$$
[\varPhi_D, \varPhi(\phi(\nu), aE_1, bE_1, \nu)]=0,[\varPhi_D, \varPhi_2]=0,
$$ 
where $\varPhi(\phi(\nu), aE_1, bE_1, \nu) \in \psi_*(\mathfrak{sl}(2, C)), $ $\varPhi_2 \in \mathfrak{spin}(6,C) \cong (({\mathfrak{e}_7}^C)^{\kappa, \mu})_{\dot{F}_1(e_k),k=2,\ldots,7}$, here a mapping $\psi_\ast$ is the differential mapping of the mapping $\psi$ in Lemma \ref{lem 4.20}. However, it is clear that $[\varPhi_D, \varPhi(\phi(\nu), aE_1, bE_1,$ $ \nu)]=0$, moreover from Lemma \ref{lem 4.15}, it is easy to verify that $[\varPhi_D, \varPhi_2]=0$. Hence $\varphi_{{E_7}^C, \sigma'_{\!\!4}, \mathfrak{so}(6,C)}$ is well-defined. Since  the mapping $\varphi_{{E_7}^C, \sigma'_{\!\!4}, \mathfrak{so}(6,C)}$ is the restricted mapping $\varphi_{{E_7}^C, \sigma'_4}$, it is clear that  the mapping $\varphi_{{E_7}^C, \sigma'_{\!\!4}, \mathfrak{so}(6,C)}$ is a homomorphism. We shall show that the mapping $\varphi_{{E_7}^C, \sigma'_{\!\!4}, \mathfrak{so}(6,C)}$ is injection. Since $\dim_C (\varphi_{{\mathfrak{e}_7}^C, \sigma'_{\!\!4}, \mathfrak{so}(6,C)})=18=3+15=\dim_C (\mathfrak{sl}(2,C) \oplus \mathfrak{spin}(6,C))$ (Lemma \ref{lem 4.1} (2)), the differential mapping ${\varphi_{{E_7}^C, \sigma'_{\!\!4}, \mathfrak{so}(6,C)}}_*$ of $\varphi_{{E_7}^C, \sigma'_{\!\!4}, \mathfrak{so}(6,C)}$ is injection. Hence we see that $\Ker\,{\varphi_{{E_7}^C, \sigma'_{\!\!4}, \mathfrak{so}(6,C)}}_*=\{0 \}$, that is, $\Ker\,\varphi_{{E_7}^C, \sigma'_{\!\!4}, \mathfrak{so}(6,C)}$ is discrete. Hence, $\Ker\,\varphi_{{E_7}^C, \sigma'_{\!\!4}, \mathfrak{so}(6,C)}$ is contained in the center $z(S\!L(2, C) \times S\!pin(6,C))=\{(E, 1), (E, \sigma), (E, -\sigma'_{\!\!4}), (E, -\sigma\sigma'_{\!\!4}), (-E, 1), (-E, \sigma), (-E, \sigma'_{\!\!4}),$ $(-E, -\sigma\sigma'_{\!\!4}) \}$. However, since the mapping ${\varphi_{{E_7}^C, \sigma'_{\!\!4}, \mathfrak{so}(6,C)}}$ maps the elements of $z(S\!L(2, C) \times S\!pin(6,C))$ to $1, \sigma,- \sigma'_{\!\!4}, -\sigma\sigma'_{\!\!4}, -1, -\sigma, \sigma'_{\!\!4}, \sigma\sigma'_{\!\!4}$, respectively, we have that $\Ker\,\varphi_{{E_7}^C, \sigma'_{\!\!4}, \mathfrak{so}(6,C)}$ $=\{(E, 1) \}$, that is, the mapping $\varphi_{{E_7}^C, \sigma'_{\!\!4}, \mathfrak{so}(6,C)}$ is injection. Finally, We shall show that the mapping $\varphi_{{E_7}^C, \sigma'_{\!\!4}, \mathfrak{so}(6,C)}$ is surjection. For $\alpha \in ({E_7}^C)^{\sigma'_{\!\!4}, \mathfrak{so}(6,C)} \subset ({E_7}^C)^{\sigma'_{\!\!4}}$, there exist $A \in S\!L(2,C), \beta_1 \in S\!pin(6,C) \cong ({F_4}^C)_{E_1, E_2, E_3,\dot{F}_1(e_k), k=0,1 }\cong  (({E_7}^C)^{\kappa, \mu})_{\ti{E}_1, \ti{E}_{-1},E_2 \dot{+} E_3, E_2 \dot{-} E_3, \dot{F}_1(e_k), k=0,1}$ and $\beta_2 \in S\!pin(6,C) \cong (({E_7}^C)^{\kappa, \mu})_{\dot{F}_1(e_k),k=2,\ldots,7}$ such that $\alpha=\psi(A)\beta_1\beta_2$ (Theorem \ref{thm 4.21}). Moreover,  from the condition $\varPhi_D \alpha=\alpha \varPhi_D$, together with $\varPhi_D \psi(A)=\psi(A) \varPhi_D$(Lemma 4.20) and $\varPhi_D \beta_2=\beta_2 \varPhi_D$(Lemma 4.14), we have $\varPhi_D \beta_1=\beta_1 \varPhi_D$ for all $D \in \mathfrak{so}(6,C)$. Hence 
$\beta_1$ is contained in the center $z(S\!pin(6,C))=z(({F_4}^C)_{E_1, E_2, E_3,\dot{F}_1(e_k), k=0,1 })=\{1, \sigma, \sigma'_{\!\!4}, \sigma\sigma'_{\!\!4} \} \cong \Z_4$. However, we see that 
\begin{eqnarray*}
&& \sigma =(1) \sigma =\psi(E)\sigma \in \psi(S\!L(2,C))S\!pin(6,C),
\\
&&\sigma'_{\!\!4}=(-1)(-\sigma'_{\!\!4})=\psi(-E)(-\sigma'_{\!\!4}) \in \psi(S\!L(2,C)) S\!pin(6,C),
\\
&&\sigma\sigma'_{\!\!4}=(-1)(-\sigma\sigma'_{\!\!4})=\psi(-E)(-\sigma\sigma'_{\!\!4}) \in \psi(S\!L(2,C)) S\!pin(6,C),
\end{eqnarray*}
that is, $\sigma, \sigma'_{\!\!4}, \sigma\sigma'_{\!\!4} \in \psi(S\!L(2,C)) S\!pin(6,C)$, where $S\!pin(6,C) \cong  (({E_7}^C)^{\kappa, \mu})_{\dot{F}_1(e_k),k=2,\ldots,7}$.  Consequently, we have $\beta_1=1$. Hence, $\varphi_{{E_7}^C, \sigma'_{\!\!4}, \mathfrak{so}(6,C)}$ is surjection. 

Thus we have the required isomorphism 
$$
({E_7}^C)^{\sigma'_{\!\!4}, \mathfrak{so}(6,C)} \cong S\!L(2, C) \times S\!pin(6,C).
$$

Therefore we see that the group $({E_7}^C)^{\sigma'_{\!\!4}, \mathfrak{so}(6,C)}$ is connected.
\end{proof}
\section{Connectedness of the group $({E_8}^C)^{\sigma'_{\!\!4},\mathfrak{so}(6,C)}$ }

We define a subgroup  $({E_8}^C)^{\sigma'_{\!\!4},\mathfrak{so}(6,C)}$ of the group $({E_8}^C)^{\sigma'_{\!\!4}}$ by
$$
     ({E_8}^C)^{\sigma'_{\!\!4},\mathfrak{so}(6,C)} = \left\{\alpha
\in {E_8}^C \, \left| \, \begin{array}{l}\sigma'_{\!\!4}\alpha = \alpha\sigma'_{\!\!4}, \\
 \varTheta(R_D)\alpha = \alpha\varTheta(R_D) \; \mbox{for all}\; D \in \mathfrak{so}(6, C) \end{array} \right. \right\}, $$
where $R_D =  (\varPhi_D, 0, 0, 0, 0, 0) \in {\mathfrak{e}_8}^C$ and $\varTheta(R_D)$ means $\ad(R_D)$. Hereafter for $R \in {\mathfrak{e}_8}^C$, we denote $\ad(R)$ by ${\varTheta}(R)$, moreover in ${\mathfrak{e}_8}^C$, we often use the following notations:
$$
\begin{array}{c}
      \varPhi = (\varPhi, 0, 0, 0, 0, 0), \quad P^- = (0, P, 0, 0, 0, 0), \quad Q_- = (0, 0, Q, 0, 0, 0), 
\vspace{1mm}\\
    \wti{r} = (0, 0, 0, r, 0, 0), \quad s^- = (0, 0, 0, 0, s, 0), \quad 
t_- = (0, 0, 0, 0, 0, t). 
\end{array} 
$$

In order to prove the connectedness of the group $({E_8}^C)^{\sigma'_{\!\!4},\mathfrak{so}(6,C)}$, we use the method used in \cite{Yokota}. However, we write this method in detail again.
\vspace{1mm}

First, we consider a subgroup $(({E_8}^C)^{\sigma'_{\!\!4},\mathfrak{so}(6,C)})_{1_-}$ of the group $({E_8}^C)^{\sigma'_{\!\!4},\mathfrak{so}(6,C)}$:
$$
   (({E_8}^C)^{\sigma'_{\!\!4},\mathfrak{so}(6,C)})_{1_-} = \Bigl\{\alpha \in ({E_8}^C)^{\sigma'_{\!\!4},\mathfrak{so}(6,C)} \,\bigm| \, \alpha1_- = 1_- \Bigr\}. 
$$

\begin{lem}\label{lem 5.1}
We have the following.

{\rm (1)} The Lie algebra $(({\mathfrak{e}_8}^C)^{\sigma'_{\!\!4},\mathfrak{so}(6,C)})_{1_-}$ of the group $(({E_8}^C)^{\sigma'_{\!\!4},\mathfrak{so}(6,C)})_{1_-}$ is given by
\begin{eqnarray*}
(({\mathfrak{e}_8}^C)^{\sigma'_{\!\!4},\mathfrak{so}(6,C)})_{1_-}
\!\!\!&=&\!\!\! \left\{ R \in {\mathfrak{e}_8}^C  \,\Bigg|\,  
\begin{array}{l}
    \sigma'_{\!\!4}R=R, 
\\[1mm]
       [R, R_D] =0 \,\, \text{for all} \,\, D \in \mathfrak{so}(6,C),
       [R, 1_-] =0 
\end{array}   \right \} 
\\[1mm]
\!\!\!&=&\!\!\! \left\{(\varPhi, 0, Q, 0, 0, t) \in {\mathfrak{e}_8}^C  \left| 
\begin{array}{l}
     \varPhi \in ({\mathfrak{e}_7}^C)^{\sigma'_{\!\!4},\mathfrak{so}(6,C)}, \\
     Q = (Z, W, \zeta, \omega), \\
 \quad
Z=\begin{pmatrix}            \zeta_1 &    0   & 0\\                                 0   &\zeta_2 & z \\
                                 0   &   \ov{z}& \zeta_3
                \end{pmatrix},
 W=\begin{pmatrix} \omega_1 &  \!\!  0      & \!\!0 \\
                        0   &\!\!\omega_2 & \!\!w \\
                        0   &   \!\!\ov{w}& \!\!\omega_3
                \end{pmatrix},
\vspace{0.5mm}\\
\quad \zeta_k, \omega_k, \zeta, \omega \in C, z, w \in \C^C,               
\\
     t \in C
\end{array} \right. \!\!\right\}, 
\end{eqnarray*}
where as for the explicit form of the Lie algebra $({\mathfrak{e}_7}^C)^{\sigma'_{\!\!4},\mathfrak{so}(6,C)}$, see Lemma 4.1 {\rm (}2{\rm)}.

In particular, 
$$
\dim_C((({\mathfrak{e}_8}^C)^{\sigma'_{\!\!4},\mathfrak{so}(6,C)})_{1_-}) = 
18 + ((3+2) \times 2+1 \times 2) + 1 = 31. 
$$

{\rm (2)} The Lie algebra $({\mathfrak{e}_8}^C)^{\sigma'_{\!\!4},\mathfrak{so}(6,C)}$ of the group $({E_8}^C)^{\sigma'_{\!\!4},\mathfrak{so}(6,C)}$ is given by
\begin{eqnarray*}
({\mathfrak{e}_8}^C)^{\sigma'_{\!\!4},\mathfrak{so}(6,C)}
\!\!\!&=&\!\!\! \left\{ R \in {\mathfrak{e}_8}^C \, \Bigg| \, 
\begin{array}{l}
    \sigma'_{\!\!4}R=R, 
\\[1mm]
       [R, R_D] =0 \,\, \text{for all} \,\, D \in \mathfrak{so}(6,C),
\end{array}   \right \} 
\\[1mm]
\!\!\!&=&\!\!\! \left\{(\varPhi, P, Q, r, s, t) \in {\mathfrak{e}_8}^C \, \left| \,
\begin{array}{l}
     \varPhi \in ({\mathfrak{e}_7}^C)^{\sigma'_{\!\!4},\mathfrak{so}(6,C)}, \\
     P = (X, Y, \xi, \eta), \\
 \quad X=\begin{pmatrix} \xi_1 &    0      & 0 \\
                                              0   &\xi_2 & x \\
                                              0   &   \ov{x}& \xi_3
                \end{pmatrix},
       Y=\begin{pmatrix} \eta_1 &   0      &  0 \\
                                              0   & \eta_2 &  y \\
                                              0   &  \ov{y}& \eta_3
                \end{pmatrix}, 
\vspace{0.5mm}\\        
\quad \xi_k, \eta_k, \xi,\eta, \ \in C, x, y \in \C^C,  \\
Q=(Z, W, \zeta, \omega)\,\,{\text{is same form as }}P, \\
     r,s,t \in C
\end{array} \right. \right\}. 
\end{eqnarray*}

In particular, 
$$
  \dim_C(({\mathfrak{e}_8}^C)^{\sigma'_{\!\!4},\mathfrak{so}(6,C)}) = 
18 + ((3+2) \times 2+1 \times 2)\times 2 + 3 = 45. 
$$
\end{lem}
\begin{proof}
(1)\,For $R=(\varPhi, P, Q, r,s,t) \in {\mathfrak{e}_8}^C$, from the condition $\sigma'_{\!\!4} R=R$, we have that 
\begin{eqnarray*}
  && \varPhi \in ({\mathfrak{e}_7}^C)^{\sigma'_{\!\!4}}, \\
  && P=(X, Y, \xi, \eta), \begin{array}{l}
             \,\,\, X=\xi_1E_1+\xi_2E_2+\xi_3E_3 +F_1(x), \xi_k,                               \xi \in C, x \in \C^C, \\
             \,\,\,Y=\eta_1E_1+\eta_2E_2+\eta_3E_3 +F_1(y),  
                              \eta_k, \eta \in C, y \in \C^C,
                          \end{array} \\
  && Q=(Z, W, \zeta, \omega),\begin{array}{l}
              \,Z=\zeta_1E_1+\zeta_2E_2+\zeta_3E_3 +F_1(z), \zeta_k,                               \zeta \in C, z \in \C^C, \\
             W=\omega_1E_1+\omega_2E_2+\omega_3E_3 +F_1(w),  
                              \omega_k, \omega \in C, w \in \C^C,
                          \end{array} \\
   && r, s, t \in C.
\end{eqnarray*}
Moreover, from the condition $[R, R_D]=0$, we have that 
$$
\varPhi \in ({\mathfrak{e}_7}^C)^{\sigma'_{\!\!4}, \mathfrak{so}(6,C)}, 
P, Q \,\,{\text{are same form above, and }}
{\text{so are }}r,s,t.
$$

Finally, from the condition $[R,1_-]=0$, we have that $P=0$ and $s=r=0$. 
Hence we have the required explicit form of the Lie algebra $(({\mathfrak{e}_8}^C)^{\sigma'_{\!\!4},\mathfrak{so}(6,C)})_{1_-}$. 
\vspace{1mm}

(2)\, By an argument similar to (1) above, we have the required result.
 \end{proof}

In Proposition \ref{prop 5.2} below, note that the subspace $(\mathfrak{P}^C)_{\sigma'_{\!\!4}}$ of $\mathfrak{P}^C$ is defined by 
\begin{eqnarray*}
   (\mathfrak{P}^C)_{\sigma'_{\!\!4}}\!\!\!& =&\!\!\!\{P \in \mathfrak{P}^C \, | \, \sigma'_{\!\!4}P=P \}
\\
\!\!\!& =&\!\!\! \{(X, Y, \xi, \eta) \in \mathfrak{P}^C \, | \, X, Y \in (\mathfrak{J}^C)_{\sigma'_{\!\!4}}, 
\xi, \eta \in C \}, 
\end{eqnarray*} 
where $(\mathfrak{J}^C)_{\sigma'_{\!\!4}}=\{X \in \mathfrak{J}^C \,|\,{\sigma'_{\!\!4}} X =X \}$.

\begin{prop}\label{prop 5.2}
The group $(({E_8}^C)^{\sigma'_{\!\!4},\mathfrak{so}(6,C)})_{1_-}$ is a semi-direct product of groups \\$\exp(\varTheta(((\mathfrak{P}^C)_{\sigma'_{\!\!4}})_- \oplus C_- ))$ and $({E_7}^C)^{\sigma'_{\!\!4},\mathfrak{so}(6,C)}${\rm:}
$$
  (({E_8}^C)^{\sigma'_{\!\!4},\mathfrak{so}(6,C)})_{1_-} = \exp(\varTheta(((\mathfrak{P}^C)_{\sigma'_{\!\!4}})_- \oplus C_- ))\rtimes ({E_7}^C)^{\sigma'_{\!\!4},\mathfrak{so}(6,C)}. 
$$

In particular, $(({E_8}^C)^{\sigma'_{\!\!4},\mathfrak{so}(6,C)})_{1_-}$ is connected.
\end{prop}
\begin{proof}Let $((\mathfrak{P}^C)_{\sigma'_{\!\!4}})_- \oplus C_- = \{(0, 0, Q, 0, 0, t) \, | \, Q \in (\mathfrak{P}^C)_{\sigma'_{\!\!4}}, t \in C\}$ be a Lie subalgebra of the Lie algebra
$(({\mathfrak{e}_8}^C)^{{\sigma'_{\!\!4}},\mathfrak{so}(6,C)})_{1_-}$ (Lemma \ref{lem 5.1} (1)). Since it follows from $[Q_-, t_-] = 0$ that $\varTheta(Q_-)$ commutes with $\varTheta(t_-)$, we have $\exp(\varTheta(Q_- + t_-)) = \exp(\varTheta(Q_-))\exp(\varTheta(t_-))$, and so we also see that $\exp(\varTheta(((\mathfrak{P}^C)_{\sigma'_{\!\!4}})_- \oplus C_-))$ is the connected subgroup of the group $(({E_8}^C)^{{\sigma'_{\!\!4}},\mathfrak{so}(6,C)})_{1-}$. 

Now, let $\alpha \in (({E_8}^C)^{{\sigma'_{\!\!4}},\mathfrak{so}(6,C)})_{1-}$ and set
$$
    \alpha\wti{1} = (\varPhi, P, Q, r, s, t), \quad \alpha1^- = (\varPhi_1, P_1, Q_1, r_1, s_1, t_1). 
$$
Then, from the relation formulas $[\alpha\wti{1}, 1_-] = \alpha[\wti{1}, 1_-] = -2\alpha1_- = -21_-, [\alpha1^-, 1_-] = \alpha[1^-, 1_-] $ $= \alpha\wti{1}$, we have that
$$
    P = 0, \; s = 0, \; r = 1, \; \varPhi = 0, \; P_1 = - Q, \; s_1 = 1, \; r_1 = -\dfrac{t}{2}. 
$$
Moreover, from $[\alpha\wti{1}, \alpha1^-] = \alpha[\wti{1}, 1^-] = 2\alpha1^-$, we have that 
$$
    \varPhi_1 = \dfrac{1}{2}Q \times Q, \; Q_1 = - \dfrac{t}{2}Q - \dfrac{1}{3}\varPhi_1Q, \; t_1 = -\dfrac{t^2}{4} - \dfrac{1}{16}\{Q, Q_1\}. 
$$
Hence we see that $\alpha$ is of the form
$$
     \alpha = \begin{pmatrix} * & * & * & 0 & \dfrac{1}{2}Q \times Q & 0 \\
                              * & * & * & 0 & -Q & 0 \\
                   * & * & * & Q & -\dfrac{t}{2}Q -\dfrac{1}{6}(Q \times Q)Q & 0 \\
                              * & * & * & 1 & -\dfrac{t}{2} & 0 \\  
                              * & * & * & 0 & 1 & 0 \\
          * & * & * & t & -\dfrac{t^2}{4} + \dfrac{1}{96}\{Q, (Q \times Q)Q\} & 1
                              \end{pmatrix}. 
$$
On the other hand, we have that

\begin{eqnarray*}
     \delta1^-\!\!\! &=& \!\!\!\exp\Big(\varTheta(\Big(\dfrac{t}{2}\Big)_-)\Big)\exp(\varTheta(Q_-))1^-
\end{eqnarray*}
\begin{eqnarray*}
        \!\!\!&=&\!\!\! \begin{pmatrix} \dfrac{1}{2}Q \times Q
\vspace{0mm}\\
                 - Q 
\vspace{0mm}\\
            - \dfrac{t}{2}Q - \dfrac{1}{6}(Q \times Q)Q 
\vspace{0mm}\\
           -\dfrac{t}{2} 
\vspace{0mm}\\
           1
\vspace{0mm}\\
       -\dfrac{t^2}{4} + \dfrac{1}{96}\{Q, (Q \times Q)Q\} \end{pmatrix}
=\alpha1^-,
\end{eqnarray*} 
and also that
$$
        \delta\wti{1} = \alpha\wti{1}, \;\; \delta1_- = \alpha1_-. 
$$
Hence we see that $\delta^{-1}\alpha \in (({E_8}^C)^{{\sigma'_{\!\!4}}, \mathfrak{so}(6,C)})_{\wti{1},1^-,1_-} = ({E_7}^C)^{{\sigma'_{\!\!4}},\mathfrak{so}(6,C)}$. 

\noindent Thus we have that
$$
     (({E_8}^C)^{{\sigma'_{\!\!4}},\mathfrak{so}(6,C)})_{1_-} = \exp(\varTheta(((\mathfrak{P}^C)_{\sigma'_{\!\!4}})_- \oplus C_-))({E_7}^C)^{{\sigma'_{\!\!4}},\mathfrak{so}(6,C)}.  
$$
Furthermore, for $\beta \in ({E_7}^C)^{{\sigma'_{\!\!4}},\mathfrak{so}(6,C)}$, it is easy to verify that
$$
   \beta(\exp(\varTheta(Q_-)))\beta^{-1} = \exp(\varTheta(\beta Q_-)),\quad \beta((\exp(\varTheta(t_-)))\beta^{-1} = \exp(\varTheta(t_-). 
$$
Indeed, for $(\varPhi', P', Q', r', s', t') \in {\mathfrak{e}_8}^C$, by doing simple computation, we have that
\begin{eqnarray*}
\beta \varTheta (Q_-) \beta^{-1}(\varPhi', P', Q', r', s', t')\!\!\!&=&\!\!\! \beta [Q_-, \beta^{-1}(\varPhi', P', Q', r', s', t')]
\\
\!\!\!&=&\!\!\! [\beta Q_-,\beta \beta^{-1}(\varPhi', P', Q', r', s', t')]\,(\beta \in {E_7}^C \subset {E_8}^C )
\\
\!\!\!&=&\!\!\! [\beta Q_-,(\varPhi', P', Q', r', s', t')]
\\
\!\!\!&=&\!\!\! \varTheta (\beta Q_-) (\varPhi', P', Q', r', s', t'),
\end{eqnarray*}
that is, $\beta \varTheta (Q_-) \beta^{-1}= \varTheta (\beta Q_-)$.
Hence we obtain that
\begin{eqnarray*}
\beta(\exp(\varTheta(Q_-)))\beta^{-1}\!\!\!&=&\!\!\!\beta \, \Bigl(\displaystyle{\sum_{n=0}^{\infty}\dfrac{1}{n!}\varTheta (Q_-)^n} \Bigr)\,\beta^{-1}
\\
\!\!\!&=&\!\!\!\displaystyle{\sum_{n=0}^{\infty}\dfrac{1}{n!}(\beta\varTheta (Q_-)\beta^{-1})^n}\,\,\,(\, \beta \varTheta (Q_-) \beta^{-1}= \varTheta (\beta Q_-))
\\
\!\!\!&=&\!\!\!\displaystyle{\sum_{n=0}^{\infty}\dfrac{1}{n!}(\varTheta (\beta Q_-))^n}
\\
\!\!\!&=&\!\!\! \exp(\varTheta(\beta Q_-)).
\end{eqnarray*}

By the argument similar to above, we have that $\beta((\exp(\varTheta(t_-)))\beta^{-1} = \exp(\varTheta(t_-)$.
\vspace{1mm}

\noindent This shows that $\exp(\varTheta(((\mathfrak{P}^C)_{\sigma'_{\!\!4}})_- \oplus C_-)) = \exp(\varTheta(((\mathfrak{P}^C)_{\sigma'_{\!\!4}})_-)\exp(\varTheta(C_-))$ is a normal subgroup of the group$(({E_8}^C)^{{\sigma'_{\!\!4}},\mathfrak{so}(6,C)})_{1_-}$. 

\noindent Moreover, we have a split exact sequence
$$
   1 \to \exp(\varTheta(((\mathfrak{P}^C)_{\sigma'_{\!\!4}})_- \oplus C_-)) \to (({E_8}^C)^{\sigma'_{\!\!4},\mathfrak{so}(6,C)})_{1_-} \to ({E_7}^C)^{\sigma'_{\!\!4},\mathfrak{so}(6,C)} \to 1.
$$

Hence the group  $(({E_8}^C)^{\sigma'_{\!\!4},\mathfrak{so}(6,C)})_{1_-}$ is a semi-direct product of $\exp(\varTheta(((\mathfrak{P}^C)_{\sigma'_{\!\!4}})_- \oplus C_-))$ and
$({E_7}^C)^{\sigma'_{\!\!4},\mathfrak{so}(6,C)}$:
$$
    (({E_8}^C)^{\sigma'_{\!\!4},\mathfrak{so}(6,C)})_{1_-} = 
\exp(\varTheta(((\mathfrak{P}^C)_{\sigma'_{\!\!4}})_- \oplus C_-))\rtimes({E_7}^C)^{\sigma'_{\!\!4},\mathfrak{so}(6,C)}. 
$$

Therefore since      
$\exp(\varTheta(((\mathfrak{P}^C)_{\sigma'_{\!\!4}})_- \oplus C_-))$ is connected and $({E_7}^C)^{\sigma'_{\!\!4},\mathfrak{so}(6,C)}$ is connected (Theorem \ref{thm 4.22}), we have that the group$(({E_8}^C)^{\sigma'_{\!\!4},\mathfrak{so}(6,C)})_{1_-}$ is connected. 
\end{proof}
\vspace{1mm}

For $R \in {\mathfrak{e}_8}^C$, we define a $C$-linear mapping $R \times R : {\mathfrak{e}_8}^C \to {\mathfrak{e}_8}^C$ by
$$
     (R \times R)R_1 = [R, [R, R_1]\,] + \dfrac{1}{30}B_8(R, R_1)R,\,\, R_1 \in {\mathfrak{e}_8}^C,
$$
where $B_8$ is the Killing form of the Lie algebra ${\mathfrak{e}_8}^C $ (As for the Killing form $B_8$, see \cite[Theorem 5.3.2]{Yokotaichiro}), and using this mapping we define a space $\mathfrak{W}^C$ by
$$
     \mathfrak{W}^C = \{R \in {\mathfrak{e}_8}^C \, | \, R \times R = 0, R \not= 0\}, 
$$
moreover, define a subspace $(\mathfrak{W}^C)_{\sigma'_{\!\!4},\mathfrak{so}(6,C)}$ of $\mathfrak{W}^C$ by
$$
      (\mathfrak{W}^C)_{\sigma'_{\!\!4},\mathfrak{so}(6,C)} = \{R \in \mathfrak{W}^C \, | \, \sigma'_{\!\!4}R = R, [R_D, R] = 0 \; \mbox{for all}\; D \in 
\vspace{3mm}
\mathfrak{so}(6,C)\}. 
$$

\begin{lem}\label{lem 5.3}
 For $R = (\varPhi, P, Q, r, s, t) \in {\mathfrak{e}_8}^C$ satisfying $\sigma'_{\!\!4}R = R$ and $[R_D, R] = 0$ for all $D \in \mathfrak{so}(6,C), R \not= 0$, $R$ belongs to $(\mathfrak{W}^C)_{\sigma'_{\!\!4},\mathfrak{so}(6,C)}$ if and only if $R$ satisfies the following conditions:
\vspace{1mm}

{\rm (1)} $2s\varPhi - P \times P = 0$ \quad {\rm (2)} $2t\varPhi + Q \times Q = 0$ 
\vspace{1.5mm}

{\rm (3)} $2r\varPhi + P \times Q = 0$ \quad {\rm (4)} $\varPhi P - 3rP - 3sQ = 0 $ 
\vspace{1.5mm}

{\rm (5)} $\varPhi Q + 3rQ - 3tP = 0 $ \quad {\rm (6)} $\{P, Q\} - 16(st + r^2) = 0$
\vspace{1.5mm}

{\rm (7)} $2(\varPhi P \times Q_1 + 2P \times \varPhi Q_1 - rP \times Q_1 - sQ \times Q_1) - \{P, Q_1\}\varPhi = 0$
\vspace{1.5mm}

{\rm (8)} $2(\varPhi Q \times P_1 + 2Q \times \varPhi P_1 + rQ \times P_1 - tP \times P_1) \!- \{Q, P_1\}\varPhi = 0$
\vspace{1.5mm}

{\rm (9)} $8((P \times Q_1)Q - stQ_1 - r^2Q_1 - \varPhi^2Q_1 + 2r\varPhi Q_1) + 5\{P, Q_1\}Q 
-2\{Q, Q_1\}P = 0$
\vspace{1.5mm}

\hspace*{-1.7mm}{\rm (10)} $8((Q \times P_1)P + stP_1 + r^2P_1 + \varPhi^2P_1 + 2r\varPhi P_1) \,+ \,5\{Q, P_1\}P 
-2\{P, Q_1\}Q= 0$
\vspace{1.5mm}

\hspace*{-1.7mm}{\rm (11)} $18(\ad\,\varPhi)^2\varPhi_1 + Q \times \varPhi_1P - P \times \varPhi_1Q) +   B_7(\varPhi, \varPhi_1)\varPhi = 0$
\vspace{1.5mm}

\hspace*{-1.7mm}{\rm (12)} $18(\varPhi_1\varPhi P -2\varPhi\varPhi_1P - r\varPhi_1P - s\varPhi_1Q) + B_7(\varPhi, \varPhi_1)P = 0$
\vspace{1.5mm}

\hspace*{-1.7mm}{\rm (13)} $18(\varPhi_1\varPhi Q -2\varPhi\varPhi_1Q + r\varPhi_1Q - t\varPhi_1P) + B_7(\varPhi, \varPhi_1)Q = 0,$

\vspace{1.5mm}
\noindent {\rm(}where $B_7$ is the Killing form of the Lie algebra ${\mathfrak{e}_7}^C${\rm)} for all $\varPhi_1 \in {\mathfrak{e}_7}^C, P_1, Q_1 \in \mathfrak{P}^C$.
\end{lem}
\begin{proof}
For $R = (\varPhi, P, Q, r, s, t) \in {\mathfrak{e}_8}^C$ satisfying $\sigma'_{\!\!4}R = R$ and $[R_D, R] = 0$ for all $D \in \mathfrak{so}(6,C), R \not= 0$, by doing simple computation of $(R\, \times\, R)R_1=0$ for all $ R_1=(\varPhi_1, P_1, Q_1, r_1, $ $s_1, t_1) \in {\mathfrak{e}_8}^C$, we have the required result above.
\end{proof}

\begin{prop}\label{prop 5.4} 
The group $(({E_8}^C)^{\sigma'_{\!\!4},\mathfrak{so}(6,C)})_0$ acts on $(\mathfrak{W}^C)_{\sigma'_{\!\!4},\mathfrak{so}(6,C)}$ transitively.
\end{prop}
\begin{proof}
Since $\alpha \in ({E_8}^C)^{\sigma'_{\!\!4},\mathfrak{so}(6,C)}$ leaves invariant the Killing form $B_8$ of ${\mathfrak{e}_8}^C\!: \!B_8(\alpha R,$ $\alpha R') = B_8(R, R')$, $R, R' \in {\mathfrak{e}_8}^C$, we have $\alpha R \in (\mathfrak{W}^C)_{\sigma'_{\!\!4},\mathfrak{so}(6,C)}$ for $R \in (\mathfrak{W}^C)_{\sigma'_{\!\!4},\mathfrak{so}(6,C)}$. 
Indeed, since we see that 
\begin{eqnarray*}
   (\alpha R \times \alpha R)R_1 \!\!\!&=&\!\!\! [\alpha R, [\alpha R, \alpha R_1]\,] + 
\dfrac{1}{30}
B_8(\alpha R, R_1)\alpha R
\\[0.5mm]
      \!\!\!&=&\!\!\! \alpha[\,[R, [R, \alpha^{-1}R_1]\,] + 
\dfrac{1}{30}
B_8(R, \alpha^{-1}R_1)\alpha R
\\[0.5mm]
      \!\!\!&=&\!\!\! \alpha((R \times R)\alpha^{-1}R_1      
\\[0.5mm]
      &=& 0,
\\[0.5mm]
[R_D, \alpha R] \!\!\!&=&\!\!\! \alpha[\alpha^{-1}R_D, R] = \alpha[R_D, R]
\\[0.5mm]
\!\!\!&=&\!\!\! 0, 
\end{eqnarray*} 
this shows that the group $({E_8}^C)^{\sigma'_{\!\!4},\mathfrak{so}(6,C)}$ acts on $(\mathfrak{W}^C)_{\sigma'_{\!\!4},\mathfrak{so}(6,C)}$. We shall show that this action is transitive. First, for    
$R_1 \in {\mathfrak{e}_8}^C$, it follows from 
\begin{eqnarray*}
     (1_- \times 1_-)R_1 \!\!\!&=&\!\!\! [1_-,[1_-, (\varPhi_1, P_1, Q_1, r_1, s_1, t_1)] \, ] + 
\dfrac{1}{30}
B_8(1_-, R_1)1_-
\\[0.5mm]
 \!\!\!&=&\!\!\! [1_-, (0, 0, P_1, -s_1, 0, 2r_1)] + 2s_11_- 
\\[0.5mm]
 \!\!\!&=&\!\!\! (0, 0, 0, 0, -2s_1) + 2s_11_- 
\\[0.5mm]
 \!\!\!&=&\!\!\! 0,   
\\[0.5mm]
 [R_D, 1_-] \!\!\!&=&\!\!\! 0,
\end{eqnarray*}
and $\sigma'_{\!\!4} 1_- =1_-$ that we confirm $1_- \in (\mathfrak{W}^C)_{\sigma'_{\!\!4},\mathfrak{so}(6,C)}$. Then, in order to prove the transitivity of this action, it is sufficient to show that any element $R \in (\mathfrak{W}^C)_{\sigma'_{\!\!4},\mathfrak{so}(6,C)}$ can be transformed to $1_- \in (\mathfrak{W}^C)_{\sigma'_{\!\!4},\mathfrak{so}(6,C)}$ by some $\alpha \in ({E_8}^C)^{\sigma'_{\!\!4},\mathfrak{so}(6,C)}$. Indeed, we have the following.
\vspace{1mm}

{Case} (i) where $R = (\varPhi, P, Q, r, s,t), t \not= 0$. 

From Lemma \ref{lem 5.3} (2),(5) and (6), we have that
$$
   \varPhi = -\dfrac{1}{2t}Q \times Q, \; P = \dfrac{r}{t}Q - \dfrac{1}{6t^2}(Q \times Q)Q, \; s = -\dfrac{r^2}{t} + \dfrac{1}{96t^3}\{Q, (Q \times Q)Q\}. 
$$
Now, for $\varTheta = \varTheta(0, P_1, 0, r_1, s_1, 0) \in \varTheta(({\mathfrak{e}_8}^C)^{\sigma'_{\!\!4},\mathfrak{so}(6,C)})$ (Lemma \ref{lem 5.1} (2)), we compute $\varTheta^n1_-$:

$$
\begin{array}{l}
\varTheta^n1_-  = \begin{pmatrix} ((-2)^{n-1} + (-1)^n){r_1}^{n-2}P_1 \times P_1 
\vspace{1mm}\\
 \Big((-2)^{n-1} - \dfrac{1 + (-1)^{n-1}}{2}\Big){r_1}^{n-2}s_1P_1 + \Big(\dfrac{1 - (-2)^n}{6} + \dfrac{(-1)^n}{2}\Big){r_1}^{n-3}(P_1 \times P_1)P_1
\vspace{1mm}\\
      ((-2)^n + (-1)^{n-1}){r_1}^{n-1}P_1
\vspace{1mm}\\
      (-2)^{n-1}{r_1}^{n-1}s_1
\vspace{1mm}\\
 -((-2)^{n-2} + 2^{n-2}){r_1}^{n-2}{s_1}^2 + \dfrac{2^{n-2} + (-2)^{n-2} - (-1)^{n-1}}{24}{r_1}^{n-4}\{P_1,(P_1 \times P_1)P_1\}
\vspace{1mm}\\
     (-2)^n{r_1}^n \end{pmatrix}. 
\end{array} 
$$
Then, by doing simple computation, we have that 
$$
\begin{array}{l}
    \exp(\varPhi(0, P_1, 0, r_1, s_1, 0))1_- = (\exp\varTheta)1_- = \Big(\dsum_{n= 0}^\infty\dfrac{1}{n!}\varTheta^n \Big)1_-
\vspace{0.5mm}\\
= \begin{pmatrix}   -\dfrac{1}{2{r_1}^2}(e^{-2r_1} -2e^{-r_1} + 1)P_1 \times P_1
\vspace{0.5mm}\\
  \dfrac{s_1}{2{r_1}^2}(-e^{-2r_1} - e^{r_1} + e^{-r_1} + 1)P_1 + \dfrac{1}{6{r_1}^3}(-e^{-2r_1} + e^{r_1} + 3e^{-r_1} - 3)(P_1 \times P_1)P_1
\vspace{0.5mm}\\
   \dfrac{1}{r_1}(e^{-2r_1} - e^{-r_1})P_1
\vspace{0.5mm}\\  
   \dfrac{s_1}{2r_1}(1 - e^{-2r_1})
\vspace{0.5mm}\\
   -\dfrac{{s_1}^2}{4{r_1}^2}(e^{-2r_1} + e^{2r_1} -2) + \dfrac{1}{96{r_1}^4}(e^{2r_1} + e^{-2r_1} - 4e^{r_1} - 4e^{-r_1} + 6)\{P_1, (P_1 \times P_1)P_1\}
\vspace{0.5mm}\\
e^{-2r_1} \end{pmatrix}. 
\end{array} 
$$
Note that if $r_1 = 0$,
$$
 \dfrac{f(r_1)}{{r_1}^k} \,\,\,{\text{means}} \,\,\,\dlim_{r_1 \to 0}\dfrac{f(r_1)}{{r_1}^k}.
$$
Here we set
$$
   Q = \dfrac{1}{r_1}(e^{-2r_1} - e^{-r_1})P_1, \;\; r = \dfrac{s_1}{2r_1}(1 - e^{-2r_1}), \;\; t = e^{-2r_1}. 
$$
Then we have that 
$$
    (\exp\,\varTheta)1_- = \begin{pmatrix} -\dfrac{1}{2t}Q \times Q
\vspace{1mm}\\
                                \dfrac{r}{t}Q - \dfrac{1}{6t^2}(Q \times Q)Q 
\vspace{0.5mm}\\
                              Q 
\vspace{0.5mm}\\
                              r 
\vspace{0.5mm}\\
         -\dfrac{r^2}{t} + \dfrac{1}{96 t^3}\{Q, (Q \times Q)Q\}
\vspace{0.5mm}\\
                              t 
\end{pmatrix} 
= \begin{pmatrix} \varPhi \vspace{1mm}\\
                P \vspace{0.5mm}\\
                Q \vspace{0.5mm}\\
                r \vspace{0.5mm}\\
                s \vspace{0.5mm}\\
                t \end{pmatrix} =: R. 
$$
Thus $R$ is transformed to $1_-$ by $(\exp\,\varTheta)^{-1} \in (({E_8}^C)^{\sigma'_{\!\!4},\mathfrak{so}(6,C)})_0$. 
\vspace{1mm}

{Case} (ii) where  $R = (\varPhi, P, Q, r, s, 0), s \not= 0$.

 First, we denote $
\exp(\varTheta(0, 0, 0, 0, {\pi}/{2}, -{\pi}/{2})) \in (({E_8}^C)^{\sigma'_{\!\!4},\mathfrak{so}(6,C)})_0$  by $\lambda'$
 (Lemma \ref{lem 5.1} (2)). Here, operate $\lambda'$ on $R$, then we have that 
$$
    \lambda'R = \lambda'(\varPhi, P, Q, r, s, 0) = (\varPhi, Q,-P, -r, 0, -s), \;\;\; -s \not= 0. 
$$
Hence this case can be reduced to Case (i).
\vspace{0.5mm}

{Case} (iii) where  $R = (\varPhi, P, Q, r, 0, 0), r \not= 0$. 

From Lemma \ref{lem 5.3} (2),(5) and (6), we have that
$$
      Q \times Q = 0, \;\; \varPhi Q = -3rQ, \;\; \{P, Q\} = 16r^2. 
$$   
Then, for $\varTheta = \varTheta(0, Q, 0, 0, 0, 0) \in \varTheta(({\mathfrak{e}_8}^C)^{\sigma'_{\!\!4},\mathfrak{so}(6,C)})$ (Lemma \ref{lem 5.1} (2)), we see that
$$
      (\exp\, \varTheta)R = (\varPhi, P + 2rQ, Q, r, -4r^2, 0), \;\; -4r^2 \not= 0. 
$$
Hence this case can be reduced to Case (ii).
\vspace{0.5mm}

{Case} (iv) where $R = (\varPhi, P, Q, 0, 0, 0), Q \not= 0$. 

 We choose $P_1 \in (\mathfrak{P}^C)_{\sigma'_{\!\!4}}$ such that $\{P_1, Q\} \not= 0$. Then, for $\varTheta = \varTheta(0, P_1, 0, 0, 0, 0) \in \varTheta(({\mathfrak{e}_8}^C)^{\sigma'_{\!\!4},\mathfrak{so}(6,C)})$ (Lemma \ref{lem 5.1} (2)), we have that
\begin{eqnarray*}
       (\exp \, \varTheta)R \!\!\!&=&\!\!\! (\varPhi+P_1 \times Q, \,P-\varPhi P_1+\dfrac{1}{2}(P_1 \times Q)P_1, \,Q,\, -\dfrac{1}{8}\{P_1, Q\},
\\[0mm]
&&\quad \dfrac{1}{4}\{P_1, P\}+\dfrac{1}{8}\{P_1, -\varPhi P_1\}+\dfrac{1}{24}\{P_1,( P_1 \times Q)P_1\} ,\, 0), \quad  -\dfrac{1}{8}\{P_1, Q\} \not =0.
\end{eqnarray*}
Hence this case can be reduced to Case (iii).       
\vspace{1mm}

{Case} (v) where $R = (\varPhi, P, 0, 0, 0, 0), P \not= 0$.

 We choose $Q_1 \in (\mathfrak{P}^C)_{\sigma'_{\!\!4}}$ such that $\{P, Q_1\} \not= 0$. Then, for $\varTheta = \varTheta(0, 0, Q_1, 0, 0, 0) \in \varTheta(({\mathfrak{e}_8}^C)^{\sigma'_{\!\!4},\mathfrak{so}(6,C)})$ (Lemma \ref{lem 5.1} (2)), we have
\begin{eqnarray*}
       (\exp \, \varTheta)R \!\!\!&=&\!\!\! (\varPhi-P \times Q_1,P ,-\varPhi Q_1-\dfrac{1}{2}(P \times Q_1)Q_1 , \dfrac{1}{8}\{P, Q_1\},
\\[0mm]
&&\quad 0 ,-\dfrac{1}{8}\{Q_1, -\varPhi Q_1\} -\dfrac{1}{24}\{Q_1, -(P \times Q_1)Q_1  \} ), \quad \dfrac{1}{8}\{P, Q_1\} \not=0. 
\end{eqnarray*}
Hence this case can be reduced to Case (iii). 
\vspace{1mm}

{Case} (vi) where $R = (\varPhi, 0, 0, 0, 0, 0), \varPhi \not= 0.$ 
From  Lemma \ref{lem 5.3} (10), we have $\varPhi^2=0$. We choose $P_1 \in (\mathfrak{P}^C)_{\sigma'_{\!\!4}}$ such that $\varPhi P_1 \not = 0$.  Then, for $\varTheta = \varTheta(0, P_1, 0, 0, 0, 0) \in \varTheta(({\mathfrak{e}_8}^C)^{\sigma'_{\!\!4},\mathfrak{so}(6,C)})$ (Lemma \ref{lem 5.1} (2)), we have that
$$
    (\exp\,\varTheta)R = \Big(\varPhi, -\varPhi P_1, 0, 0, \dfrac{1}{8}\{\varPhi P_1, P_1\}, 0 \Big). 
$$
Hence this case is also reduced to Case (v).

Thus the proof of this proposition is completed. 
\end{proof}

Now, we shall prove the theorem as the aim of this section.
\begin{thm}\label{thm 5.5}
The homogeneous space $({E_8}^C)^{\sigma'_{\!\!4},\mathfrak{so}(6,C)}/(({E_8}^C)^{\sigma'_{\!\!4},\mathfrak{so}(6,C)})_{1_-}$ is diffeomorphic to the space $(\mathfrak{W}^C)_{\sigma'_{\!\!4},\mathfrak{so}(6,C)}${\rm : } $({E_8}^C)^{\sigma'_{\!\!4},\mathfrak{so}(6,C)}/(({E_8}^C)^{\sigma'_{\!\!4},\mathfrak{so}(6,C)})_{1_-} \simeq (\mathfrak{W}^C)_{\sigma'_{\!\!4},\mathfrak{so}(6,C)}$.

In particular, the group $({E_8}^C)^{\sigma'_{\!\!4},\mathfrak{so}(6,C)}$ is connected.
\end{thm}
\begin{proof}
Since the group $({E_8}^C)^{\sigma'_{\!\!4},\mathfrak{so}(6,C)}$ acts on the space $(\mathfrak{W}^C)_{\sigma'_{\!\!4},\mathfrak{so}(6,C)}$ transitively (Proposition \ref{prop 5.4}), the former half of this theorem is proved.

The latter half can be shown as follows. Since $(({E_8}^C)^{\sigma'_{\!\!4},\mathfrak{so}(6,C)})_{1_-}$ and $(\mathfrak{W}^C)_{\sigma'_{\!\!4}}$  are connected (Propositions \ref{prop 5.2}, \ref{prop 5.4}), we have that $({E_8}^C)^{\sigma'_{\!\!4},\mathfrak{so}(6,C)}$ is also connected. 
\end{proof}
\vspace{-1mm}

\section{ Construction of $S\!pin(10, C)$ in ${E_8}^C$}
\vspace{1mm}

We define a subgroup  $({E_8})^{\sigma'_{\!\!4},\mathfrak{so}(6)} 
$ of $E_8$ by
$$
 ({E_8})^{\sigma'_{\!\!4},\mathfrak{so}(6)} = \left\{\alpha
\in E_8 \, \left| \, 
\begin{array}{l}\sigma'_{\!\!4}\alpha = \alpha\sigma'_{\!\!4}, \\
 \varTheta(R_D)\alpha = \alpha\varTheta(R_D) \; \mbox{for all}\; D \in \mathfrak{so}(6) 
\end{array} 
\right. \right\}.
$$
Then we have the following lemma.

\begin{lem}\label{lem 6.1}
The Lie algebra $({\mathfrak{e}_8})^{\sigma'_{\!\!4},\mathfrak{so}(6)}$ of the group $({E_8})^{\sigma'_{\!\!4},\mathfrak{so}(6)}$ is given by
\begin{eqnarray*}
({\mathfrak{e}_8})^{\sigma'_{\!\!4},\mathfrak{so}(6)}
\!\!\!&=&\!\!\! \left\{ R \in {\mathfrak{e}_8} \, \Bigg| \, 
\begin{array}{l}
    \sigma'_{\!\!4}R=R, 
\\[1mm]
       [R, R_D] =0 \,\, \text{for all} \,\, D \in \mathfrak{so}(6)
\end{array}   \right \} 
\\[1mm]
\!\!\!&=&\!\!\! \left\{(\varPhi, P, -\tau\lambda P, r, s, -\tau s) \in {\mathfrak{e}_8} \, \left| \,
\begin{array}{l}
     \varPhi \in ({\mathfrak{e}_7})^{\sigma'_{\!\!4},\mathfrak{so}(6)}, \\
     P = (X, Y, \xi, \eta), \\
 \quad X=\begin{pmatrix} \xi_1 &    0      & 0 \\
                                              0   &\xi_2 & x \\
                                              0   &   \ov{x}& \xi_3
                \end{pmatrix},
       Y=\begin{pmatrix} \eta_1 &   0      &  0 \\
                                              0   & \eta_2 &  y \\
                                              0   &  \ov{y}& \eta_3
                \end{pmatrix}, 
\vspace{0.5mm}\\        
\quad \xi_k, \eta_k, \xi,\eta, \ \in C, x, y \in \C^C,  \\
     r \in i\R,s \in C, 
\end{array} \right. \right\}, 
\end{eqnarray*}
where as mentioned in Lemma 4.11, as for the Lie algebra $\mathfrak{e}_7$ of the compact Lie group $E_7$, 
see \cite[theorem 4.3.4]{Yokotaichiro} in detail, and so the Lie algebra $({\mathfrak{e}_7})^{\sigma'_{\!\!4},\mathfrak{so}(6)}$ above is defined as follows{\rm:}
\begin{eqnarray*}
({\mathfrak{e}_7})^{\sigma'_{\!\!4},\mathfrak{so}(6)}\!\!\!&=&\!\!\!\left\{ \varPhi(\phi, A, -\tau A, \nu)\,\Biggm|\,
\begin{array}{l}
    \sigma'_{\!\!4}\varPhi=\varPhi\sigma'_{\!\!4}, 
\\[1mm]
       [\varPhi, \varPhi_D] =0 \,\, \text{for all} \,\, D \in \mathfrak{so}(6) \end{array} \right \}
\\[1mm]
 \!\!\! &=&\!\!\!\!\left\{\varPhi(\phi, A, -\tau A, \nu) \in {\mathfrak{e}_7} \,\left| \, 
\begin{array}{l}
     \phi = \left(
\begin{array}{@{\,}cc|cc@{\,\,\,}}
\multicolumn{2}{c|}{\raisebox{-2.0ex}[-5pt][0pt]{\large $D_2$}}&& \\
&&\multicolumn{2}{c}{\raisebox{0.9ex}[0pt]{ \large $0$}} \\
\hline
&&\multicolumn{2}{c}{\raisebox{-10pt}[0pt][0pt]{\large $0$}}\\
\multicolumn{2}{c|}{\raisebox{3pt}[0pt]{\large $0$}}&& \\
\end{array}
\right) \begin{array}{l}
\!+ \tilde{A}_1(a)
\\ \vspace{1mm}
\!+i(\tau_1 E_1\!+\!\tau_2 E_2\!+\!\tau E_3\!+\!F_1(t_1))^\sim
\end{array}
\vspace{1mm}\\
\quad D_2 \in \mathfrak{so}(2), a \in \C, \tau_k \in \R,
\\ \quad \tau_1+\tau_2+\tau_3=0, t_1 \in \C,
\\
A=\begin{pmatrix}\xi_1 & 0       & 0   \\
                    0   & \xi_2  & x_1 \\
                    0   & \ov{x}_1& \xi_3
   \end{pmatrix},\,\,\xi_k \in C, x_1 \in \C^C,
\vspace{1mm}\\
\nu \in i\R 
\end{array}
\right.\!\!\!\!\! \right \}.      
\end{eqnarray*}

In particular, 
$$
  \dim_C(({\mathfrak{e}_8}^C)^{\sigma'_{\!\!4},\mathfrak{so}(6,C)}) = 
18 + ((3+2) \times 2+1 \times 2)\times 2 + 3 = 45. 
$$
\end{lem}
\begin{proof}
By the argument similar to Lemma 5.1 (2), we have the required result.
\end{proof}
\begin{prop} \label{prop 6.2}
The Lie algebra 
$({\mathfrak{e}_8})^{\sigma'_{\!\!4}, \mathfrak{so}(6)}$ 
is isomorphic to the Lie algebra $\mathfrak{so}(10)$ and the Lie algebra 
$({\mathfrak{e}_8}^C)^{\sigma'_{\!\!4}, \mathfrak{so}(6,C)}$ 
is isomorphic to the Lie algebra $\mathfrak{so}(10,C)${\rm : }$({\mathfrak{e}_8})^{\sigma'_{\!\!4}, \mathfrak{so}(6)} \cong \mathfrak{so}(10)$ and $({\mathfrak{e}_8}^C)^{\sigma'_{\!\!4}, \mathfrak{so}(6,C)} \cong \mathfrak{so}(10,C)$.
\end{prop}
\begin{proof}
We provide the correspondence $\varphi_*$ between the Lie algebra  $\mathfrak{so}(10) \!=\! \{D \in M(10,\R) \, | $ $\, {}^tD = - D \}$ and the Lie algebra $({\mathfrak{e}_8})^{\sigma'_{\!\!4}, \mathfrak{so}(6)}$ explicitly as follows: 
$$
\begin{array}{c}                                                    
   \varphi_* : \mathfrak{so}(10) \quad \longrightarrow \quad ({\mathfrak{e}_8})^{\sigma'_{\!\!4}, \mathfrak{so}(6)},
\vspace{2mm}\\
\qquad \qquad \qquad
   G_{ij} \qquad \;\; \longmapsto \qquad R_{ij},\quad 0 \le i < j \le 9,
\end{array} 
$$
where the elements $G_{ij}$ are $\R$-basis in $\mathfrak{so}(10)$. As its example, the explicit form of $G_{26}$ as matrix is of form with $(3,7)$-component $=1$, $(7, 3)$-component $=-1$, other components $=0$, moreover the explicit forms of $C$-basis $R_{ij}$ in $({\mathfrak{e}_8})^{\sigma'_{\!\!4}, \mathfrak{so}(6)}$ are as follows.
\begin{eqnarray*}
 && R_{01} = ({\varPhi}({-i(E_2 - E_3)}^{\sim}, 0, 0, 0), 0, 0, 0, 0, 0)
\\[1mm]
&& R_{02} = \Big({\varPhi}\Big(0, -\displaystyle{\frac{i}{2}}(E_2 - E_3), -\displaystyle{\frac {i}{2}}(E_2 - E_3), 0 \Big), 0, 0, 0, 0, 0\Big)
\\[0mm]
&& R_{12} = \Big({\varPhi}\Big(0, \displaystyle{\frac{1}{2}}(E_2 + E_3), -\displaystyle{\frac{1}{2}}(E_2 + E_3), 0 \Big), 0, 0, 0, 0, 0\Big)
\end{eqnarray*}
\begin{eqnarray*}
&& R_{03} = \Big({\varPhi}\Big(0, -\displaystyle{\frac{1}{2}}(E_2 - E_3), \displaystyle{\frac{1}{2}}(E_2 - E_3), 0 \Big), 0, 0, 0, 0, 0\Big)
\\[1mm]
&& R_{13} = \Big({\varPhi}\Big(0, -\displaystyle{\frac{i}{2}}(E_2 + E_3), -\displaystyle{\frac{i}{2}}(E_2 + E_3), 0 \Big), 0, 0, 0, 0, 0\Big)
\\[1mm]
&& R_{23} = ({\varPhi}(-i(E_1 \vee E_1), 0, 0, i), 0, 0, 0, 0, 0)
\\[1mm]
&& R_{04} = (0, (-(E_2 - E_3), 0, 0, 0), (0, -(E_2 - E_3), 0, 0), 0, 0, 0)
\\[1mm]
&& R_{14} = (0, (-i(E_2 + E_3), 0, 0, 0), (0, i(E_2 + E_3), 0, 0), 0, 0, 0)
\\[1mm]
&&  R_{24} = (0, (0, iE_1, 0, -i), (iE_1, 0, -i, 0), 0, 0, 0)
\\[1mm]
&&  R_{34} = (0, (0, E_1, 0, 1), (-E_1, 0, -1, 0), 0, 0, 0)
\\[1mm]
&& R_{05} = (0, (-i(E_2 - E_3), 0, 0, 0), (0, i(E_2 - E_3), 0, 0), 0, 0, 0)
\\[1mm]
&& R_{15} = (0, (E_2 + E_3, 0, 0, 0), (0, E_2 + E_3, 0, 0), 0, 0, 0)
\\[1mm]
&&  R_{25} = (0, (0, -E_1, 0, 1), (E_1, 0, -1, 0), 0, 0, 0)
\\[1mm]
&&  R_{35} = (0, (0, iE_1, 0, i), (iE_1, 0, i, 0), 0, 0, 0)
\\[1mm]
&& R_{45} = \Big({\varPhi}\Big(i(E_1 \vee E_1), 0, 0, \displaystyle{\frac{i}{2}} \Big), 0, 0 ,-\displaystyle{\frac{i}{2}}, 0, 0 \Big)
\\[1mm]
&&  R_{06} = (0, (0, -(E_2 - E_3), 0, 0), (E_2 - E_3, 0, 0, 0), 0, 0, 0)
\\[1mm]
&&  R_{16} = (0, (0, i(E_2 + E_3), 0, 0), (i(E_2 + E_3), 0, 0, 0), 0, 0, 0)
\\[1mm]
&&  R_{26} = (0, (iE_1, 0, -i, 0), (0, -iE_1, 0, i), 0, 0, 0)
\\[1mm]
&&  R_{36} = (0, (-E_1, 0, -1, 0), (0, -E_1, 0, -1), 0, 0, 0)
\\[1mm]
&&  R_{46} = \Big({\varPhi}\Big(0, \displaystyle{\frac{1}{2}}E_1, -\displaystyle{\frac{1}{2}}E_1, 0 \Big), 0, 0, 0, -\displaystyle{\frac{1}{2}}, \displaystyle{\frac{1}{2}} \Big)
\\[0.3mm]
&& R_{56} = \Big({\varPhi}\Big(0, -\displaystyle{\frac{i}{2}}E_1, -\displaystyle{\frac{i}{2}}E_1, 0 \Big), 0, 0, 0, -\displaystyle{\frac{i}{2}}, -\displaystyle{\frac{i}{2}} \Big)
\\[1mm]
&&  R_{07} = (0, (0, -i(E_2 - E_3), 0, 0), (-i(E_2 - E_3), 0, 0, 0), 0, 0, 0)
\\[1mm]
&&  R_{17} = ( 0,(0, -(E_2 + E_3),0 ,0 ),( E_2 + E_3,0 ,0, 0 ),0 ,0 ,0 )
\\[1mm]
&&  R_{27} = (0, (-E_1, 0, 1, 0), (0, -E_1, 0, 1), 0, 0, 0)
\\[1mm]
&&  R_{37} = (0, (-iE_1, 0, -i, 0), (0, iE_1, 0, i), 0, 0, 0)
\\[1mm]
&&  R_{47} = \Big({\varPhi}\Big(0, \displaystyle{\frac{i}{2}}E_1, \displaystyle{\frac{i}{2}}E_1, 0 \Big), 0, 0, 0, -\displaystyle{\frac{i}{2}}, -\displaystyle{\frac{i}{2}} \Big)
\\[1mm]
&&  R_{57} = \Big({\varPhi}\Big(0, \displaystyle{\frac{1}{2}}E_1, -\displaystyle{\frac{1}{2}}E_1, 0 \Big), 0, 0, 0, \displaystyle{\frac{1}{2}}, -\displaystyle{\frac{1}{2}} \Big)
\\[1mm]
&&  R_{67} = \Big({\varPhi}\Big(-i(E_1 \vee E_1), 0, 0, -\displaystyle{\frac{i}{2}} \Big), 0, 0, -\displaystyle{\frac{i}{2}}, 0, 0 \Big)
\\[1mm]
&& R_{08}=(\varPhi(\tilde{A}_1(i),0,0,0),0,0,0,0,0)
\\[1mm]
&& R_{18}=(\varPhi(-F_1(1)^\sim,0,0,0),0,0,0,0,0)
\\[1mm]
&& R_{28}=\Big( \varPhi\Big(0,-\dfrac{1}{2}F_1(1), -\dfrac{1}{2}F_1(1),0\Big),0,0,0,0,0\Big)
\\[1mm]
&& R_{38}=\Big( \varPhi\Big(0,\dfrac{i}{2}F_1(1), -\dfrac{i}{2}F_1(1),0\Big),0,0,0,0,0\Big)
\\[1mm]
&& R_{48}=(0, (i F_1(1), 0,0,0),(0, i F_1(1),0,0), 0,0,0 )
\\[1mm]
&& R_{58}=(0, (- F_1(1), 0,0,0),(0,  F_1(1),0,0), 0,0,0 )
\end{eqnarray*}
\begin{eqnarray*}
&& R_{68}=(0, (0,i F_1(1), 0,0),( -i F_1(1),0,0,0), 0,0,0 )
\\[1mm]
&& R_{78}=(0, (0,-F_1(1), 0,0),( -F_1(1),0,0,0), 0,0,0 )
\\[1mm]
&& R_{09}=(\varPhi(i\tilde{A}_1(e_1),0,0,0),0,0,0,0,0)
\\[1mm]
&& R_{19}=(\varPhi(-F_1(e_1)^\sim,0,0,0),0,0,0,0,0)
\\[1mm]
&& R_{29}=\Big( \varPhi\Big(0,-\dfrac{1}{2}F_1(e_1), -\dfrac{1}{2}F_1(e_1),0\Big),0,0,0,0,0\Big)
\\[1mm]
&& R_{39}=\Big( \varPhi\Big(0,\dfrac{i}{2}F_1(e_1), -\dfrac{i}{2}F_1(e_1),0\Big),0,0,0,0,0\Big)
\\[1mm]
&& R_{49}=(0, (i F_1(e_1), 0,0,0),(0, i F_1(e_1),0,0), 0,0,0 )
\\[1mm]
&& R_{59}=(0, (- F_1(e_1), 0,0,0),(0,  F_1(e_1),0,0), 0,0,0 )
\\[1mm]
&& R_{69}=(0, (0,-i F_1(e_1), 0,0),( i F_1(e_1),0,0,0), 0,0,0 )
\\[1mm]
&& R_{79}=(0, (0,-F_1(e_1), 0,0),( -F_1(e_1),0,0,0), 0,0,0 )
\\[1mm]
&& R_{89}=(\varPhi(-[\tilde{A}_1(1), \tilde{A}_1(e_1)], 0,0,0),0,0,0,0,0)\,(-[\tilde{A}_1(1), \tilde{A}_1(e_1)] \in \mathfrak{D}_4)
\end{eqnarray*}

\noindent Then we can confirm that $\varphi_*$ is a Lie-homomorphism, that is:
$$
   \varphi_*([G_{ij}, G_{kl}]) = [\varphi_*(G_{ij}), \varphi_*(G_{kl})]. 
$$ 
In order to prove these, we need to check ${}_{45}C_2 = 990$ times if honestly doing. Then, we give only five examples. As the first example, we shall show that 
$\varphi_*([G_{01}, G_{02}]) = [\varphi_*(G_{01}), \varphi_*(G_{02})]$.
Indeed,
   $\varphi_*([G_{01}, G_{02}]) = \varphi_*(-G_{12}) = -R_{12}$. 
\vspace{1mm}

\noindent On the other hand, 
\begin{small}
\begin{eqnarray*}
&& [\varphi_*(G_{01}), \varphi_*(G_{02})] = [R_{01}, R_{02}]
\\[0.5mm]
\!\!\!&=&\!\!\! \Big[({\varPhi}(-i(E_2 - E_3)^\sim, 0, 0, 0), 0, 0, 0, 0, 0), \Big({\varPhi}\Big(0, -\frac{i}{2}(E_2 - E_3), -\frac{i}{2}(E_2 
 -E_3), 0\Big),0, 0, 0, 0, 0\Big)\Big]
\\[0.5mm] 
\!\!\!&=&\!\!\! \Big(\Big[{\varPhi}(-i(E_2 - E_3)^\sim, 0, 0, 0), \Big({\varPhi}\Big(0, -\frac{i}{2}(E_2 - E_3), -\frac{i}{2}(E_2 - E_3), 0\Big)\Big], 0,
 0, 0, 0, 0\Big)
\\[0.5mm]
\!\!\!&=&\!\!\! \Big({\varPhi}\Big(0, -\frac{1}{2}(E_2 + E_3), \frac{1}{2}(E_2 + E_3), 0\Big), 0, 0, 0, 0, 0\Big)
\\[0.5mm]
 \!\!\!&=&\!\!\! - R_{12}.
\end{eqnarray*}
\end{small}
\indent As the second example, we shall show that 
$\varphi_*([G_{04}, G_{45}]) = [\varphi_*(G_{04}),\varphi_*(G_{45})]$. 
Indeed,
   $\varphi_*([G_{04}, G_{45}]) = \varphi_*(G_{05}) = R_{05}$.
\vspace{1mm} 

\noindent On the other hand, 
\begin{small}
\begin{eqnarray*}
&& [\varphi_*(G_{04}), \varphi_*(G_{45})] = [R_{04}, R_{45}]
\\[0.5mm]
\!\!\!&=&\!\!\!  \Big[(0, (-(E_2 - E_3), 0, 0, 0), (0, -(E_2 - E_3), 0, 0), 0, 0, 0), \Big({\varPhi}\Big(i(E_1 \vee E_1),0, 0, \dfrac{i}{2}\Big), 0, 0, 
-\dfrac{i}{2}, 0, 0\Big)\Big]
\\[0.5mm]
\!\!\!&=&\!\!\! \Big(0, -\Big({\varPhi}(i(E_1 \vee E_1), 0, 0, \dfrac{i}{2}\Big)(-(E_2 - E_3), 0, 0, 0) + \dfrac{i}{2}(-(E_2 - E_3),
 0, 0, 0),
\\[0.5mm]
&& \qquad-\Big({\varPhi}(i(E_1 \vee E_1), 0, 0, \dfrac{i}{2}\Big)(0, -(E_2 - E_3), 0, 0) - \dfrac{i}{2}(0, -(E_2 -
 E_3), 0, 0), 0, 0, 0\Big) 
\\[0.5mm]
\!\!\!&=&\!\!\!(0, (-i(E_2 - E_3), 0, 0, 0), (0, i(E_2 - E_3), 0, 0), 0, 0, 0)
\\[0.5mm]
\!\!\!&=&\!\!\! R_{05}.
\end{eqnarray*}
\end{small}
\indent As the third example, we shall show that 
$\varphi_*([G_{57}, G_{67}]) = [\varphi_*(G_{57}), \varphi_*(G_{67})]$.
Indeed,
$\varphi_*([G_{57}, G_{67}]) = \varphi_*(-G_{56}) = -R_{56}$. 
\vspace{1mm}

\noindent On the other hand ,
\begin{small}
\begin{eqnarray*}
&& [\varphi_*(G_{57}), \varphi_*(G_{67})] = [R_{57}, R_{67}]
\\[0.5mm]
\!\!\!&=&\!\!\! \Big[\Big({\varPhi}\Big(0, \dfrac{1}{2}E_1, - \dfrac{1}{2}E_1, 0\Big), 0, 0, 0, \dfrac{1}{2}, -\dfrac{1}{2}\Big), \Big({\varPhi}\Big(-i(E_1 \vee E_1), 0, 0, -\dfrac{i}{2}\Big), 0,0, -\dfrac{i}{2}, 0, 0\Big)\Big] 
\end{eqnarray*}
\begin{eqnarray*}
\!\!\!&=&\!\!\!  \Big(\Big[{\varPhi}\Big(0, \dfrac{1}{2}E_1, -\dfrac{1}{2}E_1, 0\Big), {\varPhi}\Big(-i(E_1 \vee E_1), 0, 0, -\dfrac{i}{2}\Big)\Big], 0, 0, 0, -2\Big(\!-\dfrac{i}{2}\Big) 
\Big(\dfrac{1}{2}\Big), 2\Big(\!-\dfrac{i}{2}\Big)\Big(\!-\dfrac{1}{2}\Big)\Big)
\\[0.5mm]
\!\!\!&=&\!\!\! \Big({\varPhi}\Big(0, \dfrac{i}{2}E_1, \dfrac{i}{2}E_1, 0\Big), 0, 0, 0, \dfrac{i}{2}, \dfrac{i}{2}\Big) 
\\[0.5mm]
\!\!\!&=&\!\!\! -R_{56}.
\end{eqnarray*}
\end{small}
\indent As the forth example, we shall show that 
$\varphi_*([G_{36}, G_{68}]) = [\varphi_*(G_{36}), \varphi_*(G_{68})]$.
Indeed,
$\varphi_*([G_{36}, G_{68}]) = \varphi_*(G_{38}) = R_{38}$. 
\vspace{1mm}

\noindent On the other hand ,
\begin{small}
\begin{eqnarray*}
&& [\varphi_*(G_{36}), \varphi_*(G_{68})] = [R_{36}, R_{68}]
\\[0.5mm]
\!\!\!&=&\!\!\! [(0,(-E_1,0, -1, 0),(0,-E_1, 0,-1),0, 0, 0),(0,i F_1(1),0, 0),(-i F_1(1),0, 0,0),0, 0, 0)] 
\\[0.5mm]
\!\!\!&=&\!\!\! ((-E_1, 0, -1, 0) \times (-i F_1(1),0, 0,0))-(0,i F_1(1),0, 0) \times (0,-E_1, 0,-1), 0,0,0,0,0)
\\[0.5mm]
\!\!\!&=&\!\!\! \Big(\varPhi \Big(0,\dfrac{i}{4}F_1(1), -\dfrac{i}{4}F_1(1),0\Big)-\varPhi\Big(0,-\dfrac{i}{4}F_1(1), -\dfrac{i}{4}F_1(1),0 \Big), 0,0,0,0,0\Big)
\\[0.5mm]
\!\!\!&=&\!\!\! \Big (\varPhi\Big(0,\dfrac{i}{2}F_1(1), -\dfrac{i}{2}F_1(1),0\Big), 0,0,0,0,0\Big)
\\[1mm]
\!\!\!&=&\!\!\! R_{38}.
\end{eqnarray*}
\end{small}
\indent Finally, as the fifth example, we shall show that 
$\varphi_*([G_{28}, G_{89}]) = [\varphi_*(G_{28}), \varphi_*(G_{89})]$.
Indeed,
$\varphi_*([G_{28}, G_{89}]) = \varphi_*(G_{29}) = R_{29}$. 
\vspace{1mm}

\noindent On the other hand ,
\begin{small}
\begin{eqnarray*}
&& [\varphi_*(G_{28}), \varphi_*(G_{89})] = [R_{28}, R_{89}]
\\[0.5mm]
\!\!\!&=&\!\!\! \Big[\Big({\varPhi}\Big(0, -\dfrac{i}{2}F_1(1), -\dfrac{i}{2}F_1(1), 0\Big),0, 0, 0, 0, 0\Big), \Big({\varPhi}\Big(-[\tilde{A}_1(1), \tilde{A}_1(e_1)], 0, 0, 0\Big), 0, 0, 0, 0, 0\Big)\Big]
\\[0.5mm] 
\!\!\!&=&\!\!\! \Big(\Big[{\varPhi}\Big(0, -\dfrac{i}{2}F_1(1), -\dfrac{i}{2}F_1(1), 0\Big),{\varPhi}\Big(-[\tilde{A}_1(1), \tilde{A}_1(e_1)], 0, 0, 0\Big)\Big], 0, 0, 0, 0, 0\Big)
\\[0.5mm]
\!\!\!&=&\!\!\! \Big({\varPhi}\Big(0, [\tilde{A}_1(1), \tilde{A}_1(e_1)]\Big(-\dfrac{1}{2}F_1(1)\Big),  [\tilde{A}_1(1), \tilde{A}_1(e_1)]\Big(-\dfrac{1}{2}F_1(1)\Big), 0 \Big), 0, 0, 0, 0, 0\Big)
\\[0.5mm]
\!\!\!&=&\!\!\! \Big({\varPhi}\Big(0, -\dfrac{1}{2}F_1(e_1), -\dfrac{1}{2}F_1(e_1),0 \Big ), 0,0,0,0,0 \Big)
\\[0.5mm]
 \!\!\!&=&\!\!\!  R_{29}.
\end{eqnarray*}
\end{small}

Since we have $\dim(\mathfrak{so}(10))= 45 = \dim(({\mathfrak{e}_8})^{\sigma'_{\!\!4}, \mathfrak{so}(6)})$ from Lemma \ref{lem 6.1}, we see that $\varphi_*$ is an isomorphism. Thus we have the required isomorphism $({\mathfrak{e}_8})^{\sigma'_{\!\!4}, \mathfrak{so}(6)} \cong \mathfrak{so}(10)$ as a Lie algebra.
The latter case is the complexification of the former case.
\end{proof}
Now, we shall prove the theorem as the aim of this section.

\begin{thm}\label{thm 6.3}
We have that $({E_8}^C)^{\sigma'_{\!\!4}, \mathfrak{so}(6,C)} \cong S\!pin(10,C)$.
\end{thm}
\begin{proof}
The group $({E_8}^C)^{\sigma'_{\!\!4}, \mathfrak{so}(6,C)}$ is connected \!(Theorem \ref{thm 5.5}) and its type is $\mathfrak{so}(10, C)$ (Proposition \ref{prop 6.2}). Hence the group $({E_8}^C)^{\sigma'_{\!\!4}, \mathfrak{so}(6,C)}$ is isomorphic to either one of the following groups:
$$
      S\!pin(10, C), \quad S\!O(10, C), \quad S\!pin(10, C)/\Z_4. 
$$
Their centers of groups above are $\Z_4, \Z_2, \{1\}$, respectively. However, we see that the center of $({E_8}^C)^{\sigma'_{\!\!4}, \mathfrak{so}(6,C)}$ has the elements $1, \sigma, \sigma'_{\!\!4}, \sigma\sigma'_{\!\!4}$, and so its center is $\Z_4$. Hence the group $({E_8}^C)^{\sigma'_{\!\!4}, \mathfrak{so}(6,C)}$ have to be isomorphic to $S\!pin(10, C)$.   
\end{proof}
\section{The structure of the group $({E_8}^C)^{\sigma'_{\!\!4}}$}

By using the results of previous section,  the aim of this section is to 
determine the structure of the group $({E_8}^C)^{\sigma'_{\!\!4}}$. 

\begin{lem}\label{lem 7.1}
The Lie algebra $({\mathfrak{e}_8}^C)^{\sigma'_{\!\!4}}$ of the group $({E_8}^C)^{\sigma'_{\!\!4}}$ is given by 
\begin{eqnarray*}
({\mathfrak{e}_8}^C)^{\sigma'_{\!\!4}}
\!\!\!&=&\!\!\! \{ R \in {\mathfrak{e}_8}^C \,| \, 
    \sigma'_{\!\!4}R=R 
 \} 
\vspace{2mm}\\
\!\!\!&=&\!\!\! \left\{(\varPhi, P, Q, r, s, t) \in {\mathfrak{e}_8}^C \, \left| \,
\begin{array}{l}
     \varPhi \in ({\mathfrak{e}_7}^C)^{\sigma'_{\!\!4}}, \\
     P = (X, Y, \xi, \eta), \\
 \quad X=\begin{pmatrix} \xi_1 &    0      & 0 \\
                                              0   &\xi_2 & x \\
                                              0   &   \ov{x}& \xi_3
                \end{pmatrix},\,Y\!=\begin{pmatrix} \eta_1 &   0      &  0 \\
                                              0   & \eta_2 &  y \\
                                              0   &  \ov{y}& \eta_3
                \end{pmatrix}
\vspace{0.5mm}\\
\quad \xi_k, \eta_k, \xi, \eta \in C, x, y \in \C^C,  
\vspace{0.5mm}\\
Q=(Z, W, \zeta, \omega)\,\,{\text{is same form as }}P, \\
     r,s,t \in C
\end{array} \right. \right\}, 
\end{eqnarray*}
as for the explicit form of the Lie algebra $({\mathfrak{e}_7}^C)^{\sigma'_{\!\!4}}$, see Lemma 4.1 {\rm(}1{\rm )}.

In particular, 
$$
  \dim_C(({\mathfrak{e}_8}^C)^{\sigma'_{\!\!4}}) = 
33 + ((3+2) \times 2+1 \times 2)\times 2 + 3 = 60. 
$$
\end{lem}
\begin{proof}
By the argument similar to Lemma 5.1, we have the required result.
\end{proof}

Now, we shall prove the following theorem as the aim of this section.

\begin{thm}\label{thm 7.2}
We have that $({E_8}^C)^{\sigma'_{\!\!4}} \cong (S\!pin(6, C) \times S\!pin(10, C))/\Z_4, \,\,\,\Z_4=\{ (1,1), $ $(\sigma'_{\!\!4}, \sigma\sigma'_{\!\!4}), (\sigma, \sigma), (\sigma\sigma'_{\!\!4},\sigma'_{\!\!4}) \}$.
\end{thm}

\begin{proof} Let $S\!pin(6, C)\cong ({F_4}^C)_{E_1, E_2, E_3, F_1(e_k), k=0,1} \cong \vspace{0.5mm}(({E_7}^C)^{\kappa, \mu})_{\ti{E}_1,\ti{E}_{-1},E_2\dot{+}E_3, E_2\dot{-}E_3, \dot{F}_1(e_k), k=0,1} \subset ({E_7}^C)^{\sigma'_{\!\!4}} \subset ({E_8}^C)^{\sigma'_{\!\!4}}$ (Theorems \ref{thm 3.16}, \ref{thm 4.21}) and $S\!pin(10,C) \cong ({E_8}^C)^{\sigma'_{\!\!4}, \mathfrak{so}(6,C)} \subset ({E_8}^C)^{\sigma'_{\!\!4}}$ (Theorem \ref{thm 6.3}). Then we define  a mapping $\varphi_{{E_8}^C,\sigma'_{\!\!4}}: S\!pin(6, C) \times S\!pin(10, C) \to ({E_8}^C)^{\sigma'_{\!\!4}}$ by 
$$
     \varphi_{{E_8}^C,\sigma'_{\!\!4} }(\alpha, \beta) = \alpha\beta. 
$$
It is clear that $\varphi_{{E_8}^C,\sigma'_{\!\!4} }$ is well-defined. Since $[R_D, R_{10}] = 0$ for $R_D \in \mathfrak{spin}(6, C) = \mathfrak{so}(6, C) \cong  ({\mathfrak{f}_4}^C)_{E_1, E_2, E_3, F_1(e_k), k=0,1}, R_{10} \in \mathfrak{spin}(10, C) = \mathfrak{so}(10,C) \cong ({\mathfrak{e}_8}^C)^{\sigma'_{\!\!4}, \mathfrak{so}(6,C)}$ (Lemmas \ref{lem 3.14}, \ref{lem 6.1}) and $S\!pin(6, C), S\!pin(10, C)$ are connected, we see that $\alpha\beta = \beta\alpha$. Hence $\varphi_{{E_8}^C,\sigma'_{\!\!4} }$ is a homomorphism. Moreover, we obtain that $\Ker \, \varphi_{{E_8}^C,\sigma'_{\!\!4} } \cong \Z_4$. Indeed, since
we see that $\dim_C(\mathfrak{spin}(6, C) \oplus \mathfrak{spin}(10, C)) = 15 + 45 = 60 = \dim_C(({\mathfrak{e}_8})^{\sigma'_{\!\!4}})$ (Lemma \ref{lem 7.1}) and from $z({\varphi_{{E_8}^C,\sigma'_{\!\!4}}}_*)=\{0\}$(the mapping ${\varphi_{{E_8}^C,\sigma'_{\!\!4}}}_*$ is the differential mapping of $\varphi_{{E_8}^C,\sigma'_{\!\!4}}$) we have that $\Ker \, \varphi_{{E_8}^C,\sigma'_{\!\!4}}$ is discrete. Hence $\Ker \, \varphi_{{E_8}^C,\sigma'_{\!\!4}}$ is contained in the center $z(S\!pin(6, C) \times S\!pin(10, C)) = z(S\!pin(6, C)) \times z(S\!pin(10, C)) = \{1, \sigma, \sigma'_{\!\!4}, \sigma\sigma'_{\!\!4}\} \times \{1, \sigma, \sigma'_{\!\!4}, \sigma\sigma'_{\!\!4} \}$. Then, among them, the mapping $\varphi_{{E_8}^C,\sigma'_{\!\!4}}$ maps only $(1,1), (\sigma'_{\!\!4}, \sigma\sigma'_{\!\!4}), (\sigma, \sigma),$ $ (\sigma\sigma'_{\!\!4},\sigma'_{\!\!4})$ to the identity $1$. Hence we have that $\Ker \, \varphi_{{E_8}^C,\sigma'_{\!\!4} } \!\subset \{ (1,1), ((\sigma'_{\!\!4}, \sigma\sigma'_{\!\!4}), (\sigma, \sigma), (\sigma\sigma'_{\!\!4},$ $\sigma'_{\!\!4}) \}$, and vice versa. Thus we see that 
$$
\Ker \, \varphi_{{E_8}^C,\sigma'_{\!\!4} }=\{ (1,1), ((\sigma'_{\!\!4}, \sigma\sigma'_{\!\!4}), (\sigma, \sigma), (\sigma\sigma'_{\!\!4},\sigma'_{\!\!4}) \} \cong \Z_4.
$$
Since $({E_8}^C)^{\sigma'_{\!\!4}}$ is connected and $\Ker \, \varphi_{{E_8}^C,\sigma'_{\!\!4} }$ is discrete, again together with $\dim_C(\mathfrak{so}(6, C) \oplus \mathfrak{so}(10, C)) = 15 + 45 = 60 = \dim_C(({\mathfrak{e}_8})^{\sigma'_{\!\!4}})$, $\varphi_{{E_8}^C,\sigma'_{\!\!4} }$ is surjection.

Therefore we have the required isomorphism 
$$
({E_8}^C)^{\sigma'_{\!\!4}} \cong (S\!pin(6, C) \times S\!pin(10, C))/\Z_4.
$$
\end{proof}

\section{ Main theorem}

By using results above, we shall determine the structure of the group $(E_8)^{\sigma'_{\!\!4}}$, which is the main theorem.
\vspace{1mm}

\begin{thm}\label{thm 8.1}
We have that $({E_8})^{\sigma'_{\!\!4}} \cong (S\!pin(6) \times S\!pin(10))/\Z_4, \,\Z_4=\{ (1,1),(\sigma'_{\!\!4}, \sigma\sigma'_{\!\!4}), $ $ (\sigma, \sigma), (\sigma\sigma'_{\!\!4},\sigma'_{\!\!4}) \}$.
\end{thm}
\begin{proof}
For $\delta \in (E_8)^{\sigma'_{\!\!4}}\!=(({E_8}^C)^{\tau{\lambda_\omega}})^{\sigma'_{\!\!4}}\! =(({E_8}^C)^{\sigma'_{\!\!4}})^{\tau{\lambda_\omega}} \subset ({E_8}^C)^{\sigma'_{\!\!4}}$, there exist $\alpha \in S\!pin(6, C) \cong ({F_4}^C)_{E_1, E_2, E_3, F_1(e_k), k=0,1} \cong \vspace{1mm}(({E_7}^C)^{\kappa, \mu})_{\ti{E}_1,\ti{E}_{-1},E_2\dot{+}E_3, E_2\dot{-}E_3, \dot{F}_1(e_k), k=0,1} \subset ({E_7}^C)^{\sigma'_{\!\!4}} \subset ({E_8}^C)^{\sigma'_{\!\!4}}$ and $\beta \in S\!pin(10, C) \cong ({E_8}^C)^{\sigma'_{\!\!4},\mathfrak{so}(6,C)}\subset ({E_8}^C)^{\sigma'_{\!\!4}} $ such that $\delta = \varphi(\alpha, \beta) = \alpha\beta$ (Theorem \ref{thm 7.2}). From the condition 
$\tau{\lambda_\omega}\delta{\lambda_\omega}\tau = \delta$, that is, $\tau{\lambda_\omega}\varphi(\alpha, \beta){\lambda_\omega}\tau = \varphi(\alpha,\beta)$, we have $\varphi(\tau{\lambda_\omega}\alpha{\lambda_\omega}\tau, \tau\beta\tau) = \varphi(\alpha, \beta)$. Hence, we have that 
$$
\begin{array}{l}
     \mbox{(i)} \;\; \left\{\begin{array}{l}
 \tau\alpha\tau = \alpha
             \vspace{1mm}\\
 \tau{\lambda_\omega}\beta{\lambda_\omega}\tau = \beta,
             \end{array} \right.         
\qquad \qquad
\;      \mbox{(ii)} \;\; \left\{\begin{array}{l}
  \tau\alpha\tau = \sigma\sigma'_{\!\!4}\alpha
\vspace{1mm}\\
\tau{\lambda_\omega}\beta{\lambda_\omega}\tau = \sigma'_{\!\!4}\beta,
             \end{array} \right.       
\vspace{2mm}\\
\hspace*{-1mm}      \mbox{(iii)} \;\, \left\{\begin{array}{l}
\tau\alpha\tau = \sigma\alpha
               \vspace{1mm}\\
\tau{\lambda_\omega}\beta{\lambda_\omega}\tau = \sigma\beta,
\end{array} \right.
\qquad \quad \,\,
      \mbox{(iv)} \;\; \left\{\begin{array}{l}
\tau\alpha\tau = \sigma'_{\!\!4}\alpha
               \vspace{1mm}\\
\tau{\lambda_\omega}\beta{\lambda_\omega}\tau = \sigma\sigma'_{\!\!4}\beta.
\end{array} \right.        
\end{array} 
$$

Case (i). From the condition $\tau\alpha\tau = \alpha$, we have $\alpha \in S\!pin(6) \cong (F_4)_{ E_1, E_2, E_3, F_1(e_k), k=0,1}$. 
Indeed, first since $S\!pin(6, C) \cong ({F_4}^C)_{  E_1, E_2, E_3, F_1(e_k), k=0,1}$ is simply connected, the group  $(F_4)_{  E_1, E_2, E_3, F_1(e_k), k=0,1}=(({F_4}^C)_{ E_1, E_2, E_3, F_1(e_k), k=0,1})^\tau$ is connected.
Since $({F_4}^C)_{ E_1, E_2, E_3, F_1(e_k), k=0,1}$ acts on $(V^C)^6$, the group $(F_4)_{  E_1, E_2, E_3, F_1(e_k), k=0,1}=(({F_4}^C)_{ E_1, E_2, E_3, F_1(e_k), k=0,1})^\tau$ acts on 
\begin{eqnarray*}
  V^6\!\!\!&=&\!\!\!\{X \in (V^C)^6\,|\,
\tau X=X \}
\\[1mm]
\!\!\!&=&\!\!\!\{X =F_1(t)\,|\,t=t_2 e_2+t_3 e_3+t_4 e_4+t_5 e_5+t_6 e_6+t_7 e_7, t_k \in \R \}
\end{eqnarray*}
with the norm $(X, X)=2t\,\ov{t}$. We can define a homomorphism $\pi: (F_4)_{ E_1, E_2, E_3, F_1(e_k), k=0,1} \to S\!O(6)=S\!O(V^6)$ by $\pi(\alpha)=\alpha|_{{V^6}}$. Then it is easy to obtain that $\Ker\, \pi=\{1, \sigma \} \cong \Z_2$. Since, by doing simple computation as in Lemma \ref{lem 3.14} we see $\dim ((\mathfrak{f}_4)_{ E_1, E_2, E_3, F_1(e_k), k=0,1})=15$, we have that $\dim ((\mathfrak{f}_4)_{ E_1, E_2, E_3, F_1(e_k), k=0,1})$ $=15=\dim (\mathfrak{so}(6))$, moreover $S\!O(6)$ is connected. Hence the mapping $\pi$ is surjection. Thus we have that $(F_4)_{ E_1, E_2, E_3, F_1(e_k), k=0,1}/\Z_2 \cong S\!O(6)$. Therefore the group $(F_4)_{ E_1, E_2, E_3, F_1(e_k), k=0,1}$ is isomorphic to $S\!pin(6)$ as the universal covering group $S\!O(6)$, that is, $(F_4)_{ E_1, E_2, E_3, F_1(e_k), k=0,1} \cong S\!pin(6)$. 

Next, from the condition $\tau{\lambda_\omega}\beta{\lambda_\omega}\tau = \beta$, we see that the group $\{\alpha \in S\!pin(10, C) \, |\, \tau{\lambda_\omega}\beta{\lambda_\omega}\tau $ $= \beta\} = (S\!pin(10, C))^{\tau{\lambda_\omega}}$ (which is connected) $= (({E_8}^C)^{\sigma'_{\!\!4},\mathfrak{so}(6,C)})^{\tau{\lambda_\omega}}= ({E_8})^{\sigma'_{\!\!4},\mathfrak{so}(6)}$ (Theorem \ref{thm 6.3}), and so its type is $\mathfrak{so}(10)$ (Proposition \ref{prop 6.2}) (Note that $({E_8}^C)^{\tau\lambda_\omega}=E_8$ and the $C$-linear transformation $\tau\lambda_\omega$ induces the involutive automorphism of the group $({E_8}^C)^{\sigma'_{\!\!4},\mathfrak{so}(6,C)}$. Hence we see that the group $({E_8})^{\sigma'_{\!\!4},\mathfrak{so}(6)}$ is isomorphic to either one of 
$$
S\!pin(10), \quad S\!O(10), \quad S\!pin(10)/\Z_4 .
$$
Their center are $\Z_4, \Z_2$, \{1\}, respectively. However, since the center of $({E_8})^{\sigma'_{\!\!4},\mathfrak{so}(6)}$ has the elements
$1, \sigma\sigma'_{\!\!4}, \sigma, \sigma'_{\!\!4}$, the group $({E_8})^{\sigma'_{\!\!4},\mathfrak{so}(6)}$ has to be isomorphic to  $S\!pin(10)$ and its center is $\{1, \sigma\sigma'_{\!\!4}, \sigma, \sigma'_{\!\!4} \} $ $ \cong \Z_4$. Hence the group of Case (i) is isomorphic to the group $(S\!pin(6) \times S\!pin(10))/\Z_4$.

Case (ii). This case is impossible. Indeed, for $\alpha \in S\!pin(6,C) \subset S\!pin(8,C)$ we can set $\alpha=(\alpha_1, \alpha_2, \alpha_3), \alpha_1 \in S\!O(6,C) \subset S\!O(8,C), \alpha_2, \alpha_3 \in S\!O(8,C)$ satisfying $(\alpha_1 x)(\alpha_2 y)=\ov{\alpha_3(\ov{xy})}, x, y, \in \mathfrak{C}^C$, and similarly for $\sigma'_{\!\!4} \in S\!pin(8)$, set $\sigma'_{\!\!4}=(\sigma'_1, \sigma'_2, \sigma'_3), \sigma_k \in S\!O(8) \subset S\!O(8, C)$ satisfying $(\sigma'_1 x)(\sigma'_2 y)=\ov{\sigma'_3(\ov{xy})}, x, y, \in \mathfrak{C}$. 
(Remark. As a matrix, $\sigma'_1, \sigma'_2$ and $\sigma'_3$ are expressed as follows:  
\begin{eqnarray*}
\sigma'_1 \!\!\!&=&\!\!\! \diag (1,1,-1,-1,-1,-1,-1,-1),
\\
\sigma'_2 \!\!\!&=&\!\!\! \diag (-J, -J, -J, -J), 
                                          J=\begin{pmatrix}0  & 1 \\
                                                           -1 & 0
                                            \end{pmatrix},
\\
\sigma'_3 \!\!\!&=&\!\!\! \diag (J, -J, -J, -J).)
\end{eqnarray*}
Then, from the condition $\tau\alpha\tau=\sigma\sigma'_{\!\!4}\alpha$, we have $(\tau\alpha_1, \tau\alpha_2, \tau\alpha_3)=(\sigma'_1\alpha_1, -\sigma'_2\alpha_2, -\sigma'_3\alpha_3)$, 
that is, 
$$
\tau\alpha_1=\sigma'_1\alpha_1, \quad \tau\alpha_2=-\sigma'_2\alpha_2, \quad \tau\alpha_3=-\sigma'_3\alpha_3.
$$
Here, as matrix, $\alpha_1$ is expressed as follows:
$$
    \alpha_1=\arraycolsep5pt
\left(
\begin{array}{@{\,}cc|cccccc@{\,}}
1 & 0 &  & & && \\
0 & 1 & \multicolumn{5}{c}{\raisebox{1pt}[0pt][0pt]{\huge $0$}}& \\
\hline
  &&&&&&&\\
\multicolumn{2}{c|}{\raisebox{-10pt}[0pt][0pt]{\huge $0$}}
 & \multicolumn{5}{c}{\raisebox{-10pt}[0pt][0pt]{\Huge $A$}}&\\
  &&&&&&&\\
  &&&&&&&\\
\end{array}
\right), \,\,A \in S\!O(6,C),
$$
where all blanks mean 0. Then, from $\tau\alpha_1=\sigma'_1\alpha_1$, we have 
$$
  \alpha_1=\arraycolsep5pt
\left(
\begin{array}{@{\,}cc|cccccc@{\,}}
1 & 0 &  & & && &\\
0 & 1 &\multicolumn{5}{c}{\raisebox{1pt}[0pt][0pt]{\huge $0$}}& \\
\hline
  &&&&&&&\\
\multicolumn{2}{c|}{\raisebox{-10pt}[0pt][0pt]{\huge $0$}}&\multicolumn{5}{c}{\raisebox{-10pt}[0pt][0pt]{\Huge $A'$}}&\\
  &&&&&&&\\
  &&&&&&&\\
\end{array}
\right), \,\,A'=iB \in S\!O(6,C), \,i^2=-1. 
$$
As for $\alpha_2$, from $\tau\alpha_2=-\sigma'_2\alpha_2$, we have $\alpha_2=0$. Indeed, from the explicit form of $\sigma_2$, it is sufficient to confirm this in the case $2 \times 2$-matrix. From 
$$
\begin{pmatrix} a & b \\
                c & d 
\end{pmatrix} \begin{pmatrix}0  & -1 \\
                             1 & 0
              \end{pmatrix}=\begin{pmatrix} b & -a \\
                                            d & -c
                            \end{pmatrix}=\begin{pmatrix} \tau a & \tau b \\
                       \tau c & \tau d
       \end{pmatrix},
$$
we have that $\tau a=b, \tau b=-a, \tau c=d, \tau d=-c$, that is, $a=b=c=d=0$. Hence we see $\alpha_2=0$.
This is contrary to $\alpha_2 \in S\!O(8, C)$.
\vspace{2mm}

Case (iii). This case is also impossible. Indeed, from the condition $\tau\alpha\tau=\sigma\alpha$, we have $(\tau\alpha_1, \tau\alpha_2, \tau\alpha_3)=(\alpha_1, -\alpha_2, -\alpha_3)$, that is, 
$$
\tau\alpha_1=\alpha_1, \quad \tau\alpha_2=-\alpha_2, \quad \tau\alpha_3=-\alpha_3.
$$
From $\tau\alpha_1=\alpha_1$, we have $\alpha_1 \in S\!O(6) \subset S\!O(8)$. Hence, by the Principal of triality on $S\!O(8)$ we see that $\alpha_k \in S\!O(8), k=2,3$, that is, $\tau\alpha_k=\alpha_k, k=2,3$. However, from $\tau\alpha_k=-\alpha_k, k=2,3$, we have $\alpha_k=\tau\alpha_k=-\alpha_k$, that is, $\alpha_k=0$. This is contrary to $\alpha_k \in S\!O(8)$. 
\vspace{2mm}

Case (iv). This case is also impossible. Indeed, from the condition $\tau\alpha\tau=\sigma'_{\!\!4}\alpha$, we have $(\tau\alpha_1, \tau\alpha_2, \tau\alpha_3)=(\sigma'_1\alpha_1, \sigma'_2\alpha_2, \sigma'_3\alpha_3)$,  
that is, 
$$
\tau\alpha_1=\sigma'_1\alpha_1, \quad \tau\alpha_2=\sigma'_2\alpha_2, \quad \tau\alpha_3=\sigma'_3\alpha_3.
$$
As in the Case (ii), we have $\alpha_2 =0$. This is contrary to 
$\alpha_2 \in S\!O(8, C)$.
\vspace{2mm}

Therefore we have the required isomorphism 
$$
({E_8})^{\sigma'_{\!\!4}} \cong (S\!pin(6) \times S\!pin(10))/\Z_4, 
\,\Z_4=\{ (1,1),(\sigma'_{\!\!4}, \sigma\sigma'_{\!\!4}), (\sigma, \sigma), (\sigma\sigma'_{\!\!4},\sigma'_{\!\!4}) \}.
$$
\end{proof}
\vspace{2mm}

{\bf \Large Acknowledgments.}\,\,
\vspace{2mm}

The author would like to thank 
the members of {\it YSTM}: Professor Ichiro Yokota, Osamu Shukuzawa, Kentaro Takeuchi, Takashi Miyasaka, Tetsuo Ishihara and Tomohiko Tsukada for their valuable comments,  suggestions and encouragement.

\vspace{5mm}

\begin{flushright}

\begin{tabular}{c}
{\it Toshikazu Miyashita} \\
{\it Ueda-Higashi High School} \\
{\it Ueda City}\\
{\it  Nagano\,\, 386-8683}\\
{\it Japan} \\
{\it E-mail}: {\it anarchybin@gmail.com} 
\end{tabular}

\end{flushright}

\begin{thebibliography}{99}
\bibitem{Imai} T. Imai and I. Yokota,
\newblock Simply connected compact simple Lie group $E_{8(-248)}$ of type $E_8$, J. Math. Kyoto Univ. 21(1981), 741-762. 

\bibitem{kuri} H. Kurihara and K. Tojo, 
\newblock Involutions of compact Riemannian $4$-symmetric spaces, Osaka J. Math. 45(2008), 645-689.

\bibitem{gray} A. Gray, 
\newblock Riemannian manifolds with geodesic symmetries of order 3, J. Differential Geometry 7 (1972), 343-369.

\bibitem{Jim} J.A. Jim\'{e}nez, 
\newblock Riemannian 4-symmetric spaces, Trans. Amer. Math. Soc. 306 (1988), 715-734.

\bibitem{Miyashita} T. Miyashita and I. Yokota, 
\newblock An explicit isomorphism between $(\mathfrak{e}_8)^{\sigma,\sigma',\mathfrak{so}(8)}$ 
and $\mathfrak{so}(8) \oplus \mathfrak{so}(8)$ as Lie algebras, Math. J. Toyama Univ. 24(2001), 
151-158. 

\bibitem{miya} T. Miyashita,
\newblock Realizations of globally exceptional ${\bf \rm Z}_2 \times {\bf \rm Z}_2$-symmetric spaces, Tsukuba J. Math. 38-2(2014), 239-311.

\bibitem{Yokota} I. Yokota, T. Imai and O. Yasukura, 
\newblock On the homogeneous space $E_8/E_7$, J. Math. Kyoto Univ. 23-3(1983), 467-473.

\bibitem{realization G_2} I. Yokota, 
\newblock Realizations of involutive automorphisms $\sigma$ and $G^\sigma$ of 
exceptional linear Lie groups $G$, Part I, $G = G_2, F_4$, and $E_6$, Tsukuba J. Math. 14(1990), 185-223.

\bibitem{realization E_7} I. Yokota, 
\newblock Realizations of involutive automorphisms $\sigma$ and $G^\sigma$ of 
exceptional linear Lie groups $G$, Part II, $G = E_7$, Tsukuba J. Math. 14(1990), 378-404.

\bibitem{realization E_8} I. Yokota, 
\newblock Realizations of involutive automorphisms $\sigma$ and $G^\sigma$ of 
exceptional linear Lie groups $G$, Part III, $G = E_8$, Tsukuba J. Math. 15(1991), 301-314.

\bibitem{Yokotaichiro} I. Yokota, 
\newblock Exceptional simple Lie groups, arXiv:math/0902.0431vl[mathDG](2009). 
\end{thebibliography}
\end{document}